\documentclass[english,12pt]{article}
\usepackage[T1]{fontenc}
\usepackage[utf8]{inputenc}
\usepackage{lmodern}
\usepackage[a4paper]{geometry}
\usepackage{babel}
\usepackage{amssymb}
\usepackage{amsmath}
\usepackage{amsthm}
\usepackage{esint}
\usepackage{mathrsfs}
\usepackage{graphicx}

\graphicspath{{./images/}}
\usepackage{pifont}
\usepackage{amsfonts}
\usepackage{comment}
\usepackage{tikz}
\newcommand{\Dd}{\mathcal D^{L, (b)}}
\newcommand{\Dc}{\mathbb D^{(b)}}
\newcommand{\Cd}{\mathcal C^{L}}
\newcommand{\Cc}{\mathbb C}
\newcommand{\FDd}{\partial \Dd}
\newcommand{\FDc}{\partial \Dc}
\newcommand{\FCdb}{\partial_1\Cd}
\newcommand{\FCda}{\partial_0\Cd}
\newcommand{\FCcb}{\partial_1\Cc}
\newcommand{\FCca}{\partial_0\Cc}
\newcommand{\Cdd}{\mathcal C_{\delta}^{L}}
\newcommand{\Ccd}{\mathbb C_{\delta}}
\newcommand{\tDd}{\tilde{\mathcal D}^{L, (a)}}
\newcommand{\tDc}{\tilde{\mathbb D}^{(a)}}
\def\bx{\mathbf{x}}
\def\bm{\mathbf{m}}
\def\br{\mathbf{r}}
\def\N{\mathbb{N}}
\def\z{\mathcal{Z}}
\def\D{\mathbb{D}}
\def\C{\mathbb{C}}
\def\P{\mathbb{P}}
\def\E{\mathbb{E}}
\def\R{\mathbb{R}}
\def\wh{\widehat}

\def\ve{\varepsilon}
\def\dd{\mathrm{d}}

\def\lbracket {[\hspace{-.10em} [ }
\def\rrbracket{ ] \hspace{-.10em}]}

\usepackage{esint}
\usepackage{dsfont}
\usepackage{subcaption}
\usepackage{wasysym}
\usepackage{bbm}
\usepackage[all]{xy}
\usepackage{hyperref}

\theoremstyle{theorem}
\newtheorem{theorem}{Theorem}[section]
\newtheorem{lemma}[theorem]{Lemma}
\newtheorem{definition}[theorem]{Definition}
\newtheorem{proposition}[theorem]{Proposition}
\newtheorem{corollaire}[theorem]{Corollary}

\title{Drilling holes in the Brownian disk: \\The Brownian annulus}
\author{Jean-Fran\c cois Le Gall\footnote{jean-francois.le-gall@universite-paris-saclay.fr} \ and Alexis Metz--Donnadieu\footnote{alexis.metz-donnadieu@ens.fr}}

\begin{document}
	\maketitle
	\begin{abstract}
		We give a new construction of the Brownian annulus based on removing a hull
		centered at the distinguished point in the free Brownian disk. We use this construction 
		to prove that the Brownian annulus is the scaling limit of
		Boltzmann triangulations with two boundaries. We also prove that the space
		obtained by removing hulls centered at the two distinguished points of 
		the Brownian sphere is a Brownian annulus. Our proofs rely on a detailed
		analysis of the peeling by layers algorithm for Boltzmann  triangulations with a boundary.
	\end{abstract}
	
\tableofcontents
	
\section{Introduction}
	
	Brownian surfaces are basic models of random geometry that have been the subject of intensive research in the recent years. They arise as scaling limits of large classes of random planar random maps 
	viewed as random metric spaces, for the Gromov-Hausdorff topology. The first result in this direction was the convergence to the Brownian sphere 
	\cite{Le_Gall_2013,miermont2011brownian}, which is a Brownian surface in genus $0$ with no boundary. This convergence has been extended to many
	different classes of random planar maps by several authors.
	The recent paper of Bettinelli et Miermont \cite{BrownianSurfacesII} constructs general Brownian surfaces 
	in arbirary genus $g$ and with a finite number of boundaries of given sizes, as the scaling limit of large random quadrangulations with boundaries (boundaries of quadrangulations are distinguished 
	faces with arbitrary degrees, whereas the other faces have degree $4$). The construction of \cite{BrownianSurfacesII}  applies to the case where the volume of the surface is fixed as well as the 
	boundary sizes, and it is also of interest to consider ``free'' models where this volume is not fixed, which appear as scaling limits of planar maps distributed according to
	Boltzmann weights. The special case where there is only one boundary in genus $0$ corresponds to the so-called Brownian disk, which has been studied extensively (see in
	particular \cite{Bet0,BrownianDiskBettineli,GM0,albenque2020scaling,BrowDiskandtheBrowSnake}). 
	
Our object of interest in this work is the free Brownian annulus, which is a Brownian surface in genus $0$ with two boundaries. As noted in \cite[Section 1.5]{BrownianSurfacesII}, the free Brownian annulus 
is one of the very few Brownian surfaces (together with the Brownian disk and the pointed Brownian disk) for which the free model makes sense under a probability measure --- for instance,
the free Brownian sphere is defined under an infinite measure. One motivation for the present work came from the recent paper of Ang, R\'emy and Sun \cite{ang2022moduli}, which studies the
modulus of Brownian annuli in random conformal geometry. The definition of the Brownian annulus in \cite[Definition 1.1]{ang2022moduli} is based on considering the complement of a hull 
in the free pointed Brownian disk conditionally on the event that the hull boundary has a fixed size. As the authors of \cite{ang2022moduli} observe, this definition leads certain technical difficulties due to 
conditioning on an event of probability zero. In this work, we give a slightly different construction of the Brownian annulus
which involves only conditioning on an event of positive probability. We show that this definition is equivalent to the one in \cite{ang2022moduli}, and 
we also relate our construction to the scaling limit approach of Bettinelli et Miermont \cite{BrownianSurfacesII} by showing that the Brownian annulus is the scaling
limit of large random triangulations with two boundaries --- this was asserted without proof in \cite{ang2022moduli}. 

Let us give a more precise description of our main results. We start from a free pointed Brownian disk $(\mathbb{D},D)$ with boundary size $a>0$. As usual, $\partial \mathbb{D}$
denotes the boundary of $\mathbb{D}$. Then, $\mathbb{D}$ has a 
distinguished interior point denoted by $x_*$. For every $r\in (0,D(x_*,\partial \mathbb{D}))$,
we denote the closed ball of radius $r$ centered at $x_*$ by $B_r(x_*)$, and the hull $H_r$ is obtained by
``filling in the holes'' of $B_r(x_*)$. In more precise terms,  $\mathbb{D}\setminus H_r$
is the connected component of $\mathbb{D}\setminus B_r(x_*)$ that contains the boundary $\partial\mathbb{D}$. The perimeter or boundary size of $H_r$ may then be defined
as 
\begin{equation}\label{peri-def}
\mathcal{P}_r=\lim_{\ve\to 0} \ve^{-2} \mathbf{V}(\{x\in\mathbb{D}\setminus H_r:D(x,H_r)<\ve\}),
\end{equation}
where $\mathbf{V}$ is the volume mesure of $\mathbb{D}$.
The process $(\mathcal{P}_r)_{0<r<D(x_*,\partial \mathbb{D})}$ has a modification with c\`adl\`ag sample paths
and no positive jumps, and, for every
$b>0$, we set 
$$r_b=\inf\{r\in (0,D(x_*,\partial \mathbb{D})):\mathcal{P}_r=b\},$$
where $\inf\varnothing=\infty$.  Then $\P(r_b<\infty)=a/(a+b)$ (Lemma \ref{proba-condit}), and, on the event $\{r_b<\infty\}$, $r_b$ is the radius of the first hull 
of boundary size $b$. Under the conditional probability $\P(\cdot\mid r_b<\infty)$, we define the Brownian annulus of boundary sizes $a$ and $b$, denoted by $\mathbb{C}_{(a,b)}$, as the 
closure of $\mathbb{D}\setminus H_{r_b}$, which is equipped with the continuous extension $d^\circ$ of the intrinsic metric on $\mathbb{D}\setminus H_{r_b}$ and with the restriction of the volume
measure of $\mathbb{D}$ (Theorem \ref{Br-Annulus}). It is convenient to view $\mathbb{C}_{(a,b)}$ as a measure metric space marked with two compact subsets (the ``boundaries'') which are here $\partial\mathbb{D}$ and
$\partial H_{r_b}$. 

Much of the present work is devoted to proving that the space $\mathbb{C}_{(a,b)}$ is the Gromov-Hausdorff limit of rescaled triangulations with two boundaries. More precisely,
for every sufficiently large  integer $L$, let $\mathcal{C}^L$ be a random planar triangulation
with two simple boundaries of respective sizes $\lfloor a L\rfloor$ and $\lfloor b L\rfloor$ (see \cite{bernardiFusy} for precise definitions of triangulations with boundaries). Assume that
$\mathcal{C}^L$ is distributed according to Boltzmann weights, meaning that the probability of a given triangulation $\tau$ is proportional to $(12\sqrt{3})^{-k(\tau)}$ where $k(\tau)$ is the number of internal vertices of $\tau$. 
We equip the vertex set $V(\mathcal{C}^L)$ with the graph distance rescaled by the factor $\sqrt{3/2} \,L^{-1/2}$, which we denote by $d^\circ_L$. Then, Theorem \ref{PropPrinc} states that
\begin{equation}
\label{conv-triangu}
(V(\mathcal{C}^L),d^\circ_L) \overset{(d)}{\underset{L\to \infty}{\longrightarrow}} (\mathbb{C}_{(a,b)},d^\circ),
\end{equation}
in distribution in the Gromov-Hausdorff sense. Theorem \ref{main-extend} gives a stronger version of this convergence by considering the Gromov-Hausdorff-Prokhorov 
distance on measure metric spaces marked with two boundaries, in the spirit of \cite{BrownianSurfacesII} (see Section \ref{sec:conv} below). 

The proof of the convergence \eqref{conv-triangu} relies on two main ingredients. The first one is a result of Albenque, Holden and Sun \cite{albenque2020scaling} 
showing that the free Brownian disk is the scaling limit of Boltzmann triangulations with a simple boundary, when the boundary size tends to $\infty$. The second ingredient
is the peeling by layers algorithm for Boltzmann triangulations with a simple boundary, which was already investigated in the recent paper \cite{MarkovSpatial} in view of studying 
the spatial Markov property of Brownian disks. Roughly speaking, a peeling algorithm ``explores'' a Boltzmann triangulation $\mathcal{D}^L$
with boundary size $\lfloor a L\rfloor$ step by step, starting from a distinguished 
interior vertex, and, in the special case of the peeling by layers, the explored region at every step is close to a (discrete) hull centered at the distinguished vertex. At the first time when the boundary size of the explored region becomes equal to $\lfloor b L\rfloor$ (conditionally on the event 
that this time exists), the unexplored region is a Boltzmann triangulation  with two simple boundaries of sizes $\lfloor a L\rfloor$ and $\lfloor b L\rfloor$, and is therefore
distributed as $\mathcal{C}^L$. One can then use the main result of \cite{albenque2020scaling}  giving the scaling limit of $\mathcal{D}^L$ to derive 
the convergence \eqref{conv-triangu}. Making this argument precise requires a number of preliminary results, and in particular a detailed study of 
asymptotics for the peeling process of Boltzmann triangulations with a boundary, which is of independent interest (see Section \ref{sec:peeling-asymp} below).
These asymptotics are closely related to the similar results for the peeling process of the UIPT obtained in \cite{ScalingUIPT}.

As a by-product of our construction, we obtain several other results relating the Brownian annulus to the Brownian disk or the Brownian sphere. Consider again the free pointed Brownian disk $(\mathbb{D},D)$
of perimeter $a$, but now fix $r>0$. Conditionally on $\{D(x_*,\partial\mathbb{D})>r, \mathcal{P}_r=b\}$ the closure of  $\mathbb{D}\setminus H_{r}$ equipped with an (extended) intrinsic metric
has the same distribution as $\mathbb{C}_{(a,b)}$ (Proposition \ref{key-result}) and furthermore is independent of the hull $H_r$
also viewed a random metric space for the appropriate intrinsic metric. This result in fact corresponds to the definition of the Brownian annulus in  \cite[Definition 1.1]{ang2022moduli}. Another related
result involves removing two disjoint hulls in the free Brownian sphere. Write $(\bm_\infty, \mathbf{D})$ for the free Brownian sphere, which has two distinguished points denoted by $\bx_*$
and $\bx_0$  that play symmetric roles. For every $r>0$ and $x\in\bm_\infty$,  let $B^\infty_r(x)$ be the closed ball of radius $r$ centered at $x$ in $\bm_\infty$. 
Then, for $r\in(0,\mathbf{D}(\bx_*,\bx_0))$, 
let the hull $B^\bullet_r(\bx_*)$ be the complement of the connected component of $\bm_\infty\setminus B^\infty_r(\bx_*)$ that contains $\bx_0$,
and define $B^\bullet_r(\bx_0)$ by interchanging the roles of $\bx_*$ and $\bx_0$. Let $r,r'>0$. Then, conditionally on the event $\{\mathbf{D}(\bx_*,\bx_0)>r+r'\}$, the three spaces
$B^\bullet_r(\bx_*)$, $B^\bullet_{r'}(\bx_0)$ and $\bm_\infty\setminus (B^\bullet_r(\bx_*)\cup B^\bullet_{r'}(\bx_0))$ are independent conditionally on the perimeters $|\partial B^\bullet_{r}(\bx_*)|$ and $|\partial B^\bullet_{r'}(\bx_0)|$
(these perimeters are defined by a formula analogous to \eqref{peri-def}),
and $\bm_\infty\setminus (B^\bullet_r(\bx_*)\cup B^\bullet_{r'}(\bx_0))$ is a Brownian annulus with boundary sizes $|\partial B^\bullet_{r}(\bx_*)|$ and $|\partial B^\bullet_{r'}(\bx_0)|$
(see Theorem \ref{3-pieces}  below for a more precise statement, and \cite[Lemma 6.3]{ang2022moduli} for a closely related result). 

In addition to our main results, we obtain certain explicit formulas, which are of independent interest. In particular, Proposition \ref{law-perimeter-hull} gives the distribution of $\mathcal{P}_r$
under $\P(\cdot \cap\{ r<D(x_*,\partial\mathbb{D})\})$ (note that the distribution of $D(x_*,\partial\mathbb{D})$ was computed in \cite[Proposition 14]{spine}). We  also 
consider the ``length'' $\mathcal{L}_{(a,b)}$ of $\mathbb{C}_{(a,b)}$, which is the minimal distance between the two boundaries. By combining our 
definition of $\mathbb{C}_{(a,b)}$ with the Bettinelli-Miermont construction of the Brownian disk, one gets that $\mathcal{L}_{(a,b)}$
is distributed as the last passage time at level $b$ for a continuous-state branching process with branching mechanism $\psi(\lambda):=\sqrt{8/3}\,\lambda^{3/2}$
started with initial density $\frac{3}{2}\,a^{3/2}(a+z)^{-5/2}$, and conditioned to hit $b$ (this conditioning event has probability $a/(a+b)$). Unfortunately, we have not
been able to use this description to derive an explicit formula for the distribution of $\mathcal{L}_{(a,b)}$, but Proposition \ref{first-mom} gives a remarkably simple 
formula for its first moment:
$$\E[\mathcal{L}_{(a,b)}]=\sqrt{\frac{3\pi}{2}}(a+b)\left(\sqrt{ a^{-1}}+\sqrt{b^{-1}}-\sqrt{a^{-1}+b^{-1}}\right).$$

The paper is organized as follows. Section \ref{sec:preli} gathers a number of preliminaries, concerning 
in particular the peeling algorithm for random triangulations, the Bettinelli-Miermont construction of the
free pointed Brownian disk, and a useful embedding of the Brownian disk in the Brownian sphere.  
Then, Section \ref{sec:annu-def} presents our construction of the Brownian annulus, and also 
proves a technical lemma that will be used in the proof of the convergence of rescaled triangulations to
the Brownian annulus. In Section \ref{sec:preli-conv}, we recall the key convergence
of rescaled triangulations with a boundary to the Brownian disk, and we use this result to
investigate the convergence of certain explored regions in the peeling algorithm
of Boltzmann triangulations
towards hulls in the Brownian disk. Section \ref{sec:peeling-asymp} is devoted to asymptotics 
for the perimeter process in the peeling by layers algorithm
of Boltzmann triangulations: the ultimate goal of these asymptotics is to verify that the
(suitably rescaled) first radius at which the perimeter of the explored region hits the value $\lfloor bL\rfloor$
converges to $r_b$, and that this convergence takes place jointly with the convergence 
to the Brownian disk (Corollary \ref{ConvBrownianDiskWithRadius}). Section \ref{sec:conv-annulus}
gives the proof of the scaling limit \eqref{conv-triangu}. If $\mathcal{C}^L$ is constructed 
via the peeling algorithm as explained above, a  technical difficulty comes from the fact that
it is not easy to control distances near the boundary of the unexplored region, and, to overcome this problem,
we use approximating spaces obtained by removing a tubular neighborhood of the latter boundary. 
In Section \ref{sec:conv-bound-vol}, we explain how the convergence \eqref{conv-triangu} can be sharpened to hold in the sense of
the Gromov-Hausdorff-Prokhorov 
topology on measure metric spaces marked with two boundaries. 
Section \ref{sec:compl2hulls} explains the relation between our construction of the Brownian annulus and 
the definition of \cite{ang2022moduli}, and also proves Theorem \ref{3-pieces}  showing that
the complement of the union of two hulls centered at the distinguished points of the Brownian sphere is a
Brownian annulus.
Finally, Section
\ref{sec:expli-comp} discusses the distribution of the length $\mathcal{L}_{(a,b)}$ of the Brownian annulus.

\section{Preliminaries}
	\label{sec:preli}
In this section, we recall the basic definitions and the theoretical framework that we will use in this paper. Section \ref{sec:peel} introduces Boltzmann triangulations as well as the peeling 
by layers algorithm, which plays an important role in this work. In Section \ref{sec:conv}, we recall the definition of the Gromov-Hausdorff-Prohorov topology for measure metric spaces, using the formalism of \cite{BrownianSurfacesII}. Section \ref{sec:cons} gives a construction of the free Brownian disk, 
which is the compact metric space arising  as the scaling limit of Boltzmann triangulations with a boundary.

	\subsection{Boltzmann triangulations of the disk and the annulus}
	\label{sec:peel}
	
	For two integers $L\geq 1$ and $k\geq 0$, we let $\mathbb T^{1}(L, k)$ be the set of all pairs $(\tau, e)$, where $\tau$ is a type I planar triangulation with a simple boundary $\partial \tau$ of length $L$ and $k$ internal vertices, and where $e$ is a distinguished edge on $\partial \tau$. Here, type I means that we allow the presence of multiple edges and loops, but the boundary has to remain simple. Each edge $e$ of $\partial\tau$ is oriented so that the outer face lies to the left of $e$ (see Figure 1), and we write $|\partial\tau|=L$ for the boundary size of $\tau$. By convention, we will consider the map consisting of a single oriented (simple) edge $e$ as an element of $\mathbb T^{1}(2, 0)$ and in that special case it is convenient
	to consider that $\partial\tau$ consists of two oriented edges, namely $e$ and $e$ with the reverse orientation.
	
	For integers $L,p\geq 1$ and $k\geq 0$, we let $\mathbb T^2(L, p,k)$ be the set of triplets $(\tau, e_0, e_1)$, where $\tau$ is a planar triangulation of type $I$ having two vertex-disjoint simple boundaries --- namely an outer boundary $\partial_0\tau$ of length $L$ and an inner boundary $\partial_1\tau$ of length $ p$ --- and $k$ internal vertices, and where $e_0$ and $e_1$ are distinguished edges on $\partial_0\tau$ and $\partial_1\tau$ respectively. The edges on the boundaries are again oriented so that the boundary faces lie on their left. See \cite{bernardiFusy} for more precise definitions. We have the following explicit enumeration formulas (cf. \cite{bernardiFusy}):
	\begin{align} 
		\label{EnumT1} &\forall (L, k)\neq (1, 0), \ \ \ \text{Card} \ \mathbb T^{1}(L, k)=4^{k-1}\frac{(2L+3k-5)!!}{k!(2L+k-1)!!}L{2L\choose L},\\
		\label{EnumT2} 	&\forall L, p\geq 1, k\geq 0, \ \ \text{Card} \ \mathbb T^2(L, p, k)=\frac{4^k(2(L+p)+3k-2)!!}{k!(2(L+p)+k)!!}L{2L\choose L}p{2p\choose p},
	\end{align}
with the convention $(-1)!!=1$. Note that, in the case $(L,k)=(2, 0)$, formula (\ref{EnumT1}) remains valid thanks to the  previous convention making the map composed of a single edge an element of $\mathbb T^1(2,0)$. In the following, we are interested in triangulations for which the number of internal vertices is random, and we set $\mathbb T^{1}(L)=\bigcup_{k\geq 0}\mathbb T^{1}(L, k)$ and $\mathbb T^2(L, p)= \bigcup_{k\geq 0}\mathbb T^2(L, p, k)$. A random triangulation $\mathcal T$  in $\mathbb T^1(L)$ (resp. in $\mathbb T^{2}(L, p)$) is said to be \emph{Boltzmann distributed} if, for every $k\geq 0$ and every $\theta\in \mathbb T^{1}(L, k)$ (resp. $\theta\in \mathbb T^2(L, p , k)$), the probability that $\mathcal T=\theta$ is proportional to $(12\sqrt 3)^{-k}$. More precisely, asymptotics of (\ref{EnumT1}) and (\ref{EnumT2}) show that the quantities
	$$Z^{1}(L):=\sum_{k\geq 0} (12\sqrt 3)^{-k} \text{Card} \ \mathbb T^{1}(L, k), $$
and
$$Z^2(L, p):=\sum_{k\geq 0} (12\sqrt 3)^{-k} \text{Card} \ \mathbb T^2(L,p,k),
$$
are finite. The Boltzmann measure on $\mathbb T^{1}(L)$ gives probability $Z^{1}(L)^{-1}(12\sqrt 3)^{-k}$ to any triangulation $ \theta\in \mathbb T^{1}(L, k)$, where $k\geq 0$. Similarly,
the Boltzmann measure on $\mathbb T^2(L,p)$ gives probability $Z^2( L, p)^{-1}(12\sqrt 3)^{-k}$ to any triangulation $\theta\in \mathbb T^2(L,p, k)$. By \cite{PercOnRandMapsI}, Section 2.2, we have the explicit expression:
\begin{align*}
	\forall L\geq 1, \ \ Z^{1}(L)=\frac{6^L(2L-5)!!}{8\sqrt 3 L!},
\end{align*}
where again $(-1)!!=1$. In the following, it will also be useful to define $Z^1(0):=(24\sqrt 3)^{-1}$.

\begin{figure}[h]
		\centering
		\begin{subfigure}{0.3\textheight}
			\includegraphics[width=0.7\textwidth]{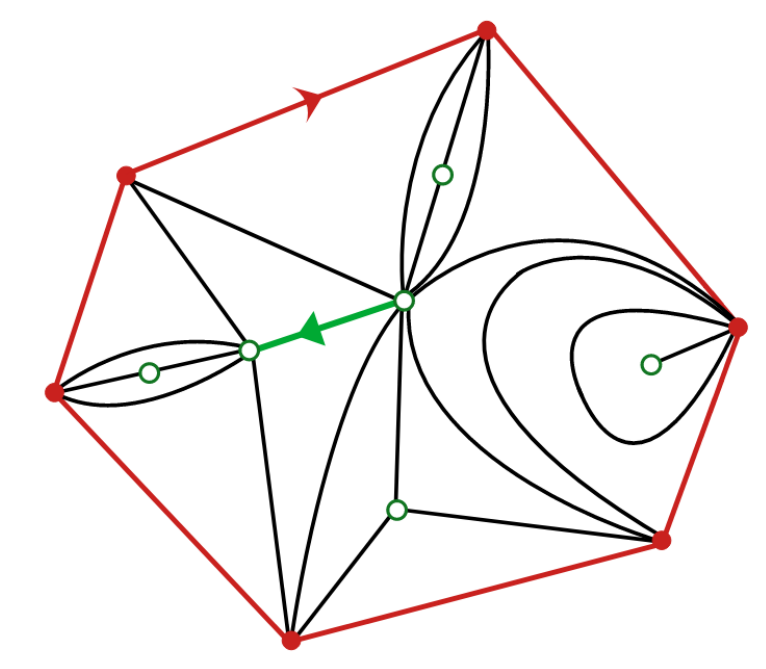}
		\end{subfigure}
		\begin{subfigure}{0.3\textheight}
			\includegraphics[width=0.7\textwidth]{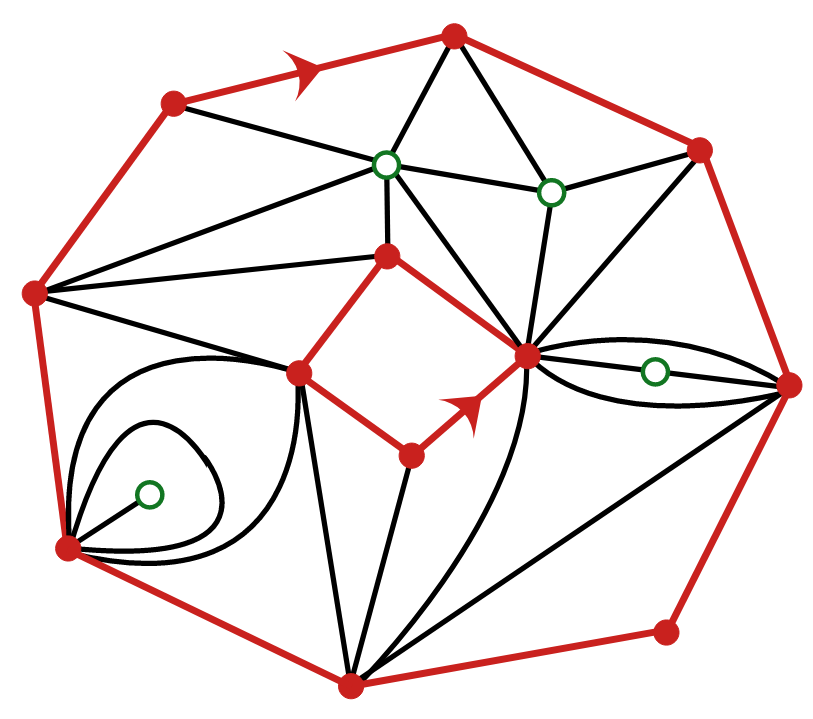}
		\end{subfigure}
		\caption{An element of $\mathbb T^{1, \bullet}(6,5)$ (left) and an element of $\mathbb T^2(8, 4, 4)$ (right), the distinguished edges are marked by an arrow.}
	\end{figure}

	Finally, we let $\mathbb T^{1, \bullet}(L, k)$ be the set of all triangulations in $\mathbb T^{1}(L, k)$ that have (in addition to the distinguished edge on the boundary) another distinguished oriented edge chosen among all edges of the triangulation. This
	second distinguished edge may or may not be part of the boundary, but we will call it the distinguished interior edge with some abuse of terminology. The Boltzmann measure on $\mathbb T^{1,\bullet}(L)=\bigcup_{k\geq 0}\mathbb T^{1,\bullet}(L, k)$ is again the probability measure that gives probability proportional to $(12\sqrt 3)^{-k}$ to any $\tau\in \mathbb T^{1,\bullet}(L, k)$. This makes sense because a triangulation $\tau\in \mathbb T^{1}(L, k)$ has $3k+2L -3$ edges, by Euler's formula, so that the number of ways of choosing an \emph{oriented} edge in $\tau$ is $6k+4L-6$ and we have:
	\begin{align*}
		Z^{1, \bullet}(L):= \sum_{k\geq 0}(6k+4L-6)(12\sqrt 3)^{-k}\text{Card} \ \mathbb T^1( L, k)<\infty,
	\end{align*}
	since $\text{Card} \ \mathbb T^1(L, k)= O((12\sqrt 3)^kk^{-5/2})$ when $k\to\infty$. Note that $\text{Card} \ \mathbb T^{1, \bullet}(2, 0)=2$.
	
	\paragraph{Peeling and the discrete spatial Markov property}
	We now recall the main properties of the so-called \emph{peeling algorithm}. We refer to \cite{peeling} for a more detailed introduction to this algorithm. In 
	the following, it will be convenient to add an isolated point $\dagger$ to the different state spaces that we will consider. The point $\dagger$ will play the role of a cemetery point when the exploration given by the peeling algorithm hits the boundary.
	
	Fix $p\geq 1$, $\gamma\in \mathbb T^2(L, p)$ and let $e$ be an edge of $\partial_1\gamma$ (this edge will be called the peeled edge). 
	Let $u$ be the vertex opposite $e$ in the unique internal face $f$ of $\gamma$ incident to $e$. Three configurations may occur:
	\begin{enumerate}
		\item $u$ is an internal vertex of $\gamma$, in this case we call peeling of $\gamma$ along the edge $e$ the sub-triangulation of $\gamma$ consisting of the internal faces of $ \gamma$ distinct from $f$. We see this triangulation as an element of $\mathbb T^2(L, p+1)$.
		\item $u$ is an element of the inner boundary $\partial_1\gamma$. In this case $f$ splits $\gamma$ into two components, only one of which is incident to the outer boundary $\partial_0\gamma$. We call peeling of $\gamma$ along the edge $e$ the sub-triangulation consisting of the faces of this component, that we see as an element of $\mathbb T^2(L, p')$ for some $1 \leq p'\leq p$.
		\item Finally, if $u$ belongs to the outer boundary of $\gamma$, we say by convention that the ``triangulation'' obtained by peeling $\gamma$ along $e$ is $\dagger$.
	\end{enumerate}
Note that this description is slightly incomplete since  it would be necessary to specify (in the first two cases) how the new distinguished edge on the inner boundary is chosen. In what follows,
we will iterate the peeling algorithm, and it will be sufficient to say that this new distinguished edge is chosen at every step as a deterministic function of the rooted planar map that
is made of the initial inner boundary and of the faces that have been ``removed'' by the peeling algorithm up to this step.
\\

	\noindent Let us fix an algorithm $\mathcal A$ that chooses for any triangulation $\tau\in \bigcup_{p\geq 1}\mathbb T^{1, \bullet}(p)$ an edge $e$ of $\partial\tau$. The peeling of a triangulation according to the algorithm $\mathcal A$ consists in recursively applying the peeling procedure described above, choosing the peeled edge at each step as prescribed by $\mathcal A$. Let us give a more precise description. 
	
	We start with a triangulation $\gamma\in \mathbb T^{1,\bullet}(L)$ and we let $e_0$ be its distinguished interior edge. If $e_0$ is incident to the boundary $\partial \gamma$ of $\gamma$, we set by convention $\gamma_0=\tau_0=\dagger$. Otherwise, if $e_0$ is a loop, we let $\tau_0$ be the triangulation induced by the faces of $\gamma$ inside the loop and we let $\gamma_0$
	be the triangulation that consists of the faces of $\gamma$ outside this loop. We view $\tau_0$ as an element of $\mathbb T^{1, \bullet}(1, k)$ for some $k\geq 0$ (we let both distinguished edges to be the loop $e_0$ oriented clockwise) and we view $\gamma_0$ as an element of $\mathbb T^2(L, 1)$ by seeing the loop as 
	bounding an internal face of degree one. Finally, if $e_0$ is a simple edge (not incident to $\partial \gamma$), we let $\tau_0$ be the unique element of $\mathbb T^{1, \bullet}(2, 0)$ with both distinguished edges oriented in the same direction, and $\gamma_0$ is the element of $\mathbb T^2 (L, 2)$ obtained from $\gamma$ by splitting the edge $e_0$ so as to create an inner boundary face of degree $2$ (cf. figure 2) -- note that our special convention for $\partial\tau_0$ explained at the beginning of
	Section \ref{sec:peel} allows us to identify $\partial_1\gamma_0$ with $\partial\tau_0$ in that case.
	
	 We then build recursively two sequences $(\tau_i)_{i\geq 0}$ (the \emph{explored} part) and $(\gamma_i)_{i\geq 0}$ (the \emph{unexplored} part), in such a way that,
	 for every $i\geq 0$ such that $\tau_i\neq \dagger$, we have $\tau_i\in \mathbb T^{1, \bullet}(p)$ 
	 and $\gamma_i\in \mathbb T^{ 2}(L, p)$, for some $p\geq 1$, and the inner boundary $\partial_1\gamma_i$ is identified with $\partial\tau_i$. Assume that we have constructed $\tau_i$ and 
	 $\gamma_i$ for some $i\geq 0$.
	  If $\tau_i=\dagger$,
	 we set $\tau_{i+1}=\gamma_{i+1}=\dagger$.   
	Otherwise the algorithm $\mathcal A$ applied to $\tau_i$ yields an edge $e$ of $\partial\tau_i=\partial_1\gamma_i$. The triangulation $\gamma_{i+1}$ is obtained by peeling $\gamma_i$ along this edge. If $\gamma_{i+1}\neq \dagger$, we let $\tau_{i+1}$ be the triangulation obtained by adding to $\tau_i$ the faces of $\gamma_i$ that we removed by the peeling of $e$. 
	 The distinguished edge on the boundary of $\tau_{i+1}$ is the one that is identified to the distinguished edge of $\gamma_{i+1}$ on its second boundary, and the other distinguished edge of 
	 $\tau_{i+1}$ is taken to be the same as the one of $\tau_i$.
	 Finally, if $\gamma_{i+1}=\dagger$, we simply take $\tau_{i+1}=\dagger$. \\

	\begin{figure}[h]
		\centering
		\begin{subfigure}{0.31 \textwidth}
			\includegraphics[width=0.8 \textwidth]{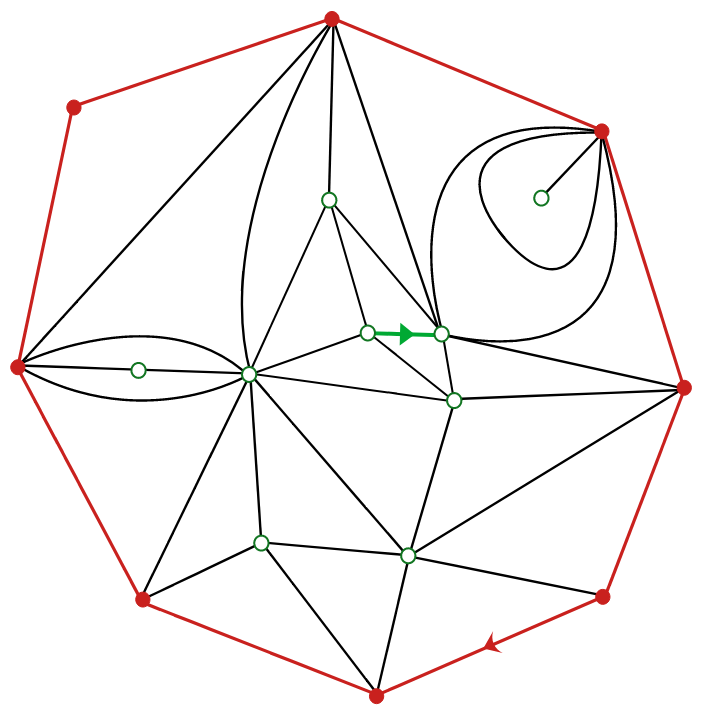}
		\end{subfigure}
		\begin{subfigure}{0.31 \textwidth}
			\includegraphics[width=0.8 \textwidth]{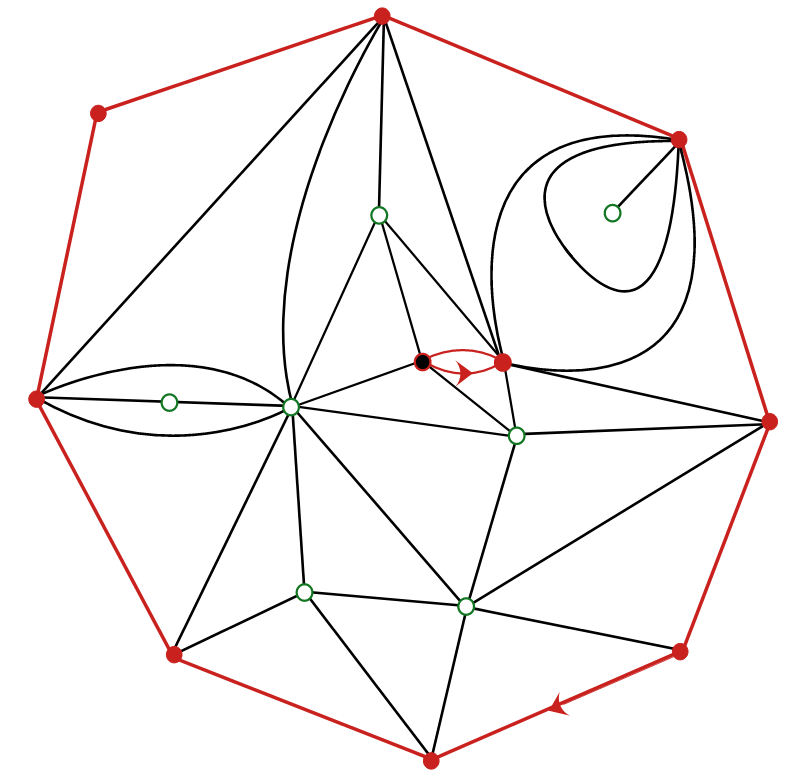}
		\end{subfigure}
		\begin{subfigure}{0.31 \textwidth}
			\includegraphics[width=0.8\textwidth]{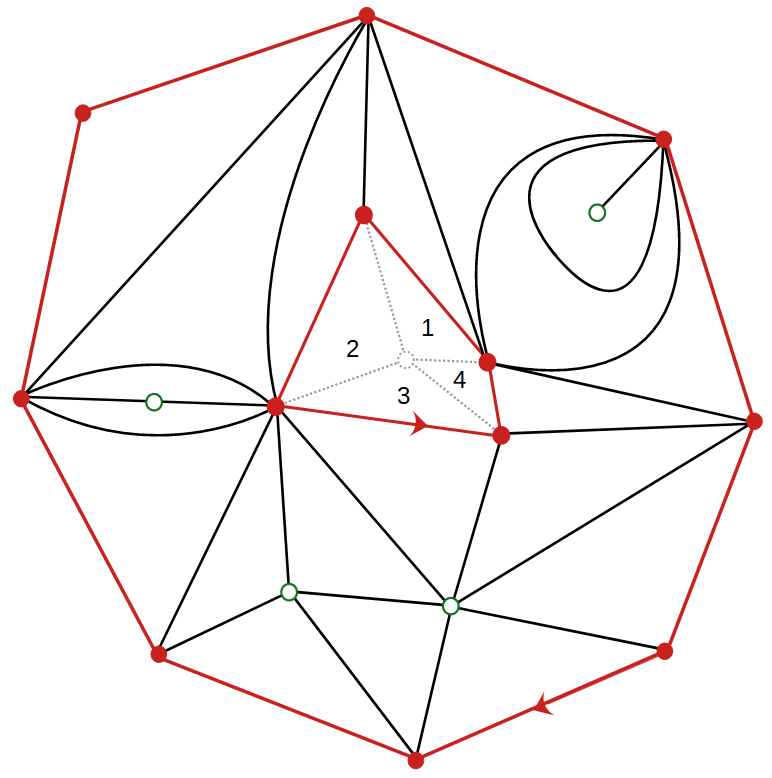}
		\end{subfigure}
		\caption{Starting from a Boltzmann triangulation (left), the interior distinguished edge is split so as to create an inner face of degree $2$ (middle). On the right, the triangulation we get after 4 peeling steps, the explored part is in dotted lines.}
	\end{figure}
	
In the case of Boltzmann triangulations, the peeling is a ``Markovian exploration''. More precisely, we apply the peeling procedure described above to a random triangulation $\mathcal D^L$ distributed according to the Boltzmann measure on $\mathbb T^{1, \bullet}(L)$. This gives rise to two sequences of random triangulations $(T_i^L)_{i\geq 0}$ (explored parts) and $(U_i^L)_{i\geq 0}$ (unexplored parts). Then, conditionally on the event $\{T_i^L\neq\dagger\}$ and on the value $|\partial T_i^L|$, the triangulation $U_i^L$ is distributed according to the Boltzmann measure on $\mathbb T^{2}(L, |\partial T_i^L|)$ and is independent of $T_i^L$. We will call this property the \emph{spatial Markov property for the peeling of Boltzmann triangulations}.
	
	\paragraph{Peeling by layers and perimeter process} Let  $x_*^L$ be the root of the distinguished interior edge of $\mathcal D^L$ and let $\Delta^L$ be the graph distance in $\mathcal D^L$. In the following, we will use a particular peeling algorithm 
	--- that is, a particular choice of $\mathcal{A}$ --- which we call the \emph{peeling by layers}. This algorithm is designed to satisfy the following additional property: 
	for every $i$ such that $T_i^L\neq \dagger$, if we set $h_i^L:=\Delta^L(x_*^L, \partial T_i^L)$, then for every vertex $u$ of $ \partial T^L_i$, we have
	\begin{align*}
		h_i^L\leq \Delta^L(u, x_*^L)\leq h_i^L+1.
	\end{align*}
 In other words, the distances from boundary vertices of $T_i^L$ to $x_*^L$ in $\mathcal D^L$ can only take at most one of two consecutive values at any time. It is easy to choose the 
 peeling algorithm so that this property holds, and we will assume that $(T_i^L)_{i\geq 0}$ and $(U_i^L)_{i\geq 0}$ are obtained by such a peeling algorithm. 
 We refer to \cite{ScalingUIPT} for a more precise description of the peeling by layers algorithm.
	
	An important object for us is the random sequence $(|\partial T_i^L|)_{i\geq 0}$ taking values in $\mathbb N\cup \{\dagger\}$ and recording the evolution of the perimeter of the part explored by the peeling algorithm by layers, where by convention $|\partial T_i^L|=\dagger$ if $T_i^L=\dagger$. By the arguments of \cite{MarkovSpatial}, Section $3$, conditionally on the value of $|T_0^L|\in \{1, 2, \dagger\}$, 
	this perimeter process is a Markov chain on $\mathbb{N}\cup\{\dagger\}$ starting from $|T_0^L| \in\{1,2,\dagger\}$ whose transition kernel $q_L$ is given for every $k\geq 1$ and $m\in \{-1, 0,  \ldots , k-1\}$ by:
	\begin{align*}
		&q_L(k, k-m)=2Z^{1}(m+1)\frac{Z^2(L, k-m)}{Z^2(L, k)},
	\end{align*}
	and $q_L(k, \dagger)=1-\sum_{m=-1}^{k-1} q_L(k, k-m)$ for all $k\geq 1$, $q_L(\dagger, \dagger)=1$. This kernel is closely related to the transition kernel $q_\infty$ of the perimeter process of the UIPT of type $I$ (cf \cite{ScalingUIPT}, Section 6.1) 
	which is defined for every $k\geq 1$ and $m\in \{-1,0, \ldots, k-1\}$ by:
	\begin{align*}
		q_\infty(k, k-m)=
		2Z^{1}(m+1)\frac{C^{(1)}(k-m)}{C^{(1)}(k)}.
	\end{align*}
	where we wrote $C^{(1)}(k):=\frac{3^{k-2}}{4\sqrt {2\pi}} k{2k\choose k}$. As noted in \cite{MarkovSpatial}, the Markov chain associated with the kernel $q_L$ is a Doob $h$-transform of the chain associated to $q_\infty$, for the harmonic function $\mathbf{h}_L(j):=\frac{L}{L+j}, j\geq 1$. More precisely, for every $p\geq 1$, $m\in \{-1,0, \ldots, p-1\}$:
	\begin{align}
		\label{DefqL}
		q_L(p, p-m)=\frac{\mathbf{h}_L(p-m)}{\mathbf{h}_L(p)}q_\infty(p, p-m).
	\end{align}

	\subsection{Convergence of metric spaces}
	\label{sec:conv}

	In order to state the convergence of (rescaled) Boltzmann triangulations with two boundaries towards the Brownian annulus,
we will consider the space $\mathbb{M}$ of all isometry classes of compact metric spaces, and we will write 
$d_{\mathtt{GH}}$ for the usual Gromov-Hausdorff distance on $\mathbb{M}$. Then $(\mathbb{M}, d_{\mathtt{GH}})$ is a Polish space.

	We will use analogs of the Gromov-Hausdorff distance for spaces marked with subspaces and measures, which we present along the lines of 
	 \cite[Section 1.3]{BrownianSurfacesII}. 
	Here and in what follows, if $(E,\Delta)$ is a compact metric space $E$, we will write $\Delta_{\mathtt H}$ and $\Delta_{\mathtt P}$ for the Hausdorff and Prohorov metrics associated with $\Delta$, which are defined respectively on the set of all nonempty compact subsets of $E$ and on the set of all finite Borel measures on $E$. 
	
	For $l\in \mathbb N$, we let $\mathbb M^{l, 1}$ be the set of all isomorphism classes (for an obvious notion of isomorphism) of compact metric spaces marked with $l$ compact subspaces and a finite measure. More precisely, we consider marked spaces of the form $((\mathcal X, d_{\mathcal X}), \mathbf A, \boldsymbol \mu)$ where:
	\begin{itemize}
		\item[\textbullet] $(\mathcal X, d_{\mathcal X})$ is a compact metric space,
		\item[\textbullet] $\mathbf A=(\mathbf A_1,\ldots,\mathbf A_l)$ is an $l$-tuple of compact subsets of $\mathcal X$,
		\item[\textbullet] $\boldsymbol \mu$ is a finite Borel measure on $\mathcal X$.
	\end{itemize} 
	The set $\mathbb M^{l,1}$ is endowed with a metric $d^{l,1}_{\mathtt {GHP}}$, which is defined for any two spaces $\mathbb X=((\mathcal X, d_{\mathcal X}), \mathbf A, \boldsymbol \mu)$ and $\mathbb Y=((\mathcal Y, d_{\mathcal Y}), \mathbf B, \boldsymbol \rho)$ in $\mathbb M^{l,1}$ by:
	\begin{align*}
		d^{l,1}_{\mathtt{GHP}}(\mathbb X, \mathbb Y)=\inf_{\substack{(\mathcal Z, \Delta)\\ \iota_X: \mathcal X\hookrightarrow \mathcal Z\\ \iota_Y: \mathcal Y\hookrightarrow \mathcal Z}} \max\left\{\Delta_{\mathtt H}\left(\iota_{\mathcal X} (\mathcal X), \iota_{\mathcal Y}(\mathcal Y)\right), \max_{1\leq i\leq l}\Delta_{\mathtt H}\left(\iota_{\mathcal X}(\mathbf A_i), \iota_{\mathcal Y}(\mathbf B_i)\right), \Delta_{\mathtt P}\left(\iota_X{}_* \boldsymbol \mu, \iota_Y{}_*\boldsymbol \rho\right)\right\},
	\end{align*}
where the infimum is taken over all compact metric spaces $(\mathcal Z, \Delta)$ and isometric embeddings $\iota_{\mathcal X}:(\mathcal X, d_{\mathcal X})\to (\mathcal Z, \Delta)$ and $\iota_{\mathcal Y}:(\mathcal Y, d_{\mathcal Y})\to (\mathcal Z, \Delta)$.
Then $d_{\mathtt {GHP}}^{l,1}$ is a metric on $\mathbb M^{l, 1}$. Furthermore, $(\mathbb M^{l,1},d^{l,1}_{\mathtt {GHP}})$ is a Polish space. 
In what follows, we will be interested in the case $l=2$: the Brownian annulus comes with a volume measure and with two distinguished 
subsets which are its boundaries.

 
	\subsection{The Bettinelli-Miermont construction of the Brownian disk}
	\label{sec:cons}
	This section presents a variant of the Bettinelli-Miermont construction of the free Brownian disk, which is based on a quotient space defined from a Poisson family of Brownian trees. We borrow the formalism of \cite{MarkovSpatial}.
	\paragraph{The Brownian snake}
	We start with a brief presentation of the Brownian snake, referring to \cite{CSBPRandomSnakes} for more details.
 Let $\mathcal W$ be the set of continuous paths $w:[0, \zeta(w)]\to \mathbb R$, where $\zeta(w)\geq 0$ is a nonnegative real number called the lifetime of $w$. We endow this set with the distance:
	\begin{align*}
		d_{\mathcal W}(w, w')=|\zeta(w)-\zeta(w')|+\sup_{t\geq 0} |w(t\wedge \zeta(w))-w(t\wedge \zeta(w'))|.
	\end{align*}
For every $x\in \mathbb R$, let $\mathcal W_x$ be the set of all $w\in\mathcal W$ such that $w(0)=x$. We identify the unique element of $\mathcal W_x$ having lifetime $0$ to the real number $x$. A snake trajectory starting at $x$ is a continuous function $\omega:\mathbb R_+\to \mathcal W_x$ satisfying:
	\begin{itemize}
		\item[\textbullet] $\omega_0=x$ and $\sigma(\omega):=\sup\{s\geq 0, \ \omega_s\neq x\}<\infty$; 
		\item[\textbullet] for all $0\leq s\leq s'$, $\omega_s(t)=\omega_{s'}(t)$ whenever $t\leq \min_{u\in [s, s']}\zeta(\omega_u)$.
	\end{itemize}
Let $\mathcal S_x$ be the set of snake trajectories starting from $x$, that we endow with the distance:
	\begin{align*}
		d_{\mathcal S_x}(\omega, \omega')=|\sigma(\omega)-\sigma(\omega')|+\sup_{s\geq 0}d_{\mathcal W}(\omega_s,\omega'_s).
	\end{align*}
	If $\omega\in \mathcal S_x$, we let $\zeta_\omega:\mathbb R_+\to \mathbb R_+$ be the function defined by setting $\zeta_\omega(s):=\zeta(\omega_s)$ 
	and we also write $\hat \omega$ for the function called the \emph{head} of the snake trajectory $\omega$ defined by $\hat \omega(s):= \omega_s(\zeta_\omega(s))$. One easily verifies that $\omega$ is entirely 
	determined by the two functions $\zeta_\omega$ and $\hat \omega$. We will also use the notation
	$$W_*(\omega)=\min\{\hat \omega_s:s\in [0,\sigma(\omega)]\}.$$
	
Given a snake trajectory $\omega$, we can define a (labelled compact) $\mathbb R$ tree $T_\omega$, which is called the \emph{genealogical tree of $\omega$}. To construct this tree, we introduce the pseudo-distance $d_\omega$ on $[0, \sigma(\omega)]$ given by:
	\begin{align*}
		\forall s, t\in [0, \sigma(\omega)], \;d_\omega(s, t)= \zeta_\omega(s)+\zeta_\omega(t)-2\min_{u\in [s, t]}\zeta_\omega(u),
	\end{align*}
	and we define $T_\omega$ as the quotient space of $[0, \sigma(\omega)]$ for the equivalence relation $s\sim t$ iff $d_\omega(s, t) =0$, which is equipped with the metric induced by $d_\omega$. 
	We let $p_{\omega}:[0, \sigma(\omega)]\to T_\omega$ be the canonical projection and we write $\rho_\omega:=p_\omega(0)$ for the ``root'' of $T_\omega$. 
	The volume measure on $T_\omega$ is just the pushforward of Lebesgue measure on $[0,\sigma(\omega)]$
	under the projection $p_\omega$.
	By the definition of snake trajectories, the property $p_\omega(s)=p_\omega(t)$ implies that $\hat\omega(s)=\hat\omega(t)$. Thus we can define a natural labelling $\ell_\omega:T_\omega\to \mathbb R$ by
	requiring  that $\hat \omega =\ell_\omega\circ p_\omega$.
	\begin{definition} Let $x\in\R$.
		The Brownian snake excursion measure with initial point $x$ is the $\sigma$-finite measure $\mathbb N_x$ on $\mathcal S_x$ such that the pushforward of $\mathbb N_x$ under the function $\omega\mapsto \zeta_\omega$ is the It\^o measure of positive Brownian excursions, normalized so that $
			\mathbb N_x\left(\sup_{s\geq 0}\zeta_\omega(s)\geq \varepsilon\right)=\frac{1}{2\varepsilon}$,
		and such that, under $\mathbb N_x$ and conditionally on $\zeta_\omega$, the process $(\hat\omega_s)_{s\geq 0}$ is a Gaussian process centered at $x$ with covariance kernel $K(s, s ')=\min_{u\in[s, s']}\zeta_\omega(u)$
		when $s\leq s'$.
	\end{definition}
	
\noindent We will use some properties of exit measures of the Brownian snake. If $w\in \mathcal W$, and $y\in\mathbb R$, we write $\tau_y(w)=\inf\{t\leq \zeta(w):\ w(t)=y\}$ with the convention $\inf\varnothing =+\infty$. 
If $x\in \mathbb R$ and $y\in(-\infty,x)$, the quantity:
\begin{align}
	\label{MesureSortie}
	\mathcal Z_y(\omega):=\lim_{\varepsilon\to 0}\frac{1}{\varepsilon^2}\int_0^{\sigma(\omega)}\mathbf{1}_{\{\tau_y(\omega_s)= \infty, \hat \omega(s)<y+\varepsilon\}}\mathrm ds,
\end{align}
exists  $\mathbb N_x(\mathrm d \omega)$ almost everywhere and is called the exit measure at $y$. The process $(\mathcal Z_y(\omega))_{y\in(-\infty,x)}$ has a c\`adl\`ag modification with no positive jumps, which 
we consider from now on.

\paragraph{The free Brownian sphere}

Let us now recall the construction of the free Brownian sphere under the measure $\mathbb{N}_0(d\omega)$. We start by recalling the definition of 
``intervals'' on the genealogical tree $T_\omega$ of a snake trajectory $\omega$. We use the convention that, if $s,t\in[0,\sigma(\omega)]$ and
$s>t$, the interval $[s,t]$ is defined by $[s,t]=[s,\sigma(\omega)]\cup[0,t]$. Then, if $u,v\in T_\omega$, there is a smallest
interval $[s,t]$, with $s,t\in[0,\sigma(\omega)]$, such that $p_\omega(s)=u$ and $p_\omega(t)=v$, and we define
$$\lbracket  u,v\rrbracket=p_\omega([s,t]).$$
We set, for every $u,v\in T_\omega$,
$$\mathbf{D}^\circ(u, v):=\ell_\omega(u)+\ell_\omega(v)-2\max\left(\min_{w\in \lbracket u, v\rrbracket}\ell_\omega(w), \min_{w\in \lbracket v, u\rrbracket}\ell_\omega(w)\right),$$
and 
$$\mathbf{D}(u, v):= \inf_{u=u_0,u_1, \ldots, u_p=v}\sum_{j=1}^p\mathbf{D}^\circ(u_i, u_{i+1}),
$$
where  the infimum is taken over all choices of the integer $p\geq 1$ and the points $u_0,\ldots, u_p\in T_\omega$
such that $u_0=u$ and $u_p=v$.
Then, $\mathbf{D}$ is a pseudo-metric on $T_\omega$, and the free Brownian sphere is the associated quotient space 
$\bm_\infty=T_\omega /\{\mathbf{D}=0\}$, which is equipped with the metric induced by $\mathbf{D}$, for which we keep the same notation.
We note that the free Brownian sphere is a geodesic space (any two points are linked by at least one geodesic). 

We emphasize that the free Brownian sphere is defined under the {\it infinite} measure $\N_0$, but later
we will consider specific conditionings of $\N_0$
giving rise to finite measures.
We write $\mathbf{\Pi}$ for the canonical projection from $T_\omega$ onto $\bm_\infty$. 
The volume measure $\mathrm{Vol}(\cdot)$ on $\bm_\infty$ is the pushforward of the volume measure on $T_\omega$ under $\mathbf{\Pi}$.

For $u,v\in T_\omega$, the property $\mathbf{D}(u,v)=0$ implies $\ell_\omega(u)=\ell_\omega(v)$, and so we can define $\ell(x)$ for 
every $x\in \bm_\infty$, in such a way that $\ell(x)=\ell_\omega(u)$ whenever $x=\mathbf{\Pi}(u)$. There is a unique point $\bx_*$ of $\bm_\infty$ such that 
$$\ell(\bx_*)=\min_{x\in\bm_\infty} \ell(x),$$
and we have $\mathbf{D}(\bx_*,x)= \ell(x)-\ell(\bx_*)$ for every $x\in\bm_\infty$. We will write $\br_*:=-\ell(\bx_*)$. We also observe that the free Brownian sphere has another
distinguished point, namely $\bx_0:=\mathbf{\Pi}(\rho_\omega)$. Note that $\mathbf{D}(\bx_*,\bx_0)=-\ell(\bx_*)=\br_*$

Let us now turn to hulls. For every $r>0$ and $x\in\bm_\infty$,  we write $B^\infty_r(x)$ for the closed ball of radius $r$ centered at $x$ in $\bm_\infty$. Then, for every
$r\in(0,\br_*)$, the hull $B^\bullet_r(\bx_*)$ is the complement of the connected component of $\bm_\infty\setminus B^\infty_r(\bx_*)$
that contains $\bx_0$ (this makes sense because $\bx_0\notin B^\infty_r(\bx_*)$ when $r<\br_*$). Note that all points of $\partial B^\bullet_r(\bx_*)$ are
at distance $r$ from $\bx_*$. By definition, the perimeter of the hull $B^\bullet_r(\bx_*)$
is the exit measure $\mathbf{P}_r:=\mathcal{Z}_{r-r_*}$. This definition is justified by the property
\begin{equation}
\label{approx-exit}
\mathbf{P}_r=\lim_{\varepsilon\to 0} \frac{1}{\varepsilon^2} \mathrm{Vol}(\{x\in \bm_\infty\setminus B^\bullet_r(\bx_*):\mathbf{D}(x,B^\bullet_r(\bx_*))<\varepsilon\}),
\end{equation}
which can be deduced from \eqref{MesureSortie}.
The process $(\mathbf{P}_r)_{r\in(0,\br_*)}$ has c\`adl\`ag sample paths and no positive jumps.

	\paragraph{The Bettinelli-Miermont construction of the Brownian disk}
	
	We now present a construction of the free pointed Brownian disk, which is the compact metric space that appears as the scaling limit of Boltzmann triangulations in $\mathbb T^{1, \bullet}(L)$. We fix $a>0$ and let $(\mathtt e(t))_{t\in [0, a]}$ be a positive Brownian excursion of duration $a$. Conditionally on $(\mathtt e(t))_{t\in[0,a]}$, let $\mathcal N=\sum_{i\in I} \delta_{(t_i, \omega^i)}$ be a Poisson point measure on $[0, a]\times \mathcal S$ with intensity $
	2\,\dd t \,\mathbb N_{\sqrt 3 \mathtt e(t)}(\dd\omega)$. We let $\mathcal I$ be  the quotient space of
	\begin{align*}
		[0, a]\cup \bigcup_{i\in I} T_{\omega^i}\,,
	\end{align*} for the equivalence relation that identifies $\rho_{\omega^i}$ and $t_i$ for every $i\in I$
	(and no other pair of points is identified). We endow $\mathcal I$ with the maximal distance $d_{\mathcal I}$ whose restriction to each tree $T_{\omega^i}$ coincides with $d_{\omega^i}$, and whose restriction to $[0,a]$
	is the usual distance. More explicitly, the distance between two points $x\in T_{\omega^i}$ and $y\in T_{\omega^j}$, $i\neq j$ is given by $d_{\omega^i}(x, \rho_{\omega^i})+|t_i-t_j|+ d_{\omega^j}(y, \rho_{\omega^ j})$. Then $\mathcal I$ is a compact metric space (in fact, a compact $\R$-tree), and we can consider the labelling $\ell:\mathcal I\to \mathbb R$ defined by:
	\begin{align*}
		\ell(x)=\left\{\begin{array}{l}
			\ell_{\omega^i}(x) \ \ \ \ \text{ if $x\in T_{\omega^i}$, for some $i\in I$},\\
			\noalign{\smallskip}
			\sqrt 3\mathtt e(x) \ \ \ \text{ if $x\in [0, a]$.}
		\end{array}\right..
	\end{align*}
By standard properties of the It\^o measure, one verifies that the quantity $\Sigma:=\sum_{i\in I}\sigma(\omega^i)$ is almost surely finite and it is possible to concatenate the functions $p_{\omega^i}$ to obtain a 
``contour exploration'' $\pi:[0, \Sigma]\to \mathcal I$. Formally, to define $\pi$, let $\mu=\sum_{i\in I}\sigma(\omega^i)\delta_{t_i}$ be the point measure on $[0, a]$ giving weight $ \sigma(\omega^i)$ to $t_i$, for every $i\in I$, and consider the left-continuous inverse $\mu^{-1}$ of its cumulative distribution function, $\mu^{-1}(s):=\inf\{t\in[0,a]: \mu([0, t])\geq s\}$ for every $s\in[0,\Sigma]$.
For every $s\in [0, \Sigma]$, we set $\pi(s)=\omega^{i}(s-\mu([0, \mu^{-1}(s)))) $ if $\mu^{-1}(s)=t_i$ for some $i\in I$ and $\pi(s)=\mu^{-1}(s)$ otherwise.\\
	
This contour exploration $\pi$ allows us to define intervals on $\mathcal I$, in a way similar to what we did on $T_{\omega}$. For every $u, v\in \mathcal I$, there exists a 
smallest interval $[s, t]$ in $[0, \Sigma]$ such that $\pi(s)=u$ and $\pi (t)=v$, where by convention $[s, t]=[s, \Sigma]\cup[0, t]$ if $s> t$, and we write $\lbracket u, v\rrbracket$ for the subset of $\mathcal I$ defined by $\lbracket u, v\rrbracket=\{\pi(b), \ b\in [s, t]\}$. We then set
	\begin{align*}
		\forall u, v\in \mathcal I, \ \ D^\circ(u, v):=\ell(u)+\ell(v)-2\max\left(\min_{w\in \lbracket u, v\rrbracket}\ell(w), \min_{w\in \lbracket v, u\rrbracket}\ell(w)\right),
	\end{align*} 
and we consider the pseudo-metric $D$ on $\mathcal{I}$ defined for $u, v\in \mathcal I$ by:
	\begin{align*}
		D(u, v):= \inf_{u=u_0,u_1, \ldots, u_p=v}\sum_{j=1}^pD^\circ(u_i, u_{i+1}),
	\end{align*}
where  the infimum is taken over all choices of the integer $p\geq 1$ and the points $u_0,\ldots, u_p\in \mathcal I$
such that $u_0=u$ and $u_p=v$.
	The space $\mathbb{D}_{(a)}$ is defined as the quotient space  $\mathcal I/\{D=0\}$, which we equip with the distance induced by $D$, for which we keep
	the notation $D$. Then $\mathbb{D}_{(a)}$ is a compact metric space.
	
Let $\Pi:\mathcal I\to \mathbb{D}_{(a)}$ be the canonical projection. It is easy to verify that $\Pi(a)=\Pi(b)$
implies $\ell(a)=\ell(b)$, and so $\mathbb{D}_{(a)}$ inherits a labelling function, still denoted by $\ell(\cdot)$ from the labelling of $\mathcal I$. 
We can then define:
	\begin{itemize}
		\item[\textbullet] $\mathbf V=(\Pi\circ \pi)_* \lambda_{[0, \Sigma]}$, where $\lambda_{[0, \Sigma]}$ denotes Lebesgue measure on $[0, \Sigma]$. This is a finite Borel measure on $\mathbb{D}_{(a)}$ called the volume measure.
		\item[\textbullet] $\partial \mathbb{D}_{(a)}:=\Pi([0, a])$, which is the ``boundary'' of $\mathbb{D}_{(a)}$.
				\item[\textbullet] $x_*$ is the point of minimal label in $\mathbb{D}_{(a)}$.	\end{itemize}
	\noindent We then view $((\mathbb{D}_{(a)}, D), (x_*, \partial \mathbb{D}_{(a)}), \mathbf{V})$ as a random variable in $\mathbb M^{2, 1} $. This is  the \emph{free pointed Brownian disk of perimeter $a$}. As the (free) Brownian sphere, the (free pointed) Brownian disk is a geodesic space. 
	
	In a way similar to the Brownian sphere, we have $D(x,x_*)=\ell(x)-\ell(x_*)$ for every $x\in\D_{(a)}$. In particular, if we set
 $r_*:= -\ell(x_*) =-\min_{x\in \mathbb{D}_{(a)} }\ell(x)$ we have $r_*=D(x_*, \partial \mathbb{D}_{(a)})$ (note that $\ell(u)=\sqrt{3}\,\mathtt{e}(u)\geq 0$ for every $u\in [0,a]\subset\mathcal{I}$).

	Occasionally (in particular in Proposition \ref{BrDisk-BrSphere} below), we will also say that 
	the space $((\mathbb{D}_{(a)},D),x_*,\mathbf{V})$ --- which is a random element of $\mathbb{M}^{1,1}$ --- is a free pointed Brownian disk of perimeter $a$:
	This makes no real difference, as the boundary $\partial\D_{(a)}$ can be recovered as the closed subset of $\D$ consisting of points that have
	no neighborhood homeomorphic to the open unit disk.
	
	\paragraph{Hulls in the Brownian disk} Consider the Brownian disk $\D_{(a)}$ as defined above. For every $r>0$, let $B_r(x_*)$ stand for the closed ball of radius $r$
	centered at $x_*$ in $\mathbb{D}_{(a)}$. For every $r\in(0,r_*]$, we define the hull $H_r$ as the complement in $\mathbb{D}_{(a)}$ of the 
	connected component of $\mathbb{D}_{(a)}\setminus B_r(x_*)$ that intersects the boundary $\partial\mathbb{D}_{(a)}$
	(in fact, for $r<r_*$, this connected component must contain the whole boundary). Points of $\partial H_r$ are at distance $r$ from $x_*$.  In a way analogous to the definition of 
	$\mathbf{P}_r$ for the Brownian sphere, we define the perimeter
	of $H_r$ by
	\begin{equation}
	\label{def-peri-hull}
	\mathcal{P}_r=\sum_{i\in I} \z_{r-r_*}(\omega^i).
	\end{equation}
	
	Then $\mathcal{P}_r$ satisfies a formula analogous to \eqref{approx-exit} (if $r<r_*$,  there are only finitely many
	indices $i$ such that $\z_{r-r_*}(\omega^i)>0$). We also take $\mathcal{P}_0=0$. The process $(\mathcal{P}_r)_{r\in[0,r_*]}$ has c\`adl\`ag sample paths and no positive jumps.
	
	\begin{proposition}
	\label{law-perimeter-hull}
	Let $r>0$. Then the law of $\mathcal{P}_r$ under $\P(\cdot \cap \{r<r_*\})$ has density
	$$y\mapsto 3\sqrt{\frac{3}{2\pi}}\,r^{-3}\,\frac{a}{a+y}\,\sqrt{y}\,e^{-3y/(2r^2)}$$
	with respect to Lebesgue measure on $(0,\infty)$.
	\end{proposition}
	
	We postpone  the proof to the Appendix, as this result is not really needed in what follows.

It will be useful to describe the hull $H_r$ in terms of the labelled tree $\mathcal{I}$ of the Bettinelli-Miermont construction. 
Let $x\in\mathcal{I}$ and suppose first that  $x\in T_{\omega^i}$ for some $i\in I$. Since $T_{\omega^i}$ is an $\R$-tree, there is a unique continuous injective path linking $x$ to the root $\rho_{\omega^ i}$ of $T_{\omega^i}$, which is  called the ancestral line of $x$. We let $m_x$ be the minimum label along this path.
	If $x\in[0,a]$, we take $m_x=\ell(x)$.  Then we have $m_x=m_y$ if $\Pi(x)=\Pi(y)$,
	and thus the mapping $\mathcal{I}\ni x\mapsto m_x$ induces a continuous function from $\mathbb D_{(a)}$ to $\mathbb R$ which we denote again by $\mathbb D_{(a)}\ni u \mapsto m_u$. Using the cactus bound (see \cite[Proposition 3.1]{CactusBound} for this bound in the setting of the Brownian sphere, 
			which is easily extended), one gets that:
	\begin{align*}
		H_r= \left\{u\in \mathbb D, \ m_u\leq -r_*+r\right\}.
	\end{align*}
Similarly, the boundary $\partial H_r$ of $H_r$ in $\mathbb D$ is the image under $\Pi$ of the set of all points $x\in \mathcal I$ such that 
we have both $\ell(x)=r-r_*$ and 
all points of the ancestral line of $x$ (with the exception of $x$) have a label greater than $r-r_*$. 

	\paragraph{Brownian disks in the Brownian sphere}
	
	We now explain how the free pointed Brownian disk of the
	previous section can be obtained as a subset of the free Brownian sphere
	under a particular conditioning of the measure $\N_0$. We first recall a result from \cite{Stars}.
	Let $r>0$, and argue under the conditional probability measure $\N_0(\cdot\mid \br_*>r)$. 
	We can then consider the hull $B^\bullet_r(\bx_*)$, and we write $\check B^\circ_r(\bx_*)=\bm_\infty\setminus B^\bullet_r(\bx_*)$, and $\check B^\bullet_r(\bx_*)$
	for the closure of $\check B^\circ_r(\bx_*)$.
	We equip the open set $\check B^\circ_r(\bx_*)$ with the intrinsic metric $\mathbf{d}^\circ$:
	for every $x,y\in \check B^\circ_r(\bx_*)$, $\mathbf{d}^\circ(x,y)$ is the infimum of lengths of 
	continuous paths connecting $x$ to $y$ that stay in $\check B^\circ_r(\bx_*)$.
	Then, according to \cite[Theorem 8]{Stars}, under the probability measure $\N_0(\cdot\mid \br_*>r)$, the intrinsic metric on the set $\check B^\circ_r(\bx_*)$
	has a continuous extension to its closure $\check B^\bullet_r(\bx_*)$, which is a metric on $\check B^\bullet_r(\bx_*)$,
	and the random metric space $(\check B^\bullet_r(\bx_*),\mathbf{d}^\circ)$
	equipped with the restriction of the volume measure on $\bm_\infty$ and with the distinguished
	point $\bx_0$ is a free pointed 
	Brownian disk of (random) perimeter $\z_r$. 
	
	For our purposes, it will be useful to have a version of the preceding result when 
	$r$ is replaced by a random radius. For every $a>0$, we define, under $\N_0$,
	$$\br_a:=\inf\{r\in(0,\br_*): \z_{r-\br_*}=a\}$$
	with the usual convention $\inf\varnothing=\infty$. 
	By \cite[Lemma 9]{GrowthFrag},
	we have $\N_0(\br_a<\infty)=(2a)^{-1}$. For future use, we record the following simple 
	fact. If $(a_n)_{n\in\N}$ is a sequence decreasing to $a$, we have  $\br_{a_n}\downarrow \br_a$ as $n\to\infty$, $\N_0$ a.e.
	on the event $\{\mathbf{r}_a<\infty\}$. This follows from the description of the
	law of the process $(\z_r)_{r<0}$ under $\N_0$, as a time change of the excursion of a stable L\'evy process, see \cite[Lemma 12]{GrowthFrag}.
	
	\begin{proposition}
	\label{BrDisk-BrSphere}
	Let $a>0$. Almost surely under the probability measure $\N_0(\cdot \mid \br_a<\infty)$, the intrinsic measure on the set
	$\check B^\circ_{\br_a}(\bx_*)$ has a continuous extension to its closure $\check B^\bullet_{\br_a}(\bx_*)$,
	which is a metric on $\check B^\bullet_{\br_a}(\bx_*)$, and the resulting random metric space equipped with the restriction of the volume measure on $\bm_\infty$ and with the distinguished
	point $\bx_0$ is a free pointed 
	Brownian disk of perimeter $a$.
	\end{proposition}
	
	The shortest way to prove this proposition is to use Proposition 10 in \cite{GrowthFrag}, which determines the distribution 
	under $\N_0(d\omega\mid \br_a<\infty)$ of the snake trajectory $\omega$ truncated at level $\br_a-\br_*$, which is denoted
	by $\mathrm{tr}_{\br_a-\br_*}(\omega)$ (we refer e.g.
	to \cite[Section 2.2]{GrowthFrag} for a definition of this truncation operation). 
	On one hand, the space $\bm_\infty\setminus B^\bullet_{\br_a}(\bx_*)$  equipped with its intrinsic measure 
	can be obtained as a function of $\mathrm{tr}_{\br_a-\br_*}(\omega)$, as it is explained in the
	proof of \cite[Theorem 8]{Stars}. On the other hand, 
	Proposition 10 in \cite{Stars} shows that this snake trajectory 
	has exactly the distribution of the random snake trajectory that codes the free pointed Brownian disk in the construction of
	\cite[Definition 13]{spine} --- which is known to be equivalent to the Bettinelli-Miermont construction presented above.
	We omit the details, since Proposition \ref{BrDisk-BrSphere} is clearly a variant of Theorem 8 in \cite{Stars}. 
	
	Proposition \ref{BrDisk-BrSphere} allows us to couple Brownian disks with different perimeters. Consider a decreasing sequence
	$(a_n)_{n\in\N}$ that converges to $a$. On the event $\{\br_{a_n}<\infty\}$, $\check B^\bullet_{\br_{a_n}}(\bx_*)$ and
	$\check B^\bullet_{\br_a}(\bx_*)$ are both well defined, and we have trivially $\check B^\bullet_{\br_{a_n}}(\bx_*)\subset \check B^\bullet_{\br_a}(\bx_*)$.
	Furthermore, a.e. on the event $\{\br_a<\infty\}$, we have $\br_{a_n}<\infty$ for all $n$ large enough, $\br_{a_n}\downarrow \br_a$ as $n\to\infty$, and
	\begin{equation}
	\label{tech-point}
	\sup\{\mathbf{D}(x,\partial \check B^\bullet_{\br_a}(\bx_*)):x\in \check B^\bullet_{\br_a}(\bx_*)\setminus \check B^\bullet_{\br_{a_n}}(\bx_*)\}\xrightarrow[n\to\infty]{} 0.
	\end{equation}
	Let us justify \eqref{tech-point}. First note that, for every $x\in  \check B^\circ_{\br_a}(\bx_*)$, there is a path from $x$ to $\bx_*$ that does not hit
	$B^\bullet_{\br_a}(\bx_*)$, and thus stays at positive distance from $\partial \check B^\bullet_{\br_a}(\bx_*)$. Since  $\br_{a_n}\downarrow \br_a$, it follows
	that $x\in \check B^\circ_{\br_{a_n}}(\bx_*)$ for $n$ large enough, and we have proved that, a.e. on the event $\{\br_a<\infty\}$,
	$$\check B^\circ_{\br_{a}}(\bx_*)=\bigcup_{n\in\N,\br_{a_n}<\infty} \check B^\circ_{\br_{a_n}}(\bx_*),$$
	from which \eqref{tech-point} easily follows via a compactness argument.
	
\section{The Brownian annulus}
\label{sec:annu-def}
	
	\subsection{The definition of the Brownian annulus}
	\label{sec:def-annulus}
	
	We again fix $a>0$ and write $(\mathbb{D}_{(a)},D)$ for the free pointed Brownian disk of perimeter $a$ in the Bettinelli-Miermont
	construction described above. Recall the
	notation $x_*$ for the distinguished point of $\mathbb{D}_{(a)}$ and $r_*=D(x_*,\partial \mathbb{D}_{(a)})$. Also recall that
	$\mathcal{P}_r$ stands for the perimeter of the hull $H_r$ of radius $r$.
	We  fix $b>0$, and set
	$$r_b=\inf\{r\in[0,r_*): \mathcal{P}_r=b\},$$
	with again the convention $\inf\varnothing=\infty$. Note that $r_b<\infty$ if and only 
	$b<\mathcal{P}^*$, where $\mathcal{P}^*=\sup\{\mathcal{P}_r:r\in[0,r_*)\}$.

	The next theorem is then an analog of Proposition \ref{BrDisk-BrSphere}.
	
	\begin{theorem}
	\label{Br-Annulus}
	Almost surely under the probability measure $\P(\cdot\mid r_b<\infty)$, 
	the intrinsic metric on $\mathbb{D}_{(a)}\setminus H_{r_b}$
	has a continuous extension to the closure of $\mathbb{D}_{(a)}\setminus H_{r_b}$,
	which is a metric on this set. The resulting random metric space, which we denote 
by $(\C_{(a,b)},d^\circ)$, is called the Brownian annulus with
	perimeters $a$ and $b$.
	\end{theorem}	
	
	The terminology will be justified by forthcoming results showing that the
	Brownian annulus is the Gromov-Hausdorff limit of triangulations 
	with two boundaries. We note that the Brownian annulus $\C_{(a,b)}$
	has two ``boundaries'', namely $\partial_0\C_{(a,b)}=\partial \D_{(a)}$, and
	$\partial_1\C_{(a,b)}=\partial H_{r_b}$. Furthermore, distances in $\C_{(a,b)}$ 
	from the second boundary $\partial_1\C_{(a,b)}$ correspond to labels in the
	Bettinelli-Miermont construction. More precisely, for every $z\in \C_{(a,b)}$,
	\begin{equation}
	\label{dist-cyl}
	D(z,\partial_1\C_{(a,b)})=D(z,x_*)-r_b=\ell(z)-(r_b-r_*).
	\end{equation}
	This follows from the interpretation of labels in terms of distances from $x_*$,
	recalling that all points of $\partial_1\C_{(a,b)}=\partial H_{r_b}$ are at distance $r_b$
	from $x_*$. 
	
	\proof
	We may and will assume that the Brownian disk $\mathbb{D}_{(a)}$ is constructed 
	as the subset $\check B^\bullet_{\br_a}(\bx_*)$ of the free Brownian sphere $\bm_\infty$ under the probability measure
	$\N_0(\cdot\mid \br_a<\infty)$,
	as in Proposition \ref{BrDisk-BrSphere}, and, in particular, the distinguished
	point of $\mathbb{D}_{(a)}$ is the point $\bx_0$ of the Brownian sphere. 
	Furthermore,
	for every $r\in (0,D(\bx_0,\partial\mathbb{D}_{(a)}))$, the hull $H_r$ in the Brownian disk $\mathbb{D}_{(a)}$
	coincides with
	the hull $B^\bullet_r(\bx_0)$ in $\bm_\infty$ (defined as the 
	complement of the connected component of $\bm_\infty\setminus B^\infty_r(\bx_0)$ that contains
	$\bx_*$). In particular, on the event $\{r_b<\infty\}$, we have $r_b=\tilde \br_b$, where
	$\tilde \br_b$ is the hitting time of $b$ by the process of perimeters 
	of the hulls $B^\bullet_r(\bx_0)$, $r\in (0,\mathbf{r}_*)$. Furthermore,
	conditioning $\mathbb{D}_{(a)}$ on the event that $r_b<\infty$ is equivalent 
	to arguing under the conditional probability $\N_0(\cdot \mid\mathbf{D}(\bx_0,\bx_*)>\br_a+\tilde\br_b)$. 
	
	Now note that $\bx_*$ and $\bx_0$ play symmetric roles in the Brownian sphere $\bm_\infty$
	(cf. \cite[Proposition 3]{Stars}), and that proving that the intrinsic metric on $\mathbb{D}_{(a)}\setminus H_{r_b}$
	has a continous extension, which is a metric, to its closure is equivalent to proving that
	the intrinsic metric on $\bm_\infty\setminus B^\bullet_{\tilde\br_b}(\bx_0)$ has a
	continuous extension, which is a metric, to its closure. By symmetry, this equivalent
	to proving that the metric on $\bm_\infty\setminus B^\bullet_{\br_b}(\bx_*)$ has a
	continuous extension, which is a metric, to its closure. But we know from 
	Proposition \ref{BrDisk-BrSphere} that this is true. \endproof
	
	It turns out that the probability of the conditioning event $\{r_b<\infty\}$ has a very simple expression, 
	which will be useful in forthcoming calculations.
	
	\begin{lemma}
	\label{proba-condit}
	We have ${\displaystyle \P(r_b<\infty)=\frac{a}{a+b}}$.
	\end{lemma}
	
	\proof Let us set $\check{\mathcal{P}}_r=\mathcal{P}_{r_*-r}$ for $r\in[0,r_*]$, so that
	$$\check{\mathcal{P}}_r=\sum_{i\in I} \z_{-r}(\omega^i),$$
	in the notation of \eqref{def-peri-hull}. From the identification of the law of the exit measure process under $\N_0$
	(see e.g.~Section 2.4 in \cite{GrowthFrag}), it is not hard to verify that $(\check{\mathcal{P}}_r)_{r\in[0,r_*]}$ is
	a continuous-state branching process with branching mechanism $\psi(\lambda):=\sqrt{8/3}\,\lambda^{3/2}$.
	Furthermore, Remark (ii) at the end of \cite[Section 5]{MarkovSpatial} shows that the initial value $\check{\mathcal{P}}_0=\mathcal{P}_{r_*}$ of this continuous-state branching process has density 
	$\frac{3}{2}\,a^{3/2}\,(a+z)^{-5/2}$. The classical Lamperti transformation allows us to write $(\check{\mathcal{P}}_r)_{r\in[0,r_*]}$ 
	as a time change of a (centered) spectrally positive L\'evy process with Laplace exponent $\psi$ and the same initial distribution, which 
	is stopped upon hitting $0$. For this L\'evy process started at $z$, the probability that it never hits $b$
	is equal to $\sqrt{(b-z)^+/b}$ (cf.~\cite[Theorem VII.8]{Bertoin}). From the preceding considerations, we get
	$$\P(r_b=\infty)= \frac{3}{2}\,a^{3/2}\,\int_0^b \frac{dz}{(a+z)^{5/2}}\,\sqrt{\frac{b-z}{b}}= \frac{b}{a+b}.$$
	This completes the proof.\endproof
	
	\subsection{A technical lemma}
	
	We keep the notation of the preceding section.
	In the following lemma,
	lengths of paths refer to the metric on the Brownian disk $\D_{(a)}$.
	
	\begin{lemma}
	\label{avoiding-path}
	Let $\eta>0$. Then, almost surely, for every $x,y\in\C_{(a,b)}\setminus \partial_1\C_{(a,b)}$, for every continuous path $\gamma$
	in $\C_{(a,b)}$ connecting $x$ to $y$ and with finite length $L(\gamma)$, we can find 
	a path $\gamma'$ staying in $\C_{(a,b)}\setminus \partial_1\C_{(a,b)}$ and connecting $x$ to $y$, whose length is
	bounded by $L(\gamma)+\eta$. 
	\end{lemma}
	
	\proof Let us set $\C^\circ_{(a,b)}=\C_{(a,b)}\setminus(\partial_0\C_{(a,b)}\cup\partial_1\C_{(a,b)})$ 
	which can be viewed as the ``interior'' of 
	$\C_{(a,b)}$. In order to prove Lemma \ref{avoiding-path}, it is enough to consider the case where $x,y\in\C^\circ_{(a,b)}$ and the path 
	$\gamma$ stays in $\C_{(a,b)}\setminus \partial_0\C_{(a,b)}$. If not the case, we can cover the set of times $t$ at which 
	$\gamma(t)$ belongs to $\partial_1\C_{(a,b)}$ by finitely many disjoint closed intervals $I=[s_I,t_I]$ such that $\gamma(t)\in \C_{(a,b)}\setminus \partial_0\C_{(a,b)}$
	for every $t\in I$ and $\gamma(s_I),\gamma(t_I)\notin \partial_1\C_{(a,b)}$, and we consider the restriction of $\gamma$ to each of these intervals. 
	
	Fix $\ve>0$ and, 
	for every $u>0$, let $E_\ve(a,u)$ denote the event where $u<\mathcal{P}^*$ and there exist $x,y\in \C^\circ_{(a,u)}$
	and a path $\gamma_0$ with finite length $L(\gamma_0)$ connecting $x$ to $y$ and staying in $\C_{(a,u)}\setminus\partial_0\C_{(a,u)}$, such that any path
	$\gamma'$ connecting $x$ to $y$ and staying in $\C^\circ_{(a,u)}$ has length at least $L(\gamma_0)+\varepsilon$. 
	Also set, for every  $u\in(0,\mathcal{P}^*)$ and $x,y\in \C^\circ_{(a,u)}$, 
	$$F(x,y,u)=\inf\{L(\gamma):\gamma\hbox{ is a path connecting }x\hbox{ to }y\hbox{ in }\C^\circ_{(a,u)}\}.$$
	If $E_\ve(a,u)$ holds, then clearly there exist $x,y\in \C^\circ_{(a,u)}$ such that the function $v\mapsto F(x,y,v)$ has a (positive) jump at $v=u$ 
	(take $\gamma_0$ as above and note that $F(x,y,v)\leq L(\gamma_0)$ if $0<v<u$). The same then holds for every $x',y'\in \C^\circ_{(a,u)}$
	sufficiently close to $x,y$: 
	To see this, consider the path obtained by concatenating $\gamma_0$ with geodesics from $x$ to $x'$ and from $y$ to $y'$. Hence, if for $n\geq 1$, we consider
	the monotone nonincreasing function
	$$(0,\mathcal{P}^*)\ni v \mapsto G_n(a,v)=\int_{\C^\circ_{(a,v)}} (n-F(x,y,v))^+\,\mathbf{V}(\dd x)\mathbf{V}(\dd y)$$
	we obtain that this function has a jump at $u$ when $E_\ve(a,u)$ holds, at least when $n$ is large enough. It follows that
	$$\mathbf{1}_{E_\ve(a,u)}\leq \liminf_{n\to\infty} \mathbf{1}_{\{G_n(a,u+)<G_n(a,u-)\}},$$
with an obvious notation for the right and left limits of $v \mapsto G_n(a,v)$ at $u$. Hence,
$$\P(E_\ve(a,u))\leq \liminf_{n\to\infty} \P(\{u<\mathcal{P}^*\}\cap \{G_n(a,u+)<G_n(a,u-)\})).$$
Since the function $(0,\mathcal{P}^*)\ni v \mapsto G_n(a,v)$ has at most countably many discontinuities, if follows that
$$\int_0^\infty \P(E_\ve(a,u))\,\dd u=0$$
and therefore $\P(E_\ve(a,u))=0$ for Lebesgue almost all $u$. 

To obtain the statement of the lemma, we need
to prove that $\P(E_\ve(a,u))=0$ for {\it every} $u>0$. Fix $u>0$, and let $(a_n)_{n\geq 0}$ be a sequence of reals decreasing to $a$.
We will verify that
\begin{equation}
\label{avoiding-path-tec}
\liminf_{n\to\infty} \P(E_\ve(a_n,u))\geq \P(E_\ve(a,u)).
\end{equation}
Thanks to Proposition \ref{BrDisk-BrSphere}, we may assume that $\D_{(a)}=\check B^\bullet_{\br_a}(\bx_*)$, resp. $\D_{(a_n)}=\check B^\bullet_{\br_{a_n}}(\bx_*)$,
which is a Brownian disk of perimeter $a$, resp. of perimeter $a_n$, under $\N_0(\cdot\mid \br_a<\infty)$, resp. under $\N_0(\cdot\mid \br_{a_n}<\infty)$. 
If $E_\ve(a,u)$ holds, we can find a path $\gamma_0$ staying in $\C_{(a,u)}\setminus\partial\D_{(a)}$ that
satisfies the properties stated at the beginning of the proof, and this path stays at positive distance from $\partial\D_{(a)}$. On the other hand, by \eqref{tech-point}, we have
$$\sup\{\mathbf{D}(x,\partial \D_{(a)}):x\in \D_{(a)}\setminus \D_{(a_n)}\} \xrightarrow[n\to\infty]{} 0,$$
$\N_0$ a.e. on $\{\br_a<\infty\}$. It follows that the path $\gamma_0$ stays in 
$\C_{(a_n,u)}\setminus \partial\D_{(a_n)}$ when $n$ is large, so that $E_\ve(a_n,u)$ also holds when $n$ is large. Hence, we get
$$\liminf_{n\to\infty} \N_0(E_\ve(a_n,u)\cap\{\br_a<\infty\})\geq \N_0(E_\ve(a,u)\cap\{\br_{a}<\infty\}),$$
and using also the fact that $\N_0(\br_{a_n}<\infty)\longrightarrow \N_0(\br_a<\infty)$ as $n\to\infty$
we get \eqref{avoiding-path-tec}.

From \eqref{avoiding-path-tec} and a scaling argument, we have then
$$\liminf_{u'\uparrow u} \P(E_\ve(a,u'))\geq \P(E_\ve(a,u)).$$
Clearly, this implies that we have $\P(E_\ve(a,u))=0$. Since $\ve>0$ was arbitrary, this completes the proof. \endproof

\section{Preliminary convergence results}
\label{sec:preli-conv}
	\subsection{Convergence towards the Brownian disk}

Let $a>0$. For every integer $L\geq 1/a$, let $\mathcal D^L_{(a)}$ a Boltzmann triangulation in $\mathbb T^{1, \bullet}(\lfloor aL\rfloor)$. Let $\Delta^L$ be the graph distance on $\mathcal D^L_{(a)}$ and consider the rescaled distance  $d_L=\sqrt{3/2}\,L^{-1/2}\Delta^L$. Let $\nu^L$ be the counting measure, rescaled by the factor $(3/4)L^{-2}$, on the vertex set of $\mathcal D^L_{(a)}$. Then,
	\begin{align}
		\label{ConvergenceBrownianDisk}
		((\mathcal D^L_{(a)}, d_L), (x_*^L, \partial \mathcal{D}^L_{(a)}), \nu^L) \xrightarrow[L\to\infty]{(d)}((\mathbb{D}_{(a)}, D), (x_*, \partial \mathbb{D}_{(a)}), \mathbf{V}),
	\end{align}
	where $((\mathbb{D}_{(a)}, D), (x_*, \partial \mathbb{D}_{(a)}), \mathbf{V})$ is the free pointed Brownian disk with perimeter $a$ as constructed 
in Section \ref{sec:cons}, and the convergence holds 
in $\mathbb M^{2,1}$ endowed with the metric $d_{\mathtt {GHP}}^{2,1}$ introduced in Section \ref{sec:conv}. In the last display, we abusively identify $\mathcal D^L_{(a)}$ with its vertex set 
(we will often make this abuse of notation in what follows).

The convergence \eqref{ConvergenceBrownianDisk} follows from \cite[Theorem 1.1]{albenque2020scaling}. Note that Theorem 1.1 in \cite{albenque2020scaling} deals with the so-called 
GHPU convergence including the uniform convergence of the ``boundary curves'', but it is straightforward to verify that this also implies the convergence
\eqref{ConvergenceBrownianDisk}  in $\mathbb M^{2,1}$. Also, \cite{albenque2020scaling} considers Boltzmann triangulations in $\mathbb T^{1}(\lfloor aL\rfloor)$ instead of $\mathbb T^{1, \bullet}(\lfloor aL\rfloor)$,
and the limit is therefore the free (unpointed) Brownian disk. However, as explained in \cite[Section 3.4]{MarkovSpatial}, the convergence for pointed objects easily follows 
from that for pointed ones (since we are here pointing at an edge and not at a point, we also need Lemma 5.1 in \cite{ABS}, stated for quadrangulations but easily extended, to verify that the
degree-biased measure on the vertex set is close to the uniform measure --- we omit the details).

\subsection{The processes of perimeters and volumes of hulls}
\label{sec:process-peri-vol}

We consider the free pointed Brownian disk $((\mathbb{D}_{(a)}, D), (x_*, \partial \mathbb{D}_{(a)}), \mathbf{V})$
as given in the Bettinelli-Miermont construction. Recall that $r_*=D(x_*, \partial \mathbb{D}_{(a)})$. For $r\in(0,r_*]$,
the perimeter $\mathcal{P}_r$ of the hull $H_r$ was defined in formula \eqref{def-peri-hull}, and we set
$$\mathcal{V}_r=\mathbf{V}(H_r).$$
We also define $\mathcal{V}_0=0$. It is not hard to verify that the process 
$(\mathcal{P}_r,\mathcal{V}_r)_{r\in[0,r_*]}$ has c\`adl\`ag sample paths.

Let $r>0$ and let us argue conditionally on the event $\{r_*> r\}$. Recall that 
$\D_{(a)}$ is obtained as a quotient space of the labelled tree $\mathcal{I}$, and that, for $x\in\mathcal{I}$, $m_x$ denotes the minimal label
along the ancestral line of $x$. We can use the restriction of the contour exploration $\pi:[0,\Sigma]\longrightarrow \mathcal{I}$ to every connected component of the open set $\{s\in[0,\Sigma] : m_{\pi(s)}< r-r_*\}$, in order
to define a snake trajectory with initial point $r-r_*$, which we call an excursion away from $r-r_*$. More precisely, if $(\alpha,\beta)$ is such a connected component, there is an index $i\in I$ such that $(\alpha,\beta)\subset (a_i,b_i)$,
where $[a_i,b_i]=\{s\in[0,\Sigma]:\pi(s)\in\ T_{\omega^i}\}$. Then, setting $\alpha'=\alpha-a_i$ and $\beta'=\beta-a_i$,  we have $\omega^i_{\alpha'}=\omega^i_{\beta'}$, $\hat\omega^i_{\alpha'}=\hat\omega^i_{\beta'}=r-r_*$ and $\zeta(\omega^i_s)>\zeta(\omega^i_{\alpha'})$ for every $s\in({\alpha'},{\beta'})$, and we define a snake trajectory 
$\omega$ by taking $\omega_s(t)=\omega^i_{({\alpha'}+s)\wedge{\beta'}}(\zeta(\omega^i_{\alpha'})+t)$ for every $0\leq t\leq \zeta(\omega^i_{({\alpha'}+s)\wedge{\beta'}})-\zeta(\omega^i_{\alpha'})$ (in the 
language of \cite{MarkovSpatial}, $\omega$ is an
excursion of $\omega^i$ away from  $r-r_*$). As a straightforward consequence of Proposition $12$ in \cite{BrowDiskandtheBrowSnake}, the snake trajectories obtained in this way and shifted so that
there initial point is $0$ correspond to the atoms of a  point measure $\mathcal{N}_r$ which conditionally on $\mathcal{P}_r$ is Poisson with intensity $\mathcal P_r\mathbb N_0(\cdot \cap \{W_* >-r\})$ 
and to which we add an extra atom $\omega_*$ distributed 
according to $\mathbb N_0(\cdot\mid W_*=-r)$ (the law of the latter atom is described in \cite{BesselProc} in terms of a Bessel process of dimension $9$). Using formula \eqref{MesureSortie}, 
it is not hard to verify that the process $(\mathcal P_s,\mathcal V_s)_{s\in[0,r]}$ is determined as a function of the point measure $\mathcal{N}_r+\delta_{\omega_*}$ (in particular,
$\mathcal P_s=\mathcal{Z}_{s-r}(\omega_*)+\int \mathcal{N}_r(d\omega)\,\mathcal{Z}_{s-r}(\omega)$ for $0<s<r$). 

Let us now consider the Brownian plane of \cite{Hull}. For the Brownian plane, we can also define the processes of perimeter and volume  of hulls $(\mathcal P_s^\infty, \mathcal V_s^ \infty)_{s\geq 0}$
and the law of this pair of processes is described in \cite{Hull}. It follows from the preceding observations and the construction of \cite{Hull} that, for every $u>0$, the conditional distribution of $(\mathcal P^\infty_s,\mathcal V^\infty_s)_{s\in[0,r]}$
knowing $\mathcal{P}^\infty_r=u$ is the same as the conditional distribution of $(\mathcal P_s,\mathcal V_s)_{s\in[0,r]}$
knowing $\mathcal{P}_r=u$. Since $\mathcal P_r$ and $\mathcal P_r^\infty$ both have a positive density on $(0,\infty)$ (by Proposition \ref{law-perimeter-hull} and \cite[Proposition 1.2]{Hull}),
we arrive at the following lemma.

	\begin{lemma} \label{AbsContinuity} The law of $(\mathcal P_s, \mathcal V_s)_{s\leq r}$ conditionally on the event $\{r_*> r\}$ is absolutely continuous with respect to the law of the pair $(\mathcal P_s^\infty, \mathcal V_s^\infty)_{s\leq r}$.
	\end{lemma}

	We end this section by stating a technical property showing that the perimeter process can be recovered as
	a deterministic function of the volume process. 
	
	\begin{lemma}
		For every $r>0$, we have almost surely on the event $\{r<r_*\}$:
		\begin{align}
			\label{PrFromJumps}
			\mathcal P_r =\lim_{\alpha\to 0^+}\Bigg(\frac{1}{\alpha} \lim_{\epsilon\to 0^+}\phi(\epsilon)^{-1} \mathrm{Card}\left\{s\in [r-\alpha, r]:\Delta\mathcal V_s\geq \epsilon\right\}\Bigg),
		\end{align}
		where $\phi(\epsilon)=c_0\epsilon^{-3/4}$, with $c_0=2^{1/4}\Gamma(4/3)$, and $\Delta\mathcal V_s=\mathcal V_s-\mathcal V_{s-}$.
	\end{lemma}
	
	\begin{proof}
		It is explained in the proof of \cite[Proposition 21]{BrowDiskandtheBrowSnake} that (\ref{PrFromJumps}) holds if $\mathcal{P}_r$ and $\mathcal{V}_s$
		are replaced by $\mathcal{P}^\infty_r$ and $\mathcal{V}^\infty_s$ respectively. It then suffices to use the absolute continuity property of Lemma \ref{AbsContinuity}. 
	\end{proof}
	\subsection{Joint convergence of hulls}
One expects that the explored sets $T_i^L$ in the peeling by layers will correspond in the limit \eqref{ConvergenceBrownianDisk}
	to the hulls $H_r$. This section aims to give a precise result in this direction. Let us start with a technical proposition giving some information about the geometry of $\mathbb{D}_{(a)}$.
	\begin{proposition} \label{Proptechni}
	For every $\delta>0$ and $s\in(0,r_*)$, let $\mathcal U_{\delta}^s$ be the set of all points $x\in \mathbb{D}_{(a)}$ such that there is a continuous path from $x$ to $\partial \mathbb{D}_{(a)}$ that stays at distance at least $s-\delta$ from $x_*$. Almost surely, for every $s$ which is not a jump of the perimeter process $(\mathcal P_r)_{r\in(0,r_*)}$ and every $\varepsilon>0$, there exists $\delta>0$ such that:
		\begin{align*}
			\mathcal U_{\delta}^s\ \ \subset \{x\in\mathbb{D}_{(a)} : D(x, \mathbb{D}_{(a)}\setminus H_s)<\varepsilon\}.
		\end{align*}
	\end{proposition}
	\begin{proof}
		We argue by contradiction. If the statement of the proposition fails, we can find $\varepsilon>0$ and $s\in(0,r_*)$ which is not a jump of the perimeter process, and then
		a sequence $\delta_n\downarrow 0$ and points $x_n\in \mathbb{D}_{(a)}$ such that $D( x_n, \mathbb{D}_{(a)}\setminus H_s)\geq \varepsilon$ and there is a path linking $x_n$ to $\partial \mathbb{D}_{(a)}$ and remaining at distance at least $s-\delta_n$ from $x_*$. By compactness, we may assume that the sequence $(x_n)$ converges to a point $x_\infty$, which therefore satisfies $D(x_\infty, \mathbb{D}_{(a)}\setminus H_s)\geq \varepsilon$. We have $m_{x_n}\leq -r_*+s$ since $x_n\in H_s$, and, on the other hand,
			an application of the cactus bound  \cite[Proposition 3.1]{CactusBound} gives $m_{x_n}\geq -r_*+s-\delta_n$.
		Letting $n\to\infty$ we get $m_{x_\infty}= -r_*+s$. On the ancestral line of $x_\infty$, we can find a point $x$ close to $x_\infty$ whose label is strictly greater than $-r_*+s$
		and is still such that
		$m_x=-r_*+s$ (if no such $x$ existed, this would mean that $x_\infty\in\partial H_s$, contradicting $D(x_\infty, \mathbb{D}_{(a)}\setminus H_s)\geq \varepsilon$). Then all points in a sufficiently small neighbourhood of $x$ are in $H_s$ but not in $H_{s-\delta}$ for any $\delta>0$. In other words the process $(\mathcal V_r)_{r\in(0,r_*)}$ has a jump at $s$. Since the jumps of $(\mathcal V_r)$ and $(\mathcal P_r)$ almost surely coincide (this holds for $\mathcal V^\infty$ and $\mathcal P^\infty$ by \cite{Hull} and therefore also for $\mathcal V$ and $\mathcal P$ using Lemma \ref{AbsContinuity}), we end up with a contradiction.
	\end{proof}

As in Section \ref{sec:peel}, we consider the sequences of random triangulations $(T_i^L)$ and $(U_i^L)$ obtained by applying the peeling by layers algorithm to the Boltzmann triangulation $\mathcal D^L_{(a)}$. It will be convenient to view the triangulations that we consider as geodesic spaces. To this end we just need to identify 
each edge with a copy of the interval $[0,1]$ in the way explained in \cite[Remark 1.2]{GM1}. If the vertex set of $\mathcal D^L_{(a)}$ is
replaced by the union of all edges equipped with the obvious extension of the (rescaled) graph distance, the convergence (\ref{ConvergenceBrownianDisk}) remains valid,
and this has the advantage of making $\mathcal D^L_{(a)}$ a geodesic space. 

{\bf From now on, we will always view triangulations as geodesic metric spaces} as we just explained. In particular, we can consider continuous paths in $\mathcal{D}^L_{(a)}$
as in Lemma \ref{approx-path} below, and, similarly, in the next proposition, we interpret $\partial T^L_{k}$ as the union of the edges on the boundary of $T^L_{k}$.

By Skorokhod's representation theorem, we may assume that (\ref{ConvergenceBrownianDisk}) holds almost surely. 
	From now on until the end of this section, we fix $\omega\in \Omega$ for which the (almost sure) convergence (\ref{ConvergenceBrownianDisk})  does take place. 
	
	By a straightforward extension of  \cite[Proposition 1.5]{GM0}, we may assume the metric spaces $(\mathcal D^L_{(a)}, d^L)$ and $(\mathbb{D}_{(a)}, D)$ are embedded isometrically in the same compact metric space $(E, \Delta)$ 
in such a way that $\mathcal{D}^L_{(a)}$ and $\partial \mathcal{D}^L_{(a)}$ converge to $\mathbb{D}_{(a)}$ and $\partial\mathbb{D}_{(a)}$ respectively, for the Hausdorff metric $\Delta_{\mathtt H}$, $x_*^L$ converges to $x_*$
and $\nu^L$ converges weakly to $\mathbf{V}$. 
In particular, we will consider the triangulations $T_i^L$ and $U_i^L$ as subsets of $E$ so that we can speak about the $\Delta_{\mathtt {H}}$-convergence of these objects in the following proposition.

 If $\gamma:[0, \sigma]\to E$ and $\gamma':[0, \sigma']\to E$ are two continuous paths in $E$, we will say that $\gamma'$ is $\varepsilon$-close to $\gamma$ if $\Delta(\gamma(0),\gamma'(0))\leq \varepsilon$, $\Delta(\gamma(\sigma), \gamma'(\sigma'))\leq \varepsilon$ and if
	\begin{align*}
		\sup_{t\in [0, \sigma']} \Delta(\gamma'(t), \gamma)\leq \varepsilon,
	\end{align*}
where we identify  $\gamma$ and the compact subset $\gamma([0, \sigma])\subset E$. Note that this definition is not symmetric in $\gamma $ and $\gamma'$.
We also write $\ell_\Delta(\gamma)$ for the length of the path $\gamma$ in $(E,\Delta)$.

	\begin{proposition}\label{ConvHaus}
		Let $\omega$ be fixed as above and let $s\in(0,r_*)$ such that the perimeter process $(\mathcal P_r)$ is continuous at $s$. Recall the notation $h_k^L:=\Delta(x^*_L, \partial T_k^L)$.
		For every sequence of integers $(N_L)_{L\geq 1}$ such that $(\sqrt{\frac{3}{2L}}h_{N_L}^L)_{L\geq 1}$ converges to $s$, we have the convergences:
		\begin{align}
			\label{ConvHullHaus} T_{N_L}^L&\xrightarrow[L\to\infty]{\Delta_{\mathtt H}}H_s, & \partial T_{N_L}^L \xrightarrow[L\to\infty]{\Delta_{\mathtt H}} \partial H_s,&\ \ \ \ \ \ \ \ \ \ \ \ \ \ \ \  U_{N_L}^L \xrightarrow[L\to\infty]{\Delta_{\mathtt H}} \bar C_s,
		\end{align}
		where $\bar C_s$ denotes the closure of $C_s:=\mathbb{D}_{(a)}\setminus H_s$.
	\end{proposition}

	\begin{proof}\emph{To simplify notation, we set $c_L=\sqrt{3/2}L^{-1/2}$ and recall that $d_L=c_L\,\Delta^L$.}
		The convergences of $T_{N_L}^L$ and $U_{N_L}^L$  are proved in a way very similar to Lemma 12 in \cite{MarkovSpatial} (which deals with the case where $N_L$ is
		replaced by the hitting time of $\partial \mathcal D^L_{(a)}$ by the peeling algorithm). We only give here the main steps of the proof. We start with a simple lemma.
		
		\begin{lemma}
		\label{approx-path}
		For every $\eta>0$ and $A>0$, there exists $\delta>0$ and $L_0\geq 0 $ such that, for every $L\geq L_0$ and any choice of points $x,y\in \mathbb{D}_{(a)}$ and $x^L, y^L\in \mathcal D^L_{(a)}$ satisfying $\Delta(x, x^L)\leq \delta$ and $\Delta(y, y^L)\leq \delta$, we have:
		\begin{enumerate}
			\item For any continuous path $\gamma$ from $x$ to $y$ in $\mathbb{D}_{(a)}$, there exists a continuous path $\gamma^L$ from $x^L$ to $y^ L$ in
			$\mathcal{D}_{(a)}^L$ which is $\eta$-close to $\gamma$. If $\gamma$ has length at most $A$, one can choose $\gamma^L$ such that $\ell_\Delta(\gamma^{L})\leq \ell_\Delta(\gamma)+\eta$.
			
			\item For any continuous path $\gamma^{L}$ from $x^L$ to $y^L$ in $\mathcal{D}_{(a)}^L$, there is a continuous path $\gamma$ from $x$ to $y$ in $ \mathbb{D}_{(a)}$ which is $\eta$-close to $\gamma^L$. If $\gamma^{L}$ has length at most $A$, one can choose $\gamma$ such that $\ell_\Delta(\gamma)\leq \ell_\Delta(\gamma^L)+\eta$.
		\end{enumerate}
		\end{lemma}
		
		We omit the proof of this lemma (see \cite[Lemma 10]{MarkovSpatial}), and proceed to the proof of Proposition \ref{ConvHaus}.

		We first consider $U_{N_L}^L$. If $\varepsilon>0$ and $K\subset E$, we write $K^\varepsilon=\{x\in E, \ \Delta(x, K)\leq \varepsilon\}$
		(only in this proof and the next one). If $\varepsilon>0$ is fixed, we need to verify that,
		for $L$ large,
				\begin{align*}
			U_{N_L}^L\subset (C_s)^\varepsilon \ \ \ \ \text{and} \ \ \ \ \bar C_s \subset (U^L_{N_L})^\varepsilon.
		\end{align*} 
	Let $x\in \bar C_s\setminus \partial H_s=C_s$. Then there is a path $\gamma$ connecting $x$ to a point $y$ of $\partial \mathbb{D}_{(a)}$ that stays in $C_s$. By compactness, this path stays at distance 
	at least $\alpha>0$ from $\partial H_s$, hence at distance at least $s+\alpha$ from $x_*$. We can assume that $\alpha\leq\varepsilon$. By part 1 of Lemma \ref{approx-path}, and using the fact that $\partial \mathcal{D}_{(a)}^L$ converges towards $\partial \mathbb{D}_{(a)}$, we can find, for $L$  large enough, points $x^L\in \mathcal{D}_{(a)}^L$ and $y^L\in \partial \mathcal{D}_{(a)}^L$  and
	a path $\gamma^L$  in $\mathcal{D}_{(a)}^L$ from $x^L$ to $y^L$ that is $(\alpha/2)$-close to $\gamma$. Since $x^L_*$ converges to $x_*$
	 and $c_L h_{N_L}^L$ converges to $s$, we get (taking $L$ even larger if necessary) that all points of $\gamma^L$
	 lie at distance greater than $c_L(h^L_{N_L}+1)$ from $x^L_*$. However, by the construction of the peeling by layers, points of $\partial T_{N_L}^L$ are at a distance at most $c_L(h_{N_L}^L+1)$ from $x_*^L$ . Therefore we found a path connecting $x^L$ to a point of $\partial \mathcal{D}_{(a)}^L$ that does not visit $\partial T_{N_L}^L$, and it follows that $x^L$ is a point of $U_{N_L}^L$. Since $\Delta(x^L, x)\leq \alpha/2<\varepsilon$ we then have $x\in (U_{N_L}^L)^\varepsilon$ for large $L$. If $x\in \partial H_s$, this is also true because we can approximate $x$ by a point of $C_s$. A compactness argument finally allows us to conclude that $\bar C_s\subset (U_{N_L}^L)^\varepsilon$ for any $L$ large enough.
		
		Let us show conversely that $U_{N_L}^L\subset (C_s)^\varepsilon$ when $L$ is large. We choose $\delta\in(0,\varepsilon)$
		such that the conclusion of Proposition \ref{Proptechni} holds with $\varepsilon$ replaced by $\varepsilon/2$. Let $v^L\in U_{N_L}^L$, which implies in particular
		that $v^L$ is at $\Delta$-distance at least $c_L\,h_{N_L}^L$ from $x_*^L$. Then there is a path $\gamma^L$ in $U_{N_L}^L$ connecting $v^L$ to $\partial \mathcal{D}_{(a)}^L$. Using part 2 of Lemma \ref{approx-path} and the convergence of $\partial\mathcal{D}_{(a)}^L$ to $\partial \mathbb{D}_{(a)}$, if $L$ is large enough (independently of the choice of $v^L$), we can approximate $\gamma^L$ by a path $\gamma$ in $\mathbb{D}_{(a)}$ that 
		is $(\delta/2)$-close to $\gamma^L$ and connects a point $v\in \mathbb{D}_{(a)}$ to a point of $\partial \mathbb{D}_{(a)}$. Notice that $\Delta(v,v^L)\leq \delta/2<\varepsilon/2$. 
		Provided that $L$ has been chosen even larger if necessary (again independently of the choice of $v^L$), it follows that the path $\gamma$ contains only points at distance at least $s-\delta$ from $x_*$. By our choice of $\delta$,
		this implies that $\Delta(v,C_s)<\varepsilon/2$ and thus $\Delta(v^L,C_s)<\varepsilon$. We therefore have $v^L\in (C_s)^\varepsilon$ and we have obtained that $U^L_{N_L}\subset ( C_s)^\varepsilon$, thus completing the proof of the convergence of $U_{N_L}^L$ to $\bar C_s$. 
		
Let us now discuss the convergence of $\partial T_{N_L}^L$. We let $\mathcal B^L(c_Lh_{N_L}^L)$ and $\mathcal B^L(c_L(h_{N_L}^L+ 2))$ denote the closed balls of respective radii $c_Lh_{N_L}^L$ and $c_L(h_{N_L}^L+2)$ centered at $x_*^L$ in $(\mathcal{D}_{(a)}^ L, \Delta)$. We also write $B_s=B_s(x_*)$ for the closed ball of radius $s$ centered at $x_*$ in $(\mathbb{D}_{(a)}, \Delta)$. Since 
$\mathcal{D}^{L}_{(a)}$ and $\D_{(a)}$ are both length spaces and $\mathcal{D}^{L}_{(a)}$ converges to $\D_{(a)}$ for the Hausdorff distance on $E$, we get that $\mathcal B^L(c_Lh_{N_L}^L)$ and $\mathcal B^L(c_L(h_{N_L}^L+2))$ both converge to $B_s$ for the Hausdorff distance. 
However,
		\begin{align*}
			\mathcal B^L(c_L h_{N_L}^L)\subset \mathcal B^L(c_L h_{N_L}^L)\cup \partial T^L_{N_L} \subset \mathcal B^L(c_L(h_{N_L}^L+2)).
		\end{align*}
		It follows that $\mathcal B'_L:=\mathcal B^L(c_L h_{N_L}^L)\cup \partial T^L_{N_L}$ also converges towards $B_s$ when $L\to\infty$.
	Observe that $\partial T_{N_L}^L=\mathcal B'_L\cap U_{N_L}^L$ and $\partial H_s=  B_s\cap \bar C_s$. Let $\varepsilon>0$. Using the convergence of $\mathcal B'_L$ towards $B_s$ and the convergence of $U_{N_L}^L$ towards $\bar C_s$, we get that for  $L$ sufficiently large and for every $x\in \partial H_s$, we have $\Delta(x, \mathcal B'_L)< \varepsilon$ and $\Delta(x,U_{N_L}^L)< \varepsilon$.
Fix $x\in\partial H_s$ and let $u_1\in U_{N_L}^L$ and $u_2\in \mathcal B'_L$ such that $\Delta(u_1, x)\leq \varepsilon$ and $\Delta(u_2, x)\leq \varepsilon$. In particular, $\Delta(u_1, u_2)\leq 2\varepsilon$ and since a geodesic path between $u_1$ and $u_2$ in $\mathcal{D}_{(a)}^L$ must intersect $\partial T_{N_L}^L$, it follows that one can find $v\in \partial T_{N_L}^L$ with $\Delta(u_1, v)\leq 2\varepsilon$. This implies $\Delta(x, v)\leq 3\varepsilon$, but since this is true for any $x\in\partial H_s$, we conclude that $\partial H_s$ is contained in the $3\varepsilon$-neighbourhood of $\partial T_{N_L}^L$ as soon as $L$ is large enough. A similar argument shows that $\partial T_{N_L}^L$ is contained in the $3\ve$-neighbourhood of $\partial H_s$ when $L$ is large enough. This proves the convergence of $\partial T_{N_L}^L$ towards $\partial H_s$.

 Once we have obtained the convergence of $U^L_{N_L}$ to $\bar C_s$ and the
 convergence of $\partial T^L_{N_L}$ to $\partial H_s$, the convergence of $T_{N_L}^L$ towards $H_s$ follows from straightforward arguments, and
we leave the details to the reader.
		\end{proof}
		For every integer $k\geq 1$, we set $\sigma_k^L:=\inf\{ n\in \mathbb N : h_n^L\geq k\}$. On the event $\{\sigma^L_k<\infty\}$,
	the (discrete) hull of radius $k$ in $\mathcal{D}_{(a)}^L$ is defined by $\mathcal H^L_k:=T_{\sigma_k^L}^L\,$.
	
	Recall that $\omega$ is fixed as explained before Proposition \ref{ConvHaus}.
	
	\begin{corollaire}
		\label{coroHull} Let $s\in(0,r_*)$ such that the perimeter process $(\mathcal P_r)$ has no jump at $s$. Then the hull $\mathcal H_{\lfloor s/c_L\rfloor }^L$ converges towards $H_s$ 
		for the Hausdorff metric, and its volume $\nu^L(\mathcal H_{\lfloor s/c_L\rfloor }^L)$ converges towards $\mathcal V_s$.
	\end{corollaire}
	\begin{proof}
		The convergence of $\mathcal H_{\lfloor s/c_L\rfloor }^L$ towards $H_s$ is an immediate corollary of the previous proposition, since by construction
		$$c_L\,h^L_{\sigma^L_{\lfloor s/c_L\rfloor}} \longrightarrow s$$
		as $L\to\infty$.
		It remains to show that $\nu^L(\mathcal H_{\lfloor s/c_L\rfloor }^L)$ converges to $\mathcal V_s$.
		We keep the notation $K^\varepsilon=\{x\in E, \ \Delta(x, K)\leq \varepsilon\}$ introduced in the previous proof.  It is easy to verify that $\mathbf V(\partial H_s)=0$. Then, if $\varepsilon>0$ is fixed, we can find $\delta>0$ such that $\mathbf V((\partial H_s)^\delta)<\varepsilon$. 
			Since $\nu_L$ converges weakly to $\mathbf V$, we get, for $L$ large enough,
		\begin{align*}
			\nu_L(\mathcal H^L_{\lfloor s/c_L\rfloor})\leq \nu_L((H_s)^{\delta/2})\leq \mathbf V((H_s)^\delta)+\varepsilon\leq \mathbf V(H_s)+\mathbf V((\partial H_s)^\delta)+\varepsilon\leq \mathbf V(H_s)+2\varepsilon.
		\end{align*}
		On the other hand, $\partial \mathcal H_{\lfloor s/c_L\rfloor}^L\to \partial H_s$ when $L\to\infty$ (by Proposition \ref{ConvHaus}), so that 
		we have also $\partial \mathcal H_{\lfloor s/c_L\rfloor}^L\subset (\partial H_s)^{\delta/2}$ for every large enough $L$. 
	It follows that, for large enough $L$, we have 
		$\nu_L((\partial\mathcal  H_{\lfloor s/c_L \rfloor }^L)^{\delta/2}) \leq \mathbf V((\partial H_s)^\delta)+\varepsilon\leq 2\varepsilon$.
		Hence we get for $L$ large,
		\begin{align*}
			\mathbf V(H_s)\leq \nu_L((\mathcal H^L_{\lfloor s/c_L\rfloor})^{\delta/2})+\varepsilon\leq  \nu_L(\mathcal H^L_{\lfloor s/c_L\rfloor})+\nu_L((\partial \mathcal H^L_{\lfloor s/c_L\rfloor})^{\delta/2})+\varepsilon \leq \nu_L(\mathcal H^L_{\lfloor s/c_L\rfloor})+3\varepsilon.
		\end{align*}
		The desired convergence of $\nu_L(\mathcal H^L_{\lfloor s/c_L\rfloor})$ towards $\mathbf V(H_s)=\mathcal V_s$ follows from the last two displays.
	\end{proof}

\section{Limit theorems for the perimeter and the volume of the explored region}
\label{sec:peeling-asymp}
	
	In this section, we take $a=1$ for simplicity, and (as in Section \ref{sec:peel}) we write $\mathcal{D}^L$ instead of $\mathcal{D}^L_{(1)}$
	for a Boltzmann triangulation in $\mathbb{T}^{1,\bullet}(L)$.
	Recall that $(T_i^L)_{i\geq 0}$ is the sequence of (explored) triangulations we get when we apply the peeling by layers algorithm of Section \ref{sec:peel} to $\mathcal D^L$. We also set
	$S_L:=\inf\{i\geq 0 : T_i^L=\dagger\}$, which corresponds to the hitting time of $\partial \mathcal D^L$. To simplify notation, we let $P^L_k=|\partial T^L_k|$
	be the boundary size of $T^L_k$, for every $0\leq k<S_L$. Still for $0\leq k<S_L$, we also write $V^L_k$ for the number of vertices of $T^L_k$, and we recall that $h^L_k$
is the graph distance from the distinguished vertex $x^L_*$ to the boundary $\partial T^L_k$. Properties of the peeling by layers ensure
that the graph distance from $x^L_*$ to any point of $\partial T^L_k$ is equal to $h^L_k$ or $h^L_k+1$.

	Let $(T_i^\infty)_{i\geq 0}$ be the sequence of triangulations with a boundary obtained by applying the same peeling algorithm to the UIPT (we refer to \cite{ScalingUIPT} for a discussion of the peeling by layers algorithm for the UIPT). We define $P^\infty_k$, $V^\infty_k$ and $h^\infty_k$, now for every integer $k\geq 0$, by replacing $T^L_k$ with $T^\infty_k$
	in the respective definitions of $P^L_k$, $V^L_k$ and $h^L_k$.
	
	Finally, we set $\wh S_L=L^{-3/2}(S_L-1)$ if $S_L>0$, and by convention we also take $\wh S_L=0$ when $S_L=0$. Recall the notation $c_L=\sqrt{3/2}\,L^{-1/2}$. We introduce the rescaled processes
	$$\wh P^L_t=\frac{1}{L} P^L_{\lfloor L^{3/2}t\rfloor}, \quad \wh V^L_t=\frac{3}{4L^2} V^L_{\lfloor L^{3/2}t\rfloor},\quad \wh h^L_t=c_L\, h^L_{\lfloor L^{3/2}t\rfloor}.$$
	for $0\leq t\leq \wh S_L$ (by convention, $\wh P^L_0=\wh V^L_0=\wh h^L_0=0$ when $S_L=0$). We similarly define, for every $t\geq 0$,
	$$\wh P^{\infty,L}_t=\frac{1}{L} P^\infty_{\lfloor L^{3/2}t\rfloor},\quad \wh V^{\infty,L}_t=\frac{3}{4L^2} V^\infty_{\lfloor L^{3/2}t\rfloor},\quad \wh h^{\infty,L}_t=c_L\, h^\infty_{\lfloor L^{3/2}t\rfloor}.$$
	
	From \cite[Proposition 10]{ScalingUIPT} (more precisely, from the version of this result for type I triangulations, as explained in Section 6.1 of \cite{ScalingUIPT}), we have
	\begin{equation}
	\label{conv-peri}
	\Big(\wh P^{\infty,L}_t,\wh V^{\infty,L}_t,\wh h^{\infty,L}_t\Big)_{ t\geq 0}\xrightarrow[L\to\infty]{(d)} \Big({\mathcal{S}}^+_t, {\mathbb{V}}_t, 2^{-3/2}\int_0^t\frac{\dd u}{{\mathcal{S}}^+_u}\Big)_{t\geq 0},
	\end{equation}
	where the convergence holds in distribution in the sense of the Skorokhod topology. Here the limiting process $({\mathcal{S}}^+_t,t\geq 0)$
	is a stable L\'evy process with no positive jumps and  Laplace exponent $\tilde\psi(\lambda)=3^{-1/2}\lambda^{3/2}$ started at $0$ and conditioned to stay positive
	(see \cite[Section VII.3]{Bertoin} for the definition of this process), and we refer to \cite{ScalingUIPT} for the description of the conditional law of the process ${\mathbb{V}}$
	knowing ${\mathcal{S}}^+$. 
	
	The next proposition gives an analog of \eqref{conv-peri} where $\wh P^{\infty,L},\wh V^{\infty,L}$, and $ \wh h^{\infty,L}$ are replaced by $\wh P^{L},\wh V^{L},$
	and $\wh h^{L}$ respectively.
	
	\begin{proposition}
	\label{key-conv}
	We have 
\begin{equation}
	\label{key-conv1}\Big(\Big(\wh P^{L}_{t\wedge \wh S_L},\wh V^{L}_{t\wedge \wh S_L},\wh h^{L}_{t\wedge \wh S_L}\Big)_{t\geq 0},\wh S_L\Big)\xrightarrow[L\to\infty]{(d)} 
\Big(\Big(\mathscr{P}_t, \mathscr{V}_t,\mathscr{A}_t\Big)_{t\geq 0}, \Sigma_\infty\Big)_{t\geq 0},
\end{equation}
where 
$$\mathscr{A}_t=2^{-3/2}\int_0^{t\wedge\Sigma_\infty}\frac{\dd u}{\mathscr{P}_u}$$
and the distribution of $((\mathscr{P}_t, \mathscr{V}_t)_{t\geq 0},\Sigma_\infty)$ is determined by 
$$\E\Big[ G\Big((\mathscr{P}_t, \mathscr{V}_t)_{t\geq0}\Big)\,f(\Sigma_\infty)\Big]
= \frac{\sqrt{3\pi}}{4}\,\int_0^\infty \dd u\,f(u)\,\E\Big[ G(({\mathcal{S}}^+_{t\wedge u})_{t\geq 0},(\mathbb{V}_{t\wedge u})_{t\geq0})\,\frac{1}{\sqrt{{\mathcal{S}}^+_u}}\,(1+{\mathcal{S}}^+_u)^{-5/2}\Big]$$
for any measurable functions $f:\R_+\longrightarrow\R_+$, and  $G:\D(\R_+,\R_+^2)\longrightarrow \R_+$. 
\end{proposition}
	
	\proof We first derive the convergence in distribution of $(\wh P^{L}_{t\wedge \wh S_L})_{t\geq 0}$ to $(\mathscr{P}_t)_{t\geq 0}$. 
	The $h$-transform relation between the Markov chains $P^L$ and $P^\infty$, which was discussed in Section \ref{sec:peel}, shows that, for every integer $k\geq 0$
	and every bounded function $F$ on $\N^{k+1}$,
	$$\E[F(P^L_0,\ldots,P^L_k)\,\mathbf{1}_{\{k<S_L\}}\mid S_L>0]= \E[F(P^\infty_0,\ldots,P^\infty_k)\,\frac{\mathbf{h}_L(P^\infty_k)}{\mathbf{h}_L(P^\infty_0)}],$$
	where we recall that $\mathbf{h}_L(j)=\frac{L}{L+j}$, and we note that $P^\infty_0=1$ if the root edge of the UIPT is a loop, and $P^\infty_0=2$ otherwise. 
	
	By the Markov property, we have $\P(S_L=k+1\mid S_L>k,P_0,P_1,\ldots,P_k)=q_L(P_k,\dagger)$. It then follows that
	\begin{equation}
	\label{h-trans}
	\E[F(P^L_0,\ldots,P^L_k)\,\mathbf{1}_{\{S_L=k+1\}}\mid S_L>0]= \E[F(P^\infty_0,\ldots,P^\infty_k)\,\frac{\mathbf{h}_L(P^\infty_k)}{\mathbf{h}_L(P^\infty_0)}\,q_L(P^\infty_k,\dagger)].
	\end{equation}

	Let $G$ be a bounded continous function on $\D(\R_+,\R_+)\times \R_+$, such that $0\leq G\leq 1$. Using \eqref{h-trans}, we have
	\begin{align*}
	&\E[G((\wh P^L_{t\wedge \wh S_L})_{t\geq 0},\wh S_L)\mid S_L>0]\\
	&\quad=\sum_{k=0}^\infty \E[\mathbf{1}_{\{S_L=k+1\}}\,G((\wh P^L_{t\wedge (L^{-3/2}k)})_{t\geq 0},L^{-3/2}k)\mid S_L>0]\\
	&\quad = \sum_{k=0}^\infty \E\Big[G((\wh P^\infty_{t\wedge (L^{-3/2}k)})_{t\geq 0},L^{-3/2}k)\frac{\mathbf{h}_L(P^\infty_k)}{\mathbf{h}_L(P^\infty_0)}\,q_L(P^\infty_k,\dagger)\Big]\\
	&\quad= L^{3/2} \int_0^\infty \dd u\,\E\Big[G((\wh P^{\infty,L}_{t\wedge (L^{-3/2}\lfloor L^{3/2} u\rfloor)})_{t\geq 0},L^{-3/2}\lfloor L^{3/2} u\rfloor)\,\frac{\mathbf{h}_L(P^\infty_{\lfloor L^{3/2} u\rfloor})}{\mathbf{h}_L(P^\infty_0)}\,q_L(P^\infty_{\lfloor L^{3/2} u\rfloor},\dagger)\Big]
	\end{align*}
	Note that $\P(S_L>0)$ tends to $1$ as $L\to\infty$. Furthermore,
	$$\mathbf{h}_L(P^\infty_{\lfloor L^{3/2} u\rfloor})=\frac{1}{1+\wh P^{\infty,L}_u}$$
	and we also know from \cite{MarkovSpatial} that
	$$L^{3/2}q_L(P^\infty_{\lfloor L^{3/2} u\rfloor},\dagger) \sim \frac{\sqrt{3\pi}}{4}\,\frac{1}{\sqrt{\wh P^{\infty,L}_u}}\,(1+\wh P^{\infty,L}_u)^{-3/2}$$
	when $L$ and $P^\infty_{\lfloor L^{3/2} u\rfloor}$ are large (see the last display before Section 3.2 in \cite{MarkovSpatial}). 
	
	Using the convergence \eqref{conv-peri} (which we may assume to hold a.s. by the Skorokhod representation theorem) and the preceding observations, we get from
	an application of Fatou's lemma that
	\begin{equation}
	\label{conv-peri1}
	\liminf_{L\to\infty} \E[G((\wh P^L_{t\wedge \wh S_L})_{t\geq 0},\wh S_L)]
	\geq \int_0^\infty \dd u\,\E\Big[ G(({\mathcal{S}}^+_{t\wedge u})_{t\geq 0},u)\,\frac{\sqrt{3\pi}}{4}\,\frac{1}{\sqrt{{\mathcal{S}}^+_u}}\,(1+{\mathcal{S}}^+_u)^{-5/2}\Big].
	\end{equation}
	At this stage, we observe that
	\begin{equation}
	\label{conv-peri2}
	\int_0^\infty \dd u\,\E\Big[\frac{\sqrt{3\pi}}{4}\,\frac{1}{\sqrt{{\mathcal{S}}^+_u}}\,(1+{\mathcal{S}}^+_u)^{-5/2}\Big]=1.
	\end{equation}
	Indeed, by the identification of the potential kernel of $\mathcal{S}^+$ in \cite[Section VII.4]{Bertoin}, we know that, for any measurable function $f:\R_+\longrightarrow \R_+$,
	$$\E\Big[\int_0^\infty \dd u\,f({\mathcal{S}}^+_u)\Big]=\int_0^\infty \dd x\,\tilde W(x)\,f(x)$$
	where the function $\tilde W$ is determined by its Laplace transform
	$$\int_0^\infty e^{-\lambda x} \tilde W(x)\,\dd x=\frac{1}{\tilde\psi(\lambda)}= 3^{1/2}\,\lambda^{-3/2}.$$
	It follows that $\tilde W(x)=\frac{2\sqrt{3}}{\sqrt{\pi}}\,\sqrt{x}$, and the left-hand side of \eqref{conv-peri2} is equal to
	$$\frac{\sqrt{3\pi}}{4}\times \frac{2\sqrt{3}}{\sqrt{\pi}} \int_0^\infty \dd x\,(1+x)^{-5/2}= 1$$
	as desired. Thanks to \eqref{conv-peri2}, we can replace $G$ by $1-G$ in \eqref{conv-peri1} to get the analog 
	of \eqref{conv-peri1} for the limsup instead of the liminf, and we conclude that
	$$\lim_{L\to\infty} \E[G((\wh P^L_{t\wedge \wh S_L})_{t\geq 0},\wh S_L)]
	= \frac{\sqrt{3\pi}}{4}\,\int_0^\infty \dd u\,\E\Big[ G(({\mathcal{S}}^+_{t\wedge u})_{t\geq 0},u)\,\frac{1}{\sqrt{{\mathcal{S}}^+_u}}\,(1+{\mathcal{S}}^+_u)^{-5/2}\Big].$$
	This gives the convergence of $((\wh P^L_{t\wedge \wh S_L})_{t\geq 0},\wh S_L)$ to the pair $((\mathscr{P}_t)_{t\geq 0},\Sigma_\infty)$
	introduced in the proposition. We can deduce the more general statement of the proposition from the convergence \eqref{conv-peri} by exactly the
	same arguments. The point is the fact that the perimeter process $(P^L_k)_{k\geq 0}$ is
	Markov with respect to the
	discrete filtration generated by the sequence $(T^L_k)$. We leave the details to the reader. \endproof
	
	Let $R_L:=h^L_{S_L-1}$ (we argue on the event where $S_L>0$). By previous observations, the graph distance 
	between $x^L_*$ and $\partial \mathcal{D}^L$ is either $R_L$ or $R_L+1$. Recall the notation 
	$$\sigma^L_k=\inf\{n\in\N:h^L_n\geq k\},$$
	so that $\sigma^L_k$ is finite
	for $1\leq k\leq R_L$. For $1\leq k\leq R_L$, we write $\mathcal{P}^L_k:=P^L_{\sigma^L_k}$ and $\mathcal{V}^L_k:=V^L_{\sigma^L_k}$, which are
	respectively the perimeter
	and the volume of the discrete hull $\mathcal{H}^L_k=T^L_{\sigma^L_k}$.
	
We also set $r^L_*=c_LR_L$, which essentially corresponds to the rescaled graph distance between $x^L_*$ and $\partial \mathcal{D}^L$.
Finally, we introduce rescaled versions of the processes $\mathcal{P}^L_k$ and $\mathcal{V}^L_k$ by setting
$$\wh{\mathcal{P}}^L_t:=\frac{1}{L} \mathcal{P}^L_{\lfloor t/c_L\rfloor}\hbox{ \ and \ } \wh{\mathcal{V}}^L_t:=\frac{3}{4L^2} \mathcal{V}^L_{\lfloor t/c_L\rfloor}$$
for $0\leq t\leq r^L_*$.
	
	\begin{corollaire}
	\label{conv-time-changed}
	Recall the processes $\mathscr{P}_t, \mathscr{V}_t,\mathscr{A}_t$ in Proposition \ref{key-conv}. 
	We have 
	\begin{equation}
	\label{key-conv2}
	\Big( \Big(\wh{\mathcal{P}}^L_{t\wedge r^L_*}, \wh{\mathcal{V}}^L_{t\wedge r^L_*}, L^{-3/2}\sigma^L_{\lfloor t/c_L\rfloor\wedge R_L}\Big)_{t\geq 0}, r^L_*\Big)
	\xrightarrow[L\to\infty]{(d)} \Big( (\wh{\mathcal{P}}^\infty_{t}, \wh{\mathcal{V}}^\infty_{t},\eta_t)_{t\geq 0}, r^\infty_*\Big)
	\end{equation}
	where $r^\infty_*=\mathscr{A}_\infty$ and, for every $t\geq0$,
	$$\wh{\mathcal{P}}^\infty_{t}=\mathscr{P}_{\eta_t},\quad \wh{\mathcal{V}}^\infty_{t}=\mathscr{V}_{\eta_t}$$
	with
	$$\eta_t=\inf\{s\geq 0: \mathscr{A}_s\geq t\wedge \mathscr{A}_\infty\}.$$
         Moreover the convergence in distribution \eqref{key-conv2} holds jointly with \eqref{key-conv1}.
	\end{corollaire}
	
	\proof Since $c_LR_L=c_L h^L_{S_L-1}=\wh h^L_{\wh S_L}$, 
Proposition \ref{key-conv} implies the convergence in distribution of $r_*^L=c_LR_L$ towards the variable $\mathscr{A}_\infty$, and this convergence holds jointly with the one stated in Proposition \ref{key-conv}. Then, for $0\leq t\leq c_LR_L$,
$$L^{-3/2}\sigma^L_{\lfloor t/c_L\rfloor\wedge R_L}=L^{-3/2}\min\{j:h^L_j\geq \lfloor t/c_L\rfloor\}=\inf\{s\geq 0:\wh h^L_s\geq c_L\lfloor t/c_L\rfloor\}$$
Since we know from Proposition \ref{key-conv} that $(\wh h^L_{t\wedge \wh S_L})_{t\geq 0}$ converges in distribution to $(\mathscr{A}_t)_{t\geq 0}$, it is now easy to obtain that
$(L^{-3/2}\sigma^L_{\lfloor t/c_L\rfloor\wedge R_L})_{t\geq 0}$ converges in distribution  to $(\eta_t)_{t\geq 0}$,
and this convergence holds jointly with that of Proposition \ref{key-conv} (very similar arguments are
used in Section 4.4 of \cite{ScalingUIPT}). 
Then,
by our definitions,
		$$ (\wh{\mathcal{P}}^L_{t\wedge r^L_*}, \wh{\mathcal{V}}^L_{t\wedge r^L_*})_{t\geq 0}
		=\Big(\wh P^L_{L^{-3/2}\sigma^L_{\lfloor t/c_L\rfloor\wedge R_L}},\wh V^L_{L^{-3/2}\sigma^L_{\lfloor t/c_L\rfloor\wedge R_L}}\Big)_{t\geq 0}
	$$
	and we just have to use \eqref{key-conv1} together with the convergence of $(L^{-3/2}\sigma^L_{\lfloor t/c_L\rfloor\wedge R_L})_{t\geq 0}$ towards $(\eta_t)_{t\geq 0}$
	to get the desired result. 
	
		 \endproof
	
	Recall the processes $(\mathcal{P}^\infty_t,\mathcal{V}^\infty_t)_{t\geq 0}$ giving the perimeters and volumes
	of hulls in the Brownian plane. 
	Then, for every $r>0$, the distribution of the pair $(\wh{\mathcal{P}}^\infty_{t}, \wh{\mathcal{V}}^\infty_{t})_{0\leq t\leq r}$ in Corollary \ref{conv-time-changed} 
	under $\P(\cdot\mid r^\infty_*>r)$
	is absolutely continuous with respect to the distribution of $(\mathcal{P}^\infty_t,\mathcal{V}^\infty_t)_{0\leq t\leq r}$.
	This follows by observing that  $(\mathcal{P}^\infty_t,\mathcal{V}^\infty_t)_{t\geq 0}$ is obtained from 
	the pair $(S^+_t,\mathbb{V}_t)$ in \eqref{conv-peri} by the same time change as the one giving 
	$(\wh{\mathcal{P}}^\infty_{t}, \wh{\mathcal{V}}^\infty_{t})_{t\geq 0}$ from $(\mathscr{P}_t,\mathscr{V}_t)_{t\geq 0}$
(combine formula (56) in \cite{ScalingUIPT} with the description 
	of the pair $(\mathcal{P}^\infty_t,\mathcal{V}^\infty_t)_{t\geq 0}$ in \cite[Theorem 1.3]{Hull} --- some care is needed here 
	because the scaling constants in \cite{ScalingUIPT} are not the same as in the present work). 
	
	The preceding absolute continuity property implies that the approximation \eqref{PrFromJumps}
	holds when $\mathcal{P}_r$ and $\mathcal{V}_s$ are replaced by $\wh{\mathcal{P}}^\infty_r$ and $ \wh{\mathcal{V}}^\infty_{s}$ respectively, a.s.
	for every $r<r^\infty_*$. In other words, we can recover $\wh{\mathcal{P}}^\infty_r$ as a deterministic function of $( \wh{\mathcal{V}}^\infty_{s})_{s\in[0,r]}$
	which is the same as the one giving $\mathcal{P}_r$ from  $(\mathcal{V}_s)_{s\in [0,r]}$. 

\begin{theorem}
\label{conv-peri-volume}
We have 
$$\Big(\mathcal{D}^L, (\wh{\mathcal{P}}^L_{t\wedge r^L_*}, \wh{\mathcal{V}}^L_{t\wedge r^L_*})_{t\geq 0}, r^L_*\Big)
\xrightarrow[L\to\infty]{(d)} \Big(\D_{(1)}, (\mathcal{P}_{t\wedge r_*},\mathcal{V}_{t\wedge r_*})_{t\geq 0},r_*\Big),$$
where the convergence holds in distribution in $\mathbb{M}^{2,1}\times\D(\R_+,\R_+^2)\times \R_+$.
Moreover, this convergence holds jointly with \eqref{key-conv1} and \eqref{key-conv2}.
\end{theorem} 

\proof By a tightness argument using Corollary \ref{conv-time-changed}, we may assume that, along a sequence of values of $L$, the triplet
$$\Big(\mathcal{D}^L, (\wh{\mathcal{P}}^L_{t\wedge r^L_*}, \wh{\mathcal{V}}^L_{t\wedge r^L_*})_{t\geq 0}, r^L_*\Big)$$
converges in distribution to a limit which we may denote as
$$\Big(\D_{(1)}, (\wh{\mathcal{P}}^\infty_{t}, \wh{\mathcal{V}}^\infty_{t})_{t\geq 0},r^\infty_*\Big).$$
By the Skorokhod representation theorem, we may assume that this convergence holds a.s.
Since $r^L_*$ is the rescaled graph distance between $x^L_*$ and $\partial \mathcal{D}^L$ (up to an error which is $O(L^{-1/2})$),
it immediate that $r^\infty_*=r_*$. On the other hand, for $t<r^L_*$, we have $\wh{\mathcal{V}}^L_{t}= \frac{3}{4L^2} \mathcal{V}^L_{\lfloor t/c_L\rfloor}=
\nu^L(\mathcal{H}^L_{\lfloor s/c_L\rfloor})$, and Corollary \ref{coroHull} then allows us to identify 
$(\wh{\mathcal{V}}^\infty_{t})_{t\geq 0}$ with $(\mathcal{V}_{t\wedge r_*})_{t\geq 0}$. Finally, we saw that, for $r<r^\infty_*=r_*$,
$\wh{\mathcal{P}}^\infty_r$ must be given by the same deterministic function of $( \wh{\mathcal{V}}^\infty_{s})_{s\in[0,r]}$
 as the one giving $\mathcal{P}_r$ from  $(\mathcal{V}_s)_{s\in [0,r]}$, and we conclude that we have also
 $(\wh{\mathcal{P}}^\infty_{t})_{t\geq 0}=(\mathcal{P}_{t\wedge r_*})_{t\geq 0}$, which completes the proof. \endproof 
 
 We now fix $b>0$ and recall the notation $r_b=\inf\{r\in[0,r_*): \mathcal{P}_r=b\}$. For every $L\geq 1$, we also set
  $$k^L_b=\inf\{k\in\{1,\ldots, S_L-1\}: P^L_k=\lfloor bL\rfloor\},$$
  and $r^L_b=c_Lh^L_{k^L_b}$ on the event where $k^L_b<\infty$. 
 In other words, $r^L_b$ corresponds to  the (rescaled) distance between the distinguished vertex and the boundary of the
 first explored region with perimeter $\lfloor bL\rfloor$. If $k^L_b=\infty$, we take $r^L_b=\infty$.
 
 We let $\D_{(1)}^{(b)}$ be distributed as $\D_{(1)}$ conditioned on the event $\{r_b<\infty\}$
and similarly, for every $L\geq 1$, we let $\mathcal{D}^{L,(b)}$ be distributed as
$\mathcal{D}^{L}$ conditioned on $\{k^L_b<\infty\}$.

	\begin{proposition}
		\label{ConvBrownianDiskWithRadius}
		The convergence in distribution 
		\begin{align}
			\label{ConvBrownianDiskConditioned}
			(\Dd, r_b^L) \xrightarrow[L\to\infty]{(d)} (\Dc_{(1)} ,r_b).
		\end{align}
		holds in $\mathbb{M}^{2, 1}\times \mathbb R$.
	\end{proposition}
	
	\proof By the Skorokhod representation theorem, we may assume that the convergence of Theorem \ref{conv-peri-volume}
holds almost surely, as well as the convergences \eqref{key-conv1} and \eqref{key-conv2}. Proposition \ref{ConvBrownianDiskWithRadius}
will follow if we can verify that $r^L_b \longrightarrow r_b$ a.s. as $L\to\infty$ (in particular $\mathbf{1}_{\{r^L_b<\infty\}}  \longrightarrow \mathbf{1}_{\{r_b<\infty\}}$).
Set
$$\xi^L_b:=\inf\{j\in\{0,1,\ldots,R_L\}: \mathcal{P}^L_j\geq \lfloor bL\rfloor\}.$$
From the (a.s.) convergence of $(\wh{\mathcal{P}}^L_{t\wedge r^L_*})_{t\geq 0}$ to $(\mathcal{P}_{t\wedge r_*})_{t\geq 0}$, one
infers that $c_L\xi^L_b$ converges a.s. to $r_b$ as $L\to\infty$, on the event $\{r_b<\infty\}$. To be precise, we need to know that immediately 
after time $r_b$, the process $\mathcal{P}_t$ takes values greater than $b$, but this follows (via a time change argument) from the analogous property 
satisfied by the process $\mathscr{P}_t$ in Proposition \ref{key-conv}. 

Argue on the event $\{r_b<\infty\}$. Then, for $L$ large we have $\xi^L_b<\infty$ and $\sigma^L_{\xi^L_b}\geq k^L_b$ (because $P^L_{\sigma^L_{\xi^L_b}}=\mathcal{P}^L_{\xi^L_b}\geq \lfloor bL\rfloor\}$). Hence,
$$c_L\xi^L_b=c_Lh^L_{\sigma^L_{\xi^L_b}}\geq c_Lh^L_{k^L_b}=r^L_b$$
and, since $c_L\xi^L_b$ converges to $r_b$,
$$\limsup_{L\to\infty} r^L_b \leq r_b. $$

To get the analogous result for the liminf, fix $\ve\in(0,b)$ and argue on the event
where $r_{b-\ve}<\infty$. Since $\mathcal{P}_r=\mathscr{P}_{\eta_r}$ for $0<r<r_*$, we have
$$\sup_{s\leq \eta_{r_{b-\ve}}}\mathscr{P}_s=\sup_{t\leq r_{b-\ve}}\mathscr{P}_{\eta_t}=\sup_{r\leq r_{b-\ve}}\mathcal{P}_r\leq b-\ve.$$
Using the (a.s.) convergence \eqref{key-conv1}, we thus get that for $L$ large,
$$\sup_{s\leq \eta_{r_{b-\ve}}}\wh P^L_{s\wedge \wh S_L}<b-\frac{\ve}{2},$$
or equivalently
$$\frac{1}{L} \sup_{j\leq L^{3/2} \eta_{r_{b-\ve}}} P^L_{j\wedge S_L}<b-\frac{\ve}{2},$$
which implies $k^L_b\geq L^{3/2} \eta_{r_{b-\ve}}$. Finally,
$$r^L_b=c_L h^L_{k^L_b}\geq c_L\,h^L_{\lfloor L^{3/2}\eta_{r_{b-\ve}}\rfloor}$$
and the right-hand side converges as $L\to\infty$ to $\mathscr{A}_{\eta_{r_{b-\ve}}}=r_{b-\ve}$. We conclude that
$$\liminf_{L\to\infty} r^L_b \geq r_{b-\ve}$$
on the event where $r_{b-\ve}<\infty$. Since this holds for any $\ve>0$, the proof is complete. \endproof

\section{Convergence to the Brownian annulus}
\label{sec:conv-annulus}
	
	\subsection{Statement of the result}
	\label{sta-res}
We no longer assume that  $a=1$. The definitions of $\Dd$ and $\Dc_{(1)}$ given before Proposition \ref{ConvBrownianDiskWithRadius}
can then be extended. In particular, we write $\Dd_{(a)}$ for a Boltzmann triangulation in $\mathbb{T}^{1,\bullet}(\lfloor aL\rfloor)$
conditioned on the event $\{k^L_b<\infty\}$, where $k^L_b$ is the first time at which the perimeter of the
explored region in the peeling algorithm is equal to $\lfloor bL\rfloor$. We keep the notation $d_L$ for the (rescaled) distance on $\Dd_{(a)}$
and $r^b_L$ for the $d_L$-distance between the 
distinguished vertex and the boundary of the explored region at time $k^L_b$. Similarly, $\D^{(b)}_{(a)}$ is distributed as $\D_{(a)}$ conditioned on the event
that the process of hull perimeters hits $b$, and $r_b$ is the corresponding hitting radius. We keep the notation $D$ for the distance on $\D^{(b)}_{(a)}$.  The convergence 
\eqref{ConvBrownianDiskConditioned} is then immediately extended to give
\begin{align}
			\label{ConvBrownianDiskConditioned2}
			(\Dd_{(a)}, r_b^L) \xrightarrow[L\to\infty]{(d)} (\Dc_{(a)} ,r_b).
		\end{align}
in distribution in $\mathbb{M}^{2,1}\times \R$. 

Recall that the metric space $\mathbb C_{(a,b)}$
is defined as the complement of the (interior of the) hull $H_{r_b}$ in $\Dc_{(a)}$, and is equipped with the (extension of the)
intrinsic metric $d^\circ$.  
The two boundaries of $\C_{(a,b)}$ are $\partial_0\C_{(a,b)}=\partial \D_{(a)}^{(b)}$, and
	$\partial_1\C_{(a,b)}=\partial H_{r_b}$. To simplify notation, we write $\C$ instead of $\C_{(a,b)}$ in this section and the next one.

We also let $\Cd$ be the unexplored triangulation at time $k_b^L$ in the peeling algorithm applied to $\Dd_{(a)}$.  We equip $\Cd$ with the graph distance 
scaled by the factor $\sqrt{3/2}\,L^{-1/2}$, which we denote by $d_L^\circ$. Recall from Section \ref{sec:peel} the definition of the outer boundary $\partial_0\Cd=\partial \Dd_{(a)}$
and the inner boundary $\partial_1\Cd$.

Our goal in this section is to prove the following theorem. Recall the Gromov-Hausdorff space $(\mathbb{M},d_{\mathtt{GH}})$ introduced
in Section \ref{sec:conv}.

	\begin{theorem}
		\label{PropPrinc}
		The random metric spaces $(\Cd, d_L^\circ)$ converge in distribution towards $(\Cc, d^\circ)$ in 
		$(\mathbb{M},d_{\mathtt{GH}})$. 
	\end{theorem}

Before we proceed to the proof of Theorem \ref{PropPrinc}, we start with some preliminaries.
By the Skorokhod representation theorem, we may assume that the convergence \eqref{ConvBrownianDiskConditioned2} holds almost surely,
	\begin{align}
		\label{ConvBrownianDiskConditionedPS}
		(\Dd_{(a)}, r_b^L) \overset{a.s.}{\underset{L\to \infty}{\longrightarrow}}(\Dc_{(a)}, r_b).
	\end{align}
In the following, it will be useful to argue  on a fixed value of $\omega$ for which \eqref{ConvBrownianDiskConditionedPS} holds. In fact, we will 
need more. We observe that the triangulation $\Cd$ is Boltzmann distributed on the set $\mathbb{T}^2(\lfloor aL\rfloor,\lfloor bL\rfloor)$ of all triangulations with
two boundaries of sizes $\lfloor aL\rfloor$ and $\lfloor bL\rfloor$, and therefore $a$ and $b$ play a symmetric role in the distribution of $\Cd$. 
For any $L\geq 0$, we may introduce a random triangulation $H_0^L$, with a boundary, which is independent of $\mathcal C^L$ 
and distributed as the triangulation discovered by the peeling algorithm applied to a Boltzmann triangulation in $\mathbb{T}^{1,\bullet}(\lfloor bL \rfloor)$ at the first time when the
perimeter of the explored region hits the value $\lfloor aL\rfloor$ (conditionally on the event that this hitting time is finite).  
Let $\tDd_{(b)}$ be the triangulation obtained by gluing $H^L_0$ onto $\Cd$ along the boundary $\FCda$ (thus identifying $\partial H^L_0$ and $\FCda$ and their distinguished boundary edges). 
See Fig.~3 for an illustration. 
By construction, $\tDd_{(b)}$ is a Boltzmann triangulation in $\mathbb{T}^{1,\bullet}(\lfloor bL \rfloor)$ conditioned on the event that the perimeter process (associated with the peeling algorithm) hits
 $\lfloor aL\rfloor$. Hence, by Proposition \ref{ConvBrownianDiskWithRadius},
	 \begin{align}
	 	\label{Convba}
	 	\tDd_{(b)} \xrightarrow[L\to \infty]{(d)} \tDc_{(b)},
	 \end{align}
 where it is implicit that distances on the spaces $\tDd_{(b)}$ are scaled by $\sqrt{3/2}L^{-1/2}$ and $(\tDc_{(b)},\tilde D)$ is a (free pointed) Brownian disk 
 with perimeter $b$ conditioned on the 
 event that the process of hull perimeters hits $a$. 
 
 \begin{figure}[h!]
	\centering
	\includegraphics[width=\textwidth]{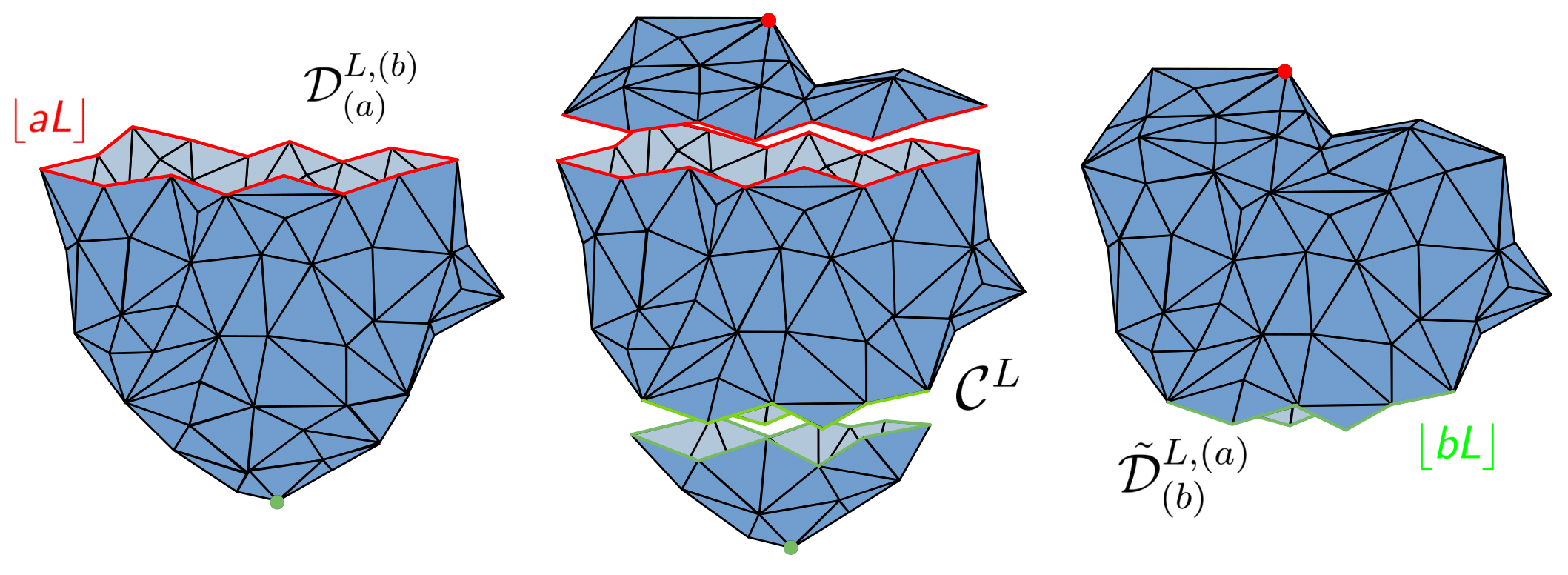}\label{cutting-gluing}
	\caption{Starting from the (conditioned) Boltzmann triangulation $\mathcal{D}^{L,(b)}_{(a)}$ with boundary size $\lfloor a L \rfloor$ (left), one removes a hull 
	of boundary size $\lfloor bL \rfloor$ centered at the distinguished vertex and then glues a hull of boundary size $\lfloor a L \rfloor$ on the boundary $\partial\mathcal{D}^{L,(b)}_{(a)}$
	(middle) to get a (conditioned) Boltzmann triangulation $\tilde{\mathcal{D}}^{L,(a)}_{(b)}$ of boundary size $\lfloor b L \rfloor$ (right)}
\end{figure}

By a tightness argument, we may assume that \eqref{Convba} holds jointly with \eqref{ConvBrownianDiskConditioned} along a subsequence
of values of $L$, and, from now on, we restrict our attention to this subsequence. By the Skorokhod representation theorem, we may assume that
we have both the almost sure convergences \eqref{ConvBrownianDiskConditionedPS} and $\tDd_{(b)} \longrightarrow \tDc_{(b)}$. From now on until the end of Section  
\ref{sec:conv-annulus}, we
fix $\omega$ such that both these convergences hold. For this value of $\omega$, we will prove that
$(\Cd, d_L^\circ)$ converges to $(\Cc, d^\circ)$ in $\mathbb{M}$.

\subsection{Reduction to approximating spaces}
\label{reduc-approx}

Since \eqref{ConvBrownianDiskConditionedPS} holds for the value $\omega$ that we have fixed, we may and will 
assume that $\Dc_{(a)}$ and the spaces $\Dd_{(a)}$ are embedded isometrically in the same compact metric space $(E, \Delta)$, in such a way that 
we have
	\begin{align}
	\label{Prop330}
		&\Dd_{(a)} \underset{L\to \infty}{\xrightarrow{\Delta_{\mathtt H}}} \Dc_{(a)}, \ \ \ \
		x_*^L\underset{L\to \infty}{\xrightarrow{ \ \  \Delta \ \ }} x_*, \ \ \ \ \FDd_{(a)} \xrightarrow[L\to \infty]{\Delta_{\mathtt H}} \FDc_{(a)}.
	\end{align}
Note that the restriction of $\Delta$ to $\Dd_{(a)}$ is the distance $d_L$ and the restriction of $\Delta$
to $\Dc_{(a)}$ is the distance $D$. As a first important remark, we observe  that  (since $r_b$ is not a jump point of the perimeter process $(\mathcal P_r)$), Proposition \ref{ConvHaus} and the fact that $r^L_b\longrightarrow r_b$  give
	\begin{align}
		&\Cd \xrightarrow[L\to \infty]{\Delta_{\mathtt H}} \Cc, & \FCdb \xrightarrow[L\to\infty]{\Delta_{\mathtt H}} \FCcb, \label{Prop512}
	\end{align}

A difficulty in the proof of Theorem \ref{PropPrinc} comes from the fact that the behaviour of the spaces $\Cd$ near the boundaries $\FCdb$
is not easily controlled. We will deal with this problem by first considering approximating subspaces 
which are obtained from $\Cd$, resp. from $\Cc$, by removing a neighborhood of the boundary $\partial_1\mathcal{C}^L$, resp. of $\partial_1\C$, and then showing that the 
convergence in Theorem \ref{PropPrinc} can be reduced to that of the approximating subspaces. For $\delta>0$,  we introduce the space
	\begin{align}
		\label{Cdd}
		\Cdd=\{x\in \Cd : d_L(x, \FCdb)\geq\delta\}=\{x\in \Cd : \Delta(x, \FCdb)\geq\delta\},
	\end{align}
	which is equipped with the restriction of the distance $d_L^\circ$, and its continuous counterpart
	\begin{align}
		\label{Ccd}
		\Ccd=\{x\in \Cc : D(x,\partial_1\C)\geq\delta\}=\{x\in \Cc :\Delta(x, \FCcb)\geq\delta\},
	\end{align}
	which is equipped with the restriction of the distance $d^\circ$. In what follows, we always assume that $\delta$
	is small enough so that $\Ccd$ is not empty and even contains points $x$ such that $d^\circ(x,\partial_1\C)>\delta$. Then, (from \eqref{Prop512}) it follows that 
	$\Cdd$ is not empty at least when $L$ is large.

\begin{lemma}
	\label{ConvCddCcd}
	If $\delta>0$ is not a local maximum of the function $x\mapsto \Delta(x, \partial_1\mathbb C)$ on $\mathbb C$, then:
		\begin{align}
		\label{Hausdo}
		\Delta_{\mathtt H}(\Cdd, \Ccd)\xrightarrow[L\to\infty]{} 0. 
	\end{align}
\end{lemma}
\begin{proof}
	Let us fix $\varepsilon>0$ and $\eta>0$. By (\ref{Prop512}), for every large enough $L$, we have  $\Delta_{\mathtt H}(\FCdb, \FCcb)< \eta/2$. 
If $x\in \Ccd$ is such that $\Delta(x, \FCcb)\geq \delta+\eta$, then (by (\ref{Prop512}) again), we can find  a point $x_L\in \Cd$ such that $\Delta(x_L, x)\leq \varepsilon \wedge\eta/2$, and it follows that:
	\begin{align*}
		\Delta(x_L, \FCdb)\geq \Delta(x, \FCcb)-\Delta(x, x_L)-\Delta_{\mathtt H}(\FCdb, \FCcb)\geq \delta,
	\end{align*}
so that $x_L\in \Cdd$ and $\Delta(x, \Cdd)\leq \Delta(x, x_L)\leq \varepsilon$. If $x$ is at distance exactly $\delta$ from $\FCcb$, we 
can approximate $x$ by a point $x'$ such that $\Delta(x', \FCcb)>\delta$ (we use our assumption that $\delta$ is not a local maximum
of $y\mapsto \Delta(y, \partial_1\mathbb C)$), and, for $L$ large enough, the same argument allows us to find a point $x_L\in \Cdd$ such that $\Delta(x, x_L)\leq \varepsilon$.
A compactness argument then gives $\sup_{x\in \Ccd} \Delta(x, \Cdd)\leq \varepsilon$ when $L$ is large. Since $\varepsilon$ was arbitrary, we have proved
that $\sup_{x\in \Ccd} \Delta(x, \Cdd)\to0$ as $L\to\infty$. A similar argument yields $
\sup_{x\in \Cdd}\Delta(x, \Ccd)\to 0$, which completes the proof.
\end{proof}

\begin{lemma}
\label{dist-bdry-compar}
Set 
$$\mathbb{D}^{(b)}_{(a),\delta}=\{x\in \mathbb{D}^{(b)}_{(a)}: \Delta(x,\partial\mathbb{D}^{(b)}_{(a)})\geq \delta\},\quad\mathcal{D}^{L,(b)}_{(a),\delta}=\{x\in\mathcal{D}^{L,(b)}_{(a)}:
\Delta(x,\partial\mathcal{D}^{L,(b)}_{(a)}) \geq \delta\}.$$
Then, for every $\delta>0$ that is not a local maximum of the function $x\mapsto \Delta(x, \partial\mathbb{D}^{(b)}_{(a)})$ on $\mathbb{D}^{(b)}_{(a)}$, we have
$$\lim_{L\to\infty} \Delta_{\mathtt{H}}(\mathcal{D}^{L,(b)}_{(a),\delta},\mathbb{D}^{(b)}_{(a),\delta})=0,$$
and consequently
$$\limsup_{L\to\infty} \Bigg( \sup_{x\in \mathcal{D}^{L,(b)}_{(a)}} \Delta(x,\mathcal{D}^{L,(b)}_{(a),\delta})\Bigg) 
\leq \sup_{x\in \mathbb{D}^{(b)}_{(a)}} \Delta(x,\mathbb{D}^{(b)}_{(a),\delta}).$$
\end{lemma}

The first assertion of the lemma is proved by arguments similar to the proof of Lemma \ref{ConvCddCcd}, and we omit the details. The second
assertion is an easy consequence of the first one and the fact that $\Delta_{\mathtt{H}}(\mathcal{D}^{L,(b)}_{(a)},\mathbb{D}^{(b)}_{(a)})$
tends to $0$ (cf.~\eqref{Prop330}). 

\medskip
\noindent{\bf Remark}. The first assertion of Lemma \ref{dist-bdry-compar} obviously requires our particular embedding of the spaces 
$\mathcal{D}^{L,(b)}_{(a)}$ and $\mathbb{D}^{(b)}_{(a)}$ in $(E,\Delta)$, but the second one holds independently
of this embedding provided we replace $\Delta$ by $d_L$ in the left-hand side and by $D$ in the right-hand side.


\medskip
	 Let us turn to the proof of Theorem \ref{PropPrinc}. For every $\delta>0$, we have
\begin{align}
	\label{MajoPrinc}
	d_{\mathtt {GH}}(\mathcal C^{L}, \mathbb C)\leq \ \ \ \underbrace{d_{\mathtt {GH}}(\mathcal C^{L}, \mathcal C^{L}_{\delta})}_{(A_{L,\delta})} \ \ + \ \ \underbrace{d_{\mathtt {GH}}(\mathcal C^{L}_{\delta}, \Ccd)}_{(A'_{L,\delta})} \ \ + \ \ \underbrace{d_{\mathtt {GH}}(\Ccd, \Cc)}_{(A''_\delta)}.
\end{align}
where we recall that $\Cd$ and $\Cc$ are equipped with the distances $d_L^\circ$ and $d^\circ$ respectively, and $\Cdd$ and $\Ccd$ are equipped with the restrictions of  these distances.

Our goal is to prove that $d_{\mathtt {GH}}(\mathcal C^{L}, \mathbb C)$ tends to $0$ as $L$ tends to infinity. To this end, we will deal separately with 
each of the terms $A_{L,\delta}$, $A'_{L,\delta}$ and $A''_{\delta}$. Let us start with $A''_\delta$.

\begin{lemma}
\label{PropPrinc1}
We have
$$\lim_{\delta\to 0}d_{\mathtt {GH}}(\Ccd, \Cc)=0.$$
\end{lemma}

\proof
It is enough to verify that
\begin{align}
	\label{PropPrinc1mainresult}\sup_{x\in\mathbb{C}} d^\circ(x,\mathbb{C}_\delta)\xrightarrow[\delta\to 0]{} 0.
\end{align}
If this does not hold, we can find $\alpha>0$ and sequences $x_n\in\mathbb{C}$, $\delta_n\longrightarrow 0$, such that
$d^\circ(x_n,\mathbb{C}_{\delta_n})\geq \alpha$. By compactness we can assume that $x_n\longrightarrow x_\infty\in\mathbb{C}$,
and we get that $d^\circ(x_\infty,\mathbb{C}_\delta)\geq \alpha/2$ for every $\delta>0$, which is absurd because we know that $x_\infty$
must be the limit (with respect to $d^\circ$) of a sequence of points in $\C\setminus \partial_1\C=\cup_{\delta>0}\C_\delta$. 
\endproof

Let us now discuss $A_{L,\delta}$.

\begin{lemma}
\label{PropPrinc2}
We have
$$\lim_{\delta\to 0}\Big( \limsup_{L\to\infty} d_{\mathtt {GH}}(\mathcal C^{L}, \mathcal C^{L}_{\delta})\Big)=0.$$
\end{lemma}

\proof We need to verify that
\begin{align}
\label{PropPrinc2mainresult}	\lim_{\delta\to 0}\Big( \limsup_{L\to\infty} \Big(\sup_{x\in \mathcal{C}^L} d^\circ_L(x,\mathcal C^{L}_{\delta})\Big)\Big) = 0.
\end{align}
Here it is convenient to view $\mathcal{C}^L$ as a subset of the triangulation $\tDd_{(b)}$ introduced in Section \ref{sta-res}. We denote
the rescaled distance on $\tDd_{(b)}$ by $\tilde d_L$, and, for every $\delta>0$, we set
$$\tDd_{(b),\delta}=\{x\in\tDd_{(b)}: \tilde d_L(x,\partial\tDd_{(b)}) \geq \delta\}.$$
We then claim that, for $\delta>0$ small enough, for every sufficiently large $L$, we have
$$\sup_{x\in \mathcal{C}^L} d^\circ_L(x,\mathcal C^{L}_{\delta})
=\sup_{x\in \tDd_{(b)}}\tilde d_L(x,\tDd_{(b),\delta}).$$
Indeed, the properties $ \FCdb \longrightarrow \FCcb$ and $\FCda=\partial \mathcal{D}^{L,(b)}_{(a)} \longrightarrow \FCca$
ensure that for $\delta>0$ small, for every sufficiently large $L$, 
all points of $\partial_0\mathcal{C}^L$ are at distance greater than $\delta$ from $\partial_1\mathcal{C}^L$, and it follows
that $\mathcal{C}^L \setminus\mathcal C^{L}_{\delta}$ is identified with $\tDd_{(b)}\setminus \tDd_{(b),\delta}$. Our claim 
easily follows. 

At this stage, we can use Lemma \ref{dist-bdry-compar} (with the roles of $a$ and $b$ interchanged) and the subsequent remark : except possibly for countably many values of $\delta$, we have
$$\limsup_{L\to\infty} \Big(\sup_{x\in \tDd_{(b)}} \tilde d_L(x,\tDd_{(b),\delta})
\Big)\leq \sup_{x\in\tilde{\mathbb{D}}^{(a)}_{(b)} }\tilde{D}(x,\tilde{\mathbb{D}}^{(a)}_{(b),\delta})$$
where $\tilde{\mathbb{D}}^{(a)}_{(b),\delta}=\{x\in \tilde{\mathbb{D}}^{(a)}_{(b)} : \tilde{D}(x,\partial\tilde{\mathbb{D}}^{(a)}_{(b)})\geq \delta\}$. 

It follows from the preceding considerations that, except possibly for countably many values of $\delta$,
$$\limsup_{L\to\infty} \Big(\sup_{x\in \mathcal{C}^L} d^\circ_L(x,\mathcal C^{L}_{\delta})
\Big)\leq \sup_{x\in\tilde{\mathbb{D}}^{(a)}_{(b)}} \tilde{D}(x,\tilde{\mathbb{D}}^{(a)}_{(b),\delta}).$$
The right-hand side tends to $0$ as $\delta\to 0$, which completes the proof. \endproof

It remains to study the terms $A'_{L,\delta}$.

\begin{lemma}
\label{PropPrinc3}
If $\delta>0$ is not a local maximum of the function $x\mapsto \Delta(x, \partial_1\mathbb C)$ on $\mathbb{C}$, we have
$$\lim_{L\to\infty} d_{\mathtt {GH}}(\mathcal C^{L}_{\delta}, \Ccd)=0.$$
\end{lemma}

Let us postpone the proof of Lemma \ref{PropPrinc3} to the next section, and recall 
the bound \eqref{MajoPrinc}. By letting first $L$ tend to infinity and then $\delta$ tend to $0$, using Lemmas
\ref{PropPrinc1}, \ref{PropPrinc2} and \ref{PropPrinc3}, we get
$$\limsup_{L\to\infty} d_{\mathtt {GH}}(\mathcal C^{L}, \mathbb C)=0$$
which completes the proof of Theorem \ref{PropPrinc}. Therefore, it only remains to
prove Lemma \ref{PropPrinc3}.

\subsection{Proof of the key lemma}
\label{sec:key-lemma}

In this section, we prove Lemma \ref{PropPrinc3}. We let $\delta>0$ such that $\delta$
 is not a local maximum of the function $x\mapsto \Delta(x, \partial_1\mathbb C)$.
Recalling Lemma \ref{ConvCddCcd},
we define a correspondence between $\Cdd$ and $\Ccd$ by setting
$$\mathcal{R}_L=\{(x_L,x')\in \Cdd\times\mathbb{C}_\delta:\Delta(x_L,x')\leq \Delta_{\mathtt H}(\Cdd, \Ccd)\}.$$
By the classical result expressing the Gromov-Hausdorff distance 
in terms of distortions of correspondences \cite[Chapter 7]{BBY}, the statement of Lemma \ref{PropPrinc3} will follow if we can prove that the distorsion 
of $\mathcal{R}_L$ tends to $0$ as $L\to\infty$, or
equivalently
\begin{equation}
\label{disto1}
\sup_{(x_L,x')\in \mathcal{R}_L,(y_L,y')\in \mathcal{R}_L} \Big|d^\circ_L(x_L,y_L)-d^\circ(x',y')\Big| \xrightarrow[L\to\infty]{} 0.
\end{equation}

We first verify that
\begin{equation}
\label{disto2}
\sup_{(x_L,x')\in \mathcal{R}_L,(y_L,y')\in \mathcal{R}_L} \Big(d^\circ_L(x_L,y_L)-d^\circ(x',y')\Big) \xrightarrow[L\to\infty]{} 0.
\end{equation}
To this end, we argue by contradiction. If \eqref{disto2} does not hold, we can find $\eta>0$
and sequences $L_k\uparrow\infty$, and $(x_{L_k},x'_k)$, $(y_{L_k},y'_k)$ in $\mathcal{R}_{L_k}$ such that
$$d^\circ_{L_k}(x_{L_k},y_{L_k})>d^\circ(x'_k,y'_k) + \eta.$$
We may assume that $x'_k\longrightarrow x'_\infty$ and $y'_k\longrightarrow y'_\infty$
where $x'_\infty,y'_\infty\in \mathbb{C}_\delta$, and for $k$ large we have also
\begin{equation}
\label{disto20}
d^\circ_{L_k}(x_{L_k},y_{L_k})>d^\circ(x'_\infty,y'_\infty) + \frac{\eta}{2}.
\end{equation}
From \eqref{Hausdo} and the definition of the correspondence $\mathcal{R}_L$, we also get that
$\Delta(x_{L_k},x'_\infty)\longrightarrow 0$ and $\Delta(y_{L_k},y'_\infty)\longrightarrow 0$. 

Since $d^\circ$ coincides with the (extension of the) intrinsic distance on $\mathbb{C}\setminus \partial_1\C$,
we can find a path $\gamma$ from $x'_\infty$ to $y'_\infty$ in $\mathbb{C}$ that does not hit the 
boundary $\partial_1\mathbb{C}$ and whose length is bounded above by $d^\circ(x'_\infty,y'_\infty)+\eta/4$. 
From part 1 of Lemma \ref{approx-path}, if $k$ is large, we can approximate
$\gamma$ by a path $\gamma_{L_k}$ going from $x_{L_k}$ to $y_{L_k}$ in $\mathcal{C}^{L_k}$, whose length is bounded
above   by $d^\circ(x'_\infty,y'_\infty)+3\eta/8$, such that $\gamma_{L_k}$
will not hit $\partial_1\mathcal{C}^{L_k}$ (we use the convergence of $\partial_1\mathcal{C}^{L_k}$
to $\partial_1\mathbb{C}$) and therefore stays in $\mathcal{C}^{L_k}$. It follows that 
$d^\circ_{L_k}(x_{L_k},y_{L_k})$ is bounded above by the length of $\gamma_{L_k}$
giving a contradiction with \eqref{disto20}. This completes 
the proof of \eqref{disto2}. 

In order to complete the proof of \eqref{disto1}, we still need to verify that
\begin{equation}
\label{disto3}
\sup_{(x_L,x')\in \mathcal{R}_L,(y_L,y')\in \mathcal{R}_L} \Big(d^\circ(x',y') -d^\circ_L(x_L,y_L)\Big) \xrightarrow[L\to\infty]{} 0.
\end{equation}
This is slightly more delicate than the proof of \eqref{disto2}, and we will need the following lemma.

	\begin{lemma}[Paths remaining far from the boundary]\label{CheminLoinDuBord} Let $\eta>0$.
		There exist $\varepsilon>0$ and $L_0\geq 0$ such that, for any choice of $x^L, y^L\in \Cdd$ with $L\geq L_0$, there is a path between $x^ L$ and $y^L$ in $\Cd$ which stays at distance at least $\varepsilon$ from $\FCdb$ and whose length is bounded by $d_{L}^\circ(x^L, y^ L)+\eta$.
	\end{lemma}
	
	\proof[Proof of Lemma \ref{CheminLoinDuBord}] 
Let us argue by contradiction. If the desired property does not hold, we can find sequences
$\varepsilon_n\longrightarrow 0$, $L_n\longrightarrow\infty$, $x_n,y_n\in \mathcal{C}^{L_n}_\delta$
such that any path from $x_n$ to $y_n$ that stays at distance at least $\varepsilon_n$ from 
$\partial_1\mathcal{C}^L$ has length greater than $d_{L_n}^\circ(x_n, y_n)+\eta$. 
By compactness, we may assume that $x_n\longrightarrow x_\infty$ and 
$y_n\longrightarrow y_\infty$ in $(E,\Delta)$ and, by \eqref{Hausdo}, we have $x_\infty,y_\infty\in\C_\delta$. Additionally, since the 
diameters of $(\mathcal{C}^{L_n}_\delta,d^\circ_{L_n})$ are bounded (this follows from \eqref{disto2} since
the diameter of $\mathbb{C}$ is finite), we can assume that $\ell_n:=d^\circ_{L_n}(x_n,y_n)$ 
converges to some real $\ell_\infty\geq 0$. For every $n$, let $\gamma_n$
be a geodesic from $x_n$ to $y_n$. By a standard argument, we can extract from the sequence
$(\gamma_n(t\wedge \ell_n),t\in[0,\ell_\infty+\frac{\eta}{3}])$ a subsequence that converges uniformly (for the metric $\Delta$) to a path $\gamma_\infty
=(\gamma_\infty(t),t\in[0,\ell_\infty+\frac{\eta}{3}])$ that
connects $x_\infty$ to $y_\infty$ in $\mathbb{D}_{(a)}^{(b)}$. By \eqref{Prop512}, $\gamma_\infty$ takes values in $\mathbb{C}$. Moreover, from 
the analogous property for the discrete paths $\gamma_n$, we get that $\gamma_\infty$ is $1$-Lipschitz, meaning that
$\Delta(\gamma_\infty(s),\gamma_\infty(t))\leq |t-s|$ for every $s,t$. It follows in particular that the length of 
$\gamma_\infty$ is at most $\ell_\infty+\frac{\eta}{3}$.

The path $\gamma_\infty$ may hit $\partial_1\mathbb{C}$.  
Using Lemma \ref{avoiding-path}, we can however find another path $\gamma'_\infty$ connecting $x_\infty$
to $y_\infty$ in $\mathbb{C}$, which does not hit $\partial_1\mathbb{C}$ and has length
at most $\ell_\infty+\frac{2\eta}{3}$. The path $\gamma'_\infty$ stays at positive distance $\alpha$
from $\partial_1\mathbb{C}$. Using part 2 of Lemma \ref{approx-path} (and the
fact that $\partial_1\mathcal{C}^L$ converges to $\partial_1\mathbb{C}$
for the $\Delta$-Hausdorff measure, by \eqref{Prop512}), we can then, for $n$ large enough, find a path 
$\gamma'_n$ connecting $x_n$ to $y_n$ in $\mathcal{C}^{L_n}$,
with length smaller that $d^\circ_{L_n}(x_n,y_n)+\eta$,  that will stay at distance
at least $\alpha/2$ from $\partial_1\mathcal{C}^{L_n}$. This is a contradiction as soon
as $\varepsilon_n<\alpha/2$. \endproof

	Let us complete the proof of \eqref{disto3}. We again argue by contradiction.
	 If \eqref{disto3} does not hold, we can find $\eta>0$
and sequences $L_k\uparrow\infty$, and $(x_{L_k},x'_k)$, $(y_{L_k},y'_k)$ in $\mathcal{R}_{L_k}$ such that
$$d^\circ(x'_k,y'_k) >d^\circ_{L_k}(x_{L_k},y_{L_k})+ \eta.$$
We may assume that $x'_k\longrightarrow x'_\infty$ and $y'_k\longrightarrow y'_\infty$ 
where $x'_\infty,y'_\infty\in \mathbb{C}_\delta$. By Lemma \ref{CheminLoinDuBord}, we can find $\varepsilon>0$ such that,
for every large enough $k$, there is a path $\gamma_{L_k}$ from $x_{L_k}$ to $y_{L_k}$
in $\mathcal{C}^{L_k}$ that stays at distance at least $\varepsilon$ from $\partial_1\mathcal{C}^{L_k}$ and whose
length is bounded by $d^\circ_{L_k}(x_{L_k},y_{L_k})+\eta/2$.

We have $\Delta(x_{L_k},x'_\infty)\longrightarrow 0$ and $\Delta(y_{L_k},y'_\infty)\longrightarrow 0$, and, by part 2 of Lemma \ref{approx-path}, we can (for $k$ large) find a path $\gamma'_k$ from $x'_\infty$
to $y'_\infty$ in $\mathbb{D}_{(a)}^{(b)}$ that stays at distance at least $\varepsilon/2$ from $\partial_1\mathbb{C}$ (we again use the
convergence of $\partial_1\mathcal{C}^{L_k}$
to $\partial_1\mathbb{C}$) and has length smaller than $d^\circ_{L_k}(x_{L_k},y_{L_k})+ 3\eta/4$. Hence 
$d^\circ(x'_\infty,y'_\infty)<d^\circ_{L_k}(x_{L_k},y_{L_k})+ 3\eta/4$, and also, for $k$ large,
$d^\circ(x'_k,y'_k)<d^\circ_{L_k}(x_{L_k},y_{L_k})+ 7\eta/8$.
We get a contradiction, which completes the proof of 
\eqref{disto1} and of Theorem \ref{PropPrinc}. \hfill$\square$

\section{Convergence of boundaries and volume measures}
\label{sec:conv-bound-vol}

In the last section, we showed that the sequence of metric spaces $(\Cd, d^\circ_L)$ converges in law towards $(\Cc, d^\circ)$ for the Gromov-Hausdorff topology. We 
will now explain how to extend this result to the setting of marked measure metric spaces. We write $\mu_L$ for the restriction to $\Cd$ of the (scaled) counting measure $\nu^L$, and $\mu$ for the restriction to $\Cc$ of the volume measure $\mathbf V$.

\begin{theorem}
\label{main-extend}
The random marked measure metric spaces
$$\mathcal X^L:=((\Cd, d^\circ_L), (\partial_0\Cd, \partial_1\Cd),  \mu_L),$$ 
converge  towards $\mathcal Y:=((\Cc, d^\circ), (\FCca, \FCcb),\mu)$ in distribution
 in the space $\mathbb M^{2, 1}$.
 \end{theorem}

\proof As in the previous section, we may restrict our attention to a sequence of values of $L$
such that the convergences (\ref{ConvBrownianDiskConditionedPS}) and (\ref{Convba})  hold almost surely.
Fixing $\omega$ in the underlying probability space, we can assume that $\Dc_{(a)}$ and the spaces $\Dd_{(a)}$ are embedded isometrically in the same compact metric space $(E, \Delta)$, in such a way that 
the convergences \eqref{Prop330} hold for the Hausdorff distance associated with $\Delta$, and moreover the measures $\nu_L$ converge weakly to $\mathbf{V}$. As explained at the beginning of
Section \ref{reduc-approx}, we can also assume that \eqref{Prop512} holds. We recall the definition of $\Cdd$ and $\Ccd$
in \eqref{Cdd} and \eqref{Ccd}, and we also set
	\begin{align*}
		&\partial_1\Cdd = \{x\in \Cdd :\Delta(x, \FCdb)=\delta\} \ \ \ \text{ and } \ \ \ \partial_1\Ccd = \{x\in \Ccd : \Delta(x, \FCcb)=\delta\}.
	\end{align*}
	In what follows, we always assume that $\delta>0$ is small enough so that $\Delta(\partial_0\mathbb{C},\partial_1\mathbb{C})>\delta$, and in particular $\Ccd $ is not empty. 

\begin{lemma}
		\label{ConvApproxBoundaries}
	If $\delta>0$ is not a local maximum of the function $x\mapsto \Delta(x, \partial_1\mathbb C)$ on $\mathbb C$  we have
	\begin{align*}
		\Delta_{\mathtt H}(\partial_1 \mathcal  C^L_\delta, \partial_1 \mathbb C_\delta)\xrightarrow[L\to\infty]{} 0.
	\end{align*}
	Moreover,
	\begin{equation}
	\label{small-measure}
	\lim_{\delta\to 0}\Big( \limsup_{L\to\infty} \mu_L(\mathcal{C}^L\setminus\Cdd)\Big) =0.
	\end{equation}
	\end{lemma}
	
	The first part of the lemma is derived by arguments similar to the proof of Lemma \ref{ConvCddCcd}. The second part follows
	from the weak convergence of $\nu_L$ to $\mathbf{V}$ and the fact that $\mathbf{V}$ puts no mass on $\partial_1\mathbb{C}$. 
	We leave the details to the reader. 

We then set
$$\theta_L(\delta)= \max\Big(\Delta_{\mathtt H}(\Cdd,\Ccd),\Delta_{\mathtt H}(\partial_1 \mathcal  C^L_\delta, \partial_1 \mathbb C_\delta),\Delta_{\mathtt H}(\partial_0 \mathcal  C^L, \partial_0 \mathbb C)\Big).$$
By \eqref{Prop330}, \eqref{Hausdo} and Lemma \ref{ConvApproxBoundaries}, we have 
$$\lim_{L\to\infty} \theta_L(\delta) = 0$$
except possibly for countably many values of $\delta$. We then slightly modify the definition of the correspondence $\mathcal{R}_L$
by setting
$$\mathcal{R}'_L=\{(x_L,x')\in \Cdd\times\mathbb{C}_\delta:\Delta(x_L,x')\leq \theta_L(\delta)\}.$$
The very same arguments as in Section \ref{sec:key-lemma} show that
the distortion of $\mathcal{R}'_L$ tends to $0$ as $L\to\infty$ (again except possibly for countably many values of $\delta$).

We will now extend $\mathcal{R}'_L$ to a correspondence between $\Cd$ and $\Cc$. We start by fixing $\eta>0$, and we set
$$\alpha_L(\delta):=\sup_{x\in \mathcal{C}^L} d^\circ_L(x,\mathcal C^{L}_{\delta})\;,\quad \alpha(\delta):=
\sup_{x\in\mathbb{C}} d^\circ(x,\mathbb{C}_\delta).$$
By \eqref{PropPrinc2mainresult} and \eqref{PropPrinc1mainresult}, we can choose $\delta\in(0,\eta)$
small enough so that we have both $\alpha(\delta)\leq \eta$ and $\alpha_L(\delta)\leq \eta$ for every sufficiently large $L$. 
Additionally, the second assertion of the lemma allows us to assume that $\mu_L(\Cd\setminus\mathcal{C}^L_{2\delta})<\eta$ for $L$ large. 
In what follows, we fix $\delta\in(0,\eta)$ so that the preceding properties hold (and both $\delta$ and $2\delta$ do not belong 
to the countable set that was excluded above). To
simplify notation, we write $\alpha_L=\alpha_L(\delta)$ and $\alpha=\alpha(\delta)$.

We define a correspondence between $\Cd$ and $\Cc$ by setting
\begin{align*}
\mathcal{R}^*_L:=\{(x_L,x)\in\Cd\!\times\! \Cc: \exists \tilde x_L\in \Cdd, \tilde x\in \Ccd\hbox{ s.t.}\,(\tilde x_L,\tilde x)\in \mathcal{R}'_L,\,
d^\circ_L(x_L,\tilde x_L)\leq \alpha_L, \,d^\circ(x,\tilde x)\leq \alpha\}. 
\end{align*}
Then, we can easily bound the distortion $\mathrm{dis}(\mathcal{R}^*_L)$ of $\mathcal{R}^*_L$ in terms of the
the distortion $\mathrm{dis}(\mathcal{R}'_L)$ of $\mathcal{R}'_L$: if $(x_L,x),(y_L,y)\in \mathcal{R}^*_L$, we can find 
$(\tilde x_L,\tilde x),(\tilde y_L,\tilde y)\in \mathcal{R}'_L$ such that
$$|d^\circ_L(x_L,y_L)-d^\circ(x,y)|\leq 2\alpha_L+|d^\circ_L(\tilde x_L,\tilde y_L)-d^\circ(\tilde x,\tilde y)| + 2\alpha,$$
and it follows that
$$\mathrm{dis}(\mathcal{R}^*_L)\leq 2(\alpha+\alpha_L)+ \mathrm{dis}(\mathcal{R}'_L)\leq 4\eta + \mathrm{dis}(\mathcal{R}'_L).$$

To prove the desired convergence of $\mathcal X^L$ towards $\mathcal{Y}$ in $\mathbb{M}^{2,1}$, we will use the definition of the
Gromov-Hausdorff-Prokhorov distance $d^{2,1}_{\mathtt{GHP}}$. By a classical argument (cf. \cite[Section 7.3]{BBY}), we can
define a distance $\Delta^{L,*}$ on the disjoint union $\mathcal C^L\sqcup \mathbb C$, such that
the restriction of  $\Delta^{L,*}$ to $\Cd$ is  $d^\circ_L$, the restriction of $\Delta^{L,*}$ to $\Cc$ is $d^\circ$, and, for every $x_L\in \mathcal{C}^L$ and $x\in \mathbb{C}$:
	\begin{align*}
		\Delta^{L,*}(x_L,x)=
		\frac{1}{2}\mathrm{dis}(\mathcal{R}^*_L)+\underset{(y_L,y)\in \mathcal R_L^*}{\inf}\Big(d_L^\circ(x_L,y_L)+d^\circ(x,y)\Big).
	\end{align*}
	Since $(\mathcal C^L, d^\circ_L)$ and $(\mathbb C, d^\circ)$ are embedded isometrically in $(\mathcal C^L\sqcup \mathbb C, \Delta^{L,*})$, we can then use the
	definition of the
Gromov-Hausdorff-Prokhorov distance to bound $d^{2,1}_{\mathtt{GHP}}(\mathcal X^L,\mathcal{Y})$. We need to bound each of the four terms 
appearing in the infimum of the definition. We again use the notation $\Delta^{L,*}_{\mathtt{H}}$, resp. $\Delta^{L,*}_{\mathtt{P}}$, for the Hausdorff distance,
resp.~the Prokhorov distance, associated with $\Delta^{L,*}$. 

\smallskip
\noindent{\it First step.} We verify that
\begin{equation}
\label{first-step}
\max\Big(\Delta^{L,*}_{\mathtt{H}}(\Cd,\Cc),\Delta^{L,*}_{\mathtt{H}}(\partial_1\Cd,\partial_1\Cc),\Delta^{L,*}_{\mathtt{H}}(\partial_0\Cd,\partial_0\Cc)\Big)\leq\frac{1}{2}\mathrm{dis}(\mathcal{R}^*_L)
+\max(\alpha,\alpha_L)+\delta.
\end{equation}
First, it is immediate from the definition of $\Delta^{L,*}$ that $\Delta^{L,*}_{\mathtt{H}}(\Cd,\Cc)\leq \frac{1}{2}\mathrm{dis}(\mathcal{R}^*_L)$. Similarly, the fact that
$\Delta_{\mathtt H}(\partial_0 \mathcal  C^L, \partial_0 \mathbb C)\leq \theta_L(\delta)$ and the definition of $\mathcal{R}'_L$ give 
$\Delta^{L,*}_{\mathtt{H}}(\partial_0\Cd,\partial_0\Cc)\leq \frac{1}{2}\mathrm{dis}(\mathcal{R}^*_L)$. 

Let us bound $\Delta^{L,*}_{\mathtt{H}}(\partial_1\Cd,\partial_1\Cc)$. Let $x_L\in \partial_1\Cd$. From the definition of $\alpha_L$, we can find $x'_L\in \Cdd$ such that 
$d^\circ_L(x_L,x'_L)\leq \alpha_L$. By considering a geodesic from $x'_L$ to $x_L$, we can even assume that $x'_L\in \partial_1\Cdd$. It follows 
that there exists $x'\in\partial_1\Ccd$ such that $\Delta(x',x'_L)\leq \Delta_{\mathtt{H}}(\partial_1\Ccd,\partial_1\Cdd)$, hence $(x'_L,x')\in \mathcal{R}'_L\subset \mathcal{R}^*_L$. 
From the definition of $\Delta^{L,*}$, we get
$$\Delta^{L,*}(x_L,x')\leq \frac{1}{2}\mathrm{dis}(\mathcal{R}^*_L) + d^\circ_L(x_L,x'_L)\leq \frac{1}{2}\mathrm{dis}(\mathcal{R}^*_L)+\alpha_L.$$
Finally, since $x'\in \partial_1\Ccd$, we can find $x''\in \partial_1\Cc$ such that $\Delta^{L,*}(x',x'')=\delta$, and we get
$$\Delta^{L,*}(x_L,x'')\leq \frac{1}{2}\mathrm{dis}(\mathcal{R}^*_L)+\alpha_L+\delta.$$
In a symmetric manner, we can verify that, for any $y\in \partial_1\Cc$, we can find $y_L\in \partial_1\mathcal{C}^L$ such that
$$\Delta^{L,*}(y_L,y)\leq \frac{1}{2}\mathrm{dis}(\mathcal{R}^*_L)+\alpha+\delta.$$
This gives the desired bound for $\Delta^{L,*}_{\mathtt{H}}(\partial_1\Cd,\partial_1\Cc)$, thus completing the proof of \eqref{first-step}. As an immediate consequence,
using also our estimate for $\mathrm{dis}(\mathcal{R}^*_L)$, we get
\begin{equation}
\label{tech-first}
\limsup_{L\to\infty}\Big(\max\Big(\Delta^{L,*}_{\mathtt{H}}(\Cd,\Cc),\Delta^{L,*}_{\mathtt{H}}(\partial_1\Cd,\partial_1\Cc),\Delta^{L,*}_{\mathtt{H}}(\partial_0\Cd,\partial_0\Cc)\Big)\Big)\leq 4\eta.
\end{equation}

\smallskip
\noindent{\it Second step.} We now want to bound $\Delta^{L,*}_{\mathtt{P}}(\mu_L,\mu)$. We start by observing that, if $L$ is large enough,  if $x\in \Ccd$ and $x_L\in\Cdd$
are such that $\Delta(x_L,x)<\delta/2$, we have
\begin{equation}
\label{tech-second}
\Delta^{L,*}(x_L,x)\leq \Delta(x_L,x)+ \frac{1}{2}\mathrm{dis}(\mathcal{R}^*_L)+\theta_L(\delta).
\end{equation}
Indeed, we can find $x'\in\Ccd$ such that $\Delta(x_L,x')\leq \theta_L(\delta)$ (and in particular $(x_L,x')\in\mathcal{R}'_L$), then $\Delta(x,x')\leq \Delta(x,x_L)+\theta_L(\delta)<\delta$
provided that $L$ is large enough so that $\theta_L(\delta)<\delta/2$. Since $x$ and $x'$ both belong to $\Ccd$ and $\Delta(x,x')<\delta$,
we must have $\Delta(x,x')=d^\circ(x,x')$, and 
$$\Delta^{L,*}(x_L,x)\leq \frac{1}{2}\mathrm{dis}(\mathcal{R}^*_L)+d^\circ(x,x')=\frac{1}{2}\mathrm{dis}(\mathcal{R}^*_L)+\Delta(x,x')\leq \frac{1}{2}\mathrm{dis}(\mathcal{R}^*_L)+\theta_L(\delta)+\Delta(x_L,x),$$
which gives our claim \eqref{tech-second}. 

Let $A$ be a measurable subset of $\Cd$. We have
$\mu_L(A)\leq \mu_L(A\cap \mathcal{C}^L_{2\delta})+ \mu_L(\Cd\setminus \mathcal{C}^L_{2\delta})$
and we know that $\mu_L(\Cd\setminus \mathcal{C}^L_{2\delta})<\eta$ when $L$ is large. On the other hand,
by the weak convergence of $\nu_L$ to $\mathbf{V}$, we have also for $L$ large,
$$\mu_L(A\cap \mathcal{C}^L_{2\delta})=\nu_L(A\cap \mathcal{C}^L_{2\delta})\leq \mathbf{V}(\{x\in\mathbb{D}^{(b)}_{(a)}:\Delta(x,A\cap \mathcal{C}^L_{2\delta})<\frac{\delta}{2}\})+\frac{\delta}{2}.$$
Since $\Delta_{\mathtt{H}}(\mathcal{C}^L_{2\delta},\mathbb{C}_{2\delta})$ tends to $0$ as $L\to\infty$, the properties $x\in \mathbb{D}^{(b)}_{(a)}$ and  $\Delta(x,\mathcal{C}^L_{2\delta})<\frac{\delta}{2}$
imply (for $L$ large) that $x\in \Ccd$, and in particular we can replace $\mathbb{D}^{(b)}_{(a)}$ by $\Cc$ and $\mathbf{V}$
by $\mu$ in the last display. But then we can use \eqref{tech-second} to get that, for $x\in \Cc_\delta$,
$$\Delta(x,A\cap \mathcal{C}^L_{2\delta})<\frac{\delta}{2}\quad\Rightarrow\quad \Delta^{L,*}(x,A\cap\mathcal{C}^L_{2\delta})<\frac{\delta}{2} + \frac{1}{2}\mathrm{dis}(\mathcal{R}^*_L)+\theta_L(\delta).$$
Finally, we have, for $L$ large,
\begin{align*}
\mu_L(A)\leq \mu_L(A\cap \mathcal{C}^L_{2\delta}) +\eta&\leq \mu(\{x\in\Cc:\Delta(x,A\cap \mathcal{C}^L_{2\delta})<\frac{\delta}{2}\})+\frac{\delta}{2}+\eta\\
&\leq \mu(\{x\in\Cc:\Delta^{L,*}(x,A)<\frac{\delta}{2}+\frac{1}{2}\mathrm{dis}(\mathcal{R}^*_L)+\theta_L(\delta)\}) +\frac{\delta}{2}+\eta\\
&\leq \mu(\{x\in\Cc:\Delta^{L,*}(x,A)<3\eta\}) + 2\eta.
\end{align*}
A symmetric argument (left to the reader) shows that for $L$ large, for any measurable subset $A$ of $\Cc$, we have
$$\mu(A)\leq \mu_L(\{x\in\Cd:\Delta^{L,*}(x_L,A)<3\eta\}) + 2\eta.$$
This proves that $\Delta_{\mathtt{P}}(\mu_L,\mu)\leq 3\eta$ when $L$ is large. Since $\eta$ was arbitrary, we can combine this with \eqref{tech-first} to get
the desired convergence of $d^{2,1}_{\mathtt{GHP}}(\mathcal{X}^L,\mathcal{Y})$ to $0$. \endproof

%

\section{The complement of two hulls in the Brownian sphere}
\label{sec:compl2hulls}

 In this section, we  fix $r, r'>0$. Recall that $B_r^\bullet(\bx_*)$ is the hull of radius $r$ centered at $\bx_*$ in the free Brownian sphere $\bm_\infty$ (this hull is defined 
 on the event $\{\mathbf{D}(\bx_*,\bx_0)>r\}$).
 It is shown in \cite[Theorem 8]{Stars} that  the intrinsic metric on $B^\circ_r(\bx_*)= B^\bullet_r(\bx_*)\setminus \partial B^\bullet_r(\bx_*)$ 
 has a.s.~a continuous extension to its closure $B^\bullet_r(\bx_*)$. In the following, we implicitly endow $B^\bullet_r(\bx_*)$ with this extended intrinsic metric and we equip it with the restriction of the volume measure on $\bm_\infty$, the distinguished point $\bx_*$ and the boundary $\partial B^\bullet_r(\bx_*)$, so that we can consider $B_r^\bullet(\bx_*)$ as a random variable in $\mathbb M^{2, 1}$. 
 Since $\bx_*$ and $\bx_0$ play symmetric roles in the Brownian sphere \cite[Proposition 3]{Stars}, we can similarly consider, on the event  $\{\mathbf{D}(\bx_*,\bx_0)>r'\}$, the hull of radius $r'$ centered at $\bx_0$ in $\bm_\infty$,
 which we denote by $B_{r'}^\bullet(\bx_0)$ (this is defined as the 
  complement of the connected component of $\bm_\infty\setminus B^\infty_{r'}(\bx_0)$ that contains
  $\bx_*$). We can endow this space with its (extended) intrinsic metric as we did for $B_r^\bullet(\bx_*)$ and consider $B_{r'}^\bullet(\bx_0)$ as a random variable in $\mathbb M^{2, 1}$ by equipping it with the restriction of the volume measure on $\bm_\infty$, the distinguished point $\bx_0$ and the boundary $\partial B^\bullet_{r'}(\bx_0)$. 
  We also consider the perimeter of these hulls. The perimeter of $B_r^\bullet(\bx_*)$ is $\mathcal Z_r^{\bx_*}:=\mathbf P_r$ as given by formula (\ref{approx-exit})
  and symmetrically the perimeter $\mathcal Z_{r'}^{\bx_0}$ of $B_{r'}^\bullet(\bx_0)$ may be defined by the analog of (\ref{approx-exit}) where $\bx_*$
  is replaced by $\bx_0$:
 \begin{equation*}
 	\label{approx-exit2}
 	\mathcal Z_{r'}^{\bx_0}:=\lim_{\varepsilon\to 0} \frac{1}{\varepsilon^2} \mathrm{Vol}(\{x\in \bm_\infty\setminus B^\bullet_{r'}(\bx_0):\mathbf{D}(x,B^\bullet_{r'}(\bx_0))<\varepsilon\}).
 \end{equation*}
On the event where $\mathbf D(\bx_*, \bx_0)>r+r'$, the hulls  $B^\bullet_r(\bx_*)$ and $B^\bullet_{r'}(\bx_0)$ are disjoint, and we consider the subspace
 \begin{align*}
 	\mathcal	C^{\bx_*, \bx_0}_{r,r'}:= \text{Closure} \left(\bm_\infty \setminus \left(B_r^\bullet(\bx_*)\cup B_{r'}^\bullet(\bx_0)\right)\right).
 \end{align*}
It is shown in \cite[Corollary 9]{Stars} that, a.s.~on the event $\{\mathbf D(\bx_*, \bx_0)>r+r'\}$, the intrinsic metric on $\bm_\infty \setminus \left(B_r^\bullet(\bx_*)\cup B_{r'}^\bullet(\bx_0)\right)$ has a continuous extension on $\mathcal C^{\bx_*, \bx_0}_{r,r'}$, which is a metric on this space
(to be precise, \cite[Corollary 9]{Stars} considers only the case $r=r'$, but the argument is the same without this condition). So we can view $\mathcal C^{\bx_*, \bx_0}_{r,r'}$ as a random variable in $\mathbb{M}^{2,1}$ by equipping 
this space with the restriction of the volume measure of $\bm_\infty$ and with the ``boundaries'' $\partial B^\bullet_r(\bx_*)$ and $\partial B^\bullet_{r'}(\bx_0)$.

\begin{figure}[h!]
	\centering
	\includegraphics[width=0.6\textwidth]{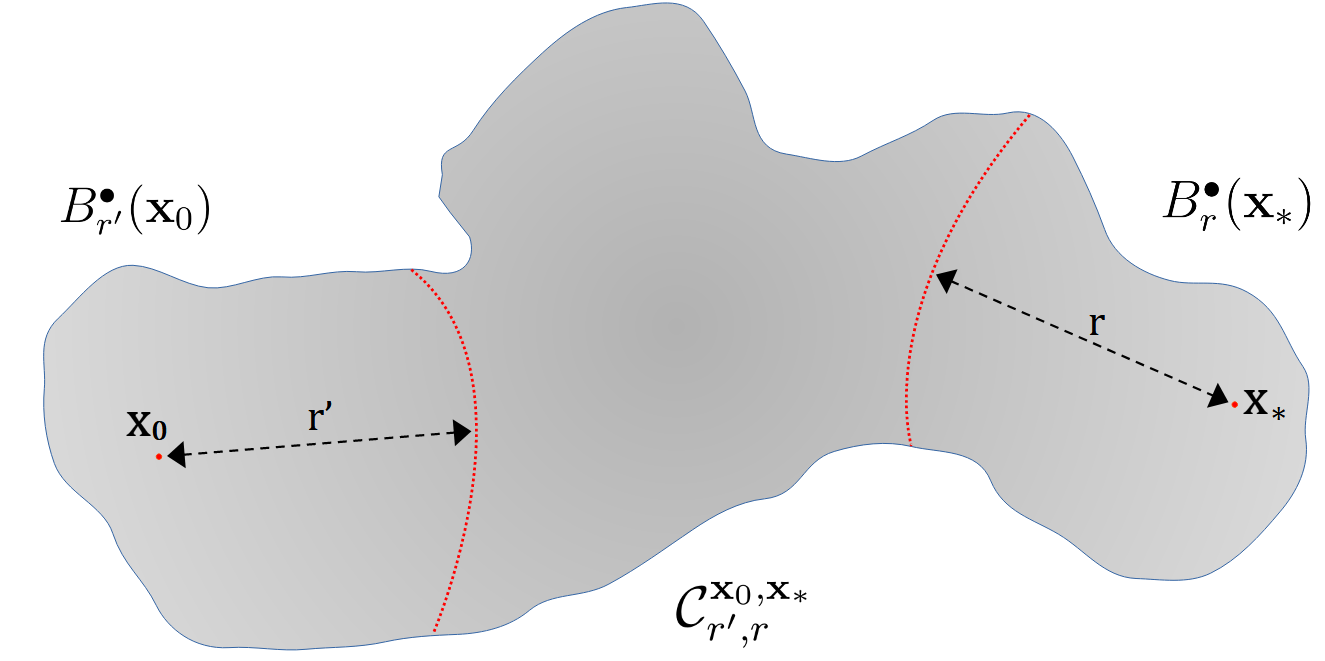}
	\caption{Cutting out two hulls centered at the points $\bx_0$ and $\bx_*$ in the Brownian sphere yields three subsets: the two hulls $B_{r'}^\bullet(\bx_0)$ and $B_r^\bullet(\bx_*)$ and an intermediate part $\mathcal C_{r,r'}^{\bx_*, \bx_0}$. When endowed with their intrinsic metrics, and conditionally on the values taken by the perimeters of the hulls, these metric spaces are independent and distributed as two standard hulls and a Brownian annulus.}
\end{figure}

We finally recall the notion of a standard hull with radius $r$ and perimeter $z>0$, as defined in \cite[Section 3.1]{Stars}.

\begin{theorem}
\label{3-pieces}
Under the probability measure $\N_0(\cdot\,|\,\, \mathbf D(\bx_*, \bx_0)>r+r')$,
the three spaces  $B_r^\bullet(\bx_*)$,  $B_{r'}^\bullet(\bx_0)$ and $\mathcal C^{\bx_*, \bx_0}_{r, r'}$
are conditionally independent given the pair $(\mathcal Z_{r}^{\bx_*},\mathcal Z_{r'}^{\bx_0})$, and their
conditional distribution can be described as follows. The spaces $B_r^\bullet(\bx_*)$ and  $B_{r'}^\bullet(\bx_0)$
are standard hulls of respective radii $r$ and $r'$ and of respective perimeters $\mathcal Z_{r}^{\bx_*}$ and $\mathcal Z_{r'}^{\bx_0}$.
The space $\mathcal	C^{\bx_*, \bx_0}_{r,r'}$ is a Brownian annulus with perimeters $\mathcal Z_{r}^{\bx_*}$ and $\mathcal Z_{r'}^{\bx_0}$.
\end{theorem}

This theorem is closely related to \cite[Corollary 9]{Stars} (see also \cite[Lemma 6.3]{ang2022moduli}). In fact,
\cite[Corollary 9]{Stars} (stated for $r=r'$ but easily extended) already gives the 
conditional independence of $B_r^\bullet(\bx_*)$,  $B_{r'}^\bullet(\bx_0)$ and $\mathcal C^{\bx_*, \bx_0}_{r,r'}$ given $(\mathcal Z_{r}^{\bx_*},\mathcal Z_{r'}^{\bx_0})$,
and identifies the conditional distribution of the hulls $B_r^\bullet(\bx_*)$ and  $B_{r'}^\bullet(\bx_0)$. In order to complete the proof 
of Theorem \ref{3-pieces}, it only remains to identify the conditional distribution of $\mathcal C^{\bx_*, \bx_0}_{r,r'}$. To do so, we will first 
state and prove a proposition, which may be viewed as a variant of our definition of the Brownian annulus. This proposition also corresponds to 
Definition 1.1 in \cite{ang2022moduli}). 
%
%

We consider now the (free pointed) Brownian disk $\mathbb D_{(a)}$. Recall the notation $H_r$ for the hull of radius $r$ centered at the distinguished point $x_*$ of $\mathbb{D}_a$, which
is defined on the event $\{r<r_*\}$.
We also let $C_r$ be the closure of $\mathbb{D}_{(a)}\setminus H_r$. In a way similar to the results recalled at the beginning of this section, one proves that the
intrinsic metric on $H_r\setminus \partial H_r$ (resp. on $\mathbb{D}_{(a)}\setminus H_r$) has a continuous extension to $H_r$ (resp. to $C_r$) which is a 
metric on this space. The shortest way to verify these properties is to view the Brownian disk $\mathbb{D}_{(a)}$ as embedded in the Brownian sphere, as in Proposition \ref{BrDisk-BrSphere} above,
and then to use the analogous properties in the Brownian sphere recalled at the beginning of this section (we omit the details). In the next proposition, we thus 
view $H_r$ (resp. $C_r$) equipped with the extended intrinsic metric, with the marked subsets $\{x_*\}$ and $\partial H_r$ (resp. with the boundaries $\partial \mathbb D_{(a)}$ and $\partial H_r$) and
with the restriction of the volume measure on $\mathbb{D}_{(a)}$, as a random variable in $\mathbb{M}^{2,1}$. Recall the notation $\mathcal{P}_r$ for the boundary size of $H_r$.

\begin{proposition}
	\label{key-result} Under $\P(\cdot \mid r<r_*)$, $C_r$ and $H_r$ are conditionally independent given $\mathcal{P}_r$,
	$H_r$ is distributed as a standard hull with radius $r$ and perimeter $\mathcal{P}_r$ and 
	$C_r$ is distributed as a free Brownian annulus with perimeters $a$ and $\mathcal{P}_r$.
\end{proposition}

\begin{proof} By Proposition \ref{BrDisk-BrSphere}, we may and will assume that the Brownian disk $\mathbb{D}_{(a)}$ is constructed as the subspace $\check B^\bullet_{\br_a}(\bx_*)$ of the free Brownian sphere $\bm_\infty$ under
	$\N_0(\cdot\mid \br_a<\infty)$, where
	$$\br_a=\inf\{r\in(0,\br_*):\z_{r-\br_*}=a\},$$
	and we recall that $\check B^\bullet_{\br_a}(\bx_*)$ is the closure of $\bm_\infty\setminus B^\bullet_{\br_a}(\bx_*)$.
	The distance between the distinguished point $x_*=\bx_0$ and the
	boundary of $\D_{(a)}$ is then $r_*=\br_*-\br_a$, where $\br_*= \mathbf D(\bx_0, \bx_*)$. Furthermore, conditioning $\mathbb D_{(a)}$ on the event $\{r<r_*\}$ 
	is then equivalent to arguing under $\N_0(\cdot\mid r+\br_a<\br_*)$.
On the event $\{r+\br_a<\br_*\}$, the hull $H_r$ is identified to the hull $B^\bullet_r(\bx_0)$ and ${C}_r$ is identified to $\check B^\bullet_{\br_a}(\bx_*)\setminus B^\circ_r(\bx_0)$
where $B^\circ_r(\bx_0)$ denotes the interior of $B^\bullet_r(\bx_0)$.
In particular, the perimeter $\mathcal{P}_r$ of $H_r$ is identified with the boundary size $\mathcal{Z}^{\bx_0}_r$ of $B^\bullet_r(\bx_0)$.

As explained at the beginning of this section, we view $B^\bullet_r(\bx_0)$ as a random variable in $\mathbb{M}^{2,1}$. Similarly \cite[Theorem 8]{Stars}
(with the roles of $\bx_*$ and $\bx_0$ interchanged)
allows us to view $\check B^\bullet_r(\bx_0):=\bm_\infty\setminus B^\circ_r(\bx_0)$, equipped with the extended intrinsic metric, as a random variable in $\mathbb{M}^{2,1}$
(the marked subsets are $\{\bx_*\}$ and $\partial B^\bullet_r(\bx_0)$). 

\medskip
\noindent{\bf Fact.} Under $\N_0(\cdot\,|\,\br_*>r)$, $B^\bullet_r(\bx_0)$ and $\check B^\bullet_r(x_0)$ are independent conditionally
given the perimeter $\mathcal{Z}^{\bx_0}_r$, $B^\bullet_r(\bx_0)$ is distributed as a standard hull of radius $r$ and perimeter $\mathcal{Z}^{\bx_0}_r$, and
$\check B^\bullet_r(x_0)$ is distributed as a free Brownian disk of perimeter $\mathcal{Z}^{\bx_0}_r$.

\medskip
This follows from \cite[Theorem 8]{Stars}, up to the interchange of $\bx_*$ and $\bx_0$. 
We then  note that the event $\{r+\br_a<\br_*\}$ is measurable with respect
to $\check B^\bullet_r(\bx_0)$, and that, on this event, $\check B^\bullet_{\br_a}(\bx_*)\setminus B^\circ_r(\bx_0)$ is a function
of $\check B^\bullet_r(\bx_0)$ (indeed  $\check B^\bullet_{\br_a}(\bx_*)\setminus B^\circ_r(\bx_0)$ is obtained 
from $\check B^\bullet_r(\bx_0)$ by ``removing'' the hull of radius $\br_a$ centered at $\bx_*$). It
follows from these observations and the preceding Fact that, under $\N_0(\cdot \mid r+\br_a<\br_*)$, $B^\bullet_r(\bx_0)$
and $\check B^\bullet_{\br_a}(\bx_*)\setminus B^\circ_r(\bx_0)$ are independent conditionally on $\mathcal{Z}^{\bx_0}_r$.

To get the statement of the proposition, it only remains to determine the conditional distribution of $\check B^\bullet_{\br_a}(\bx_*)\setminus B^\circ_r(\bx_0)$ knowing $\mathcal{Z}^{\bx_0}_r$,
under $\N_0(\cdot \mid r+\br_a<\br_*)$. To this end, we observe that, by construction, on the event $\{r+\br_a<\br_*\}$,
$$\check B^\bullet_{\br_a}(\bx_*)\setminus B^\circ_r(\bx_0)=\check B^\bullet_{r}(\bx_0)\setminus B^\circ_{\br_a}(\bx_*),$$
By the preceding Fact, $\check B^\bullet_{r}(\bx_0)$ is distributed under $\N_0(\cdot\,|\,\br_*>r)$ as a free pointed Brownian disk 
with perimeter $\mathcal{Z}^{\bx_0}_r$ (whose distinguished point is $\bx_*$). Under $\N_0(\cdot \mid r+\br_a<\br_*)$, this 
Brownian disk is further conditioned on the event that there is a hull of perimeter $a$ centered at the distinguished point $\bx_*$, and $\br_a$
is the first radius at which this occurs. By our definition of the Brownian annulus, this means that, under $\N_0(\cdot \mid r+\br_a<\br_*)$ and conditionally
on $\mathcal{Z}^{\bx_0}_r$,  $\check B^\bullet_{r}(\bx_0)\setminus B^\circ_{\br_a}(\bx_*)$ is a Brownian annulus with
perimeters $\mathcal{Z}^{\bx_0}_r$ and $a$. This completes the proof. 
\end{proof}

\proof[Proof of Theorem \ref{3-pieces}] As in the preceding proof (interchanging again the roles of $\bx_*$ and $\bx_0$), we know that, under $\N_0(\cdot\,|\,\br_*>r)$ and conditionally on $\z^{(\bx_*)}_r$,
$\check B^\bullet_r(\bx_*)$ is a (free pointed) Brownian disk with perimeter  $\z^{(\bx_*)}_r$, whose distinguished 
point is $\bx_0$. Under $\N_0(\cdot\,|\,\br_*>r+r')$, this Brownian disk is conditioned on the event that 
the distinguished point is at distance greater than $r'$ from the boundary. We can thus apply 
Proposition \ref{key-result}, with $r$ replaced by $r'$, to this Brownian disk, and it follows that, under $\N_0(\cdot\,|\,\br_*>r+r')$,
conditionally on the pair $(\z^{(\bx_*)}_r,\z^{(\bx_0)}_{r'})$, the space $\check B^\bullet_r(\bx_*)\setminus B^\circ_{r'}(\bx_0)$ is
a Brownian annulus with perimeters $\z^{(\bx_*)}_r$ and $\z^{(\bx_0)}_{r'}$. This completes the proof since
$\mathcal	C^{\bx_*, \bx_0}_{r,r'}= \check B^\bullet_r(\bx_*)\setminus B^\circ_{r'}(\bx_0)$ by construction. 
\endproof

\section{Explicit computations for the length of the annulus}
\label{sec:expli-comp}

Recall the setting of Section \ref{sec:def-annulus}. We define the \emph{length} $\mathcal L_{(a, b)}$ of the annulus $\C_{(a,b)}$ as the
distance between the two boundaries $\partial_1\mathbb C_{(a,b)}$ and $\partial_0\mathbb C_{(a,b)}$.
Our goal in this section is to discuss the distribution of $\mathcal L_{(a, b)}$. From formula \eqref{dist-cyl}, we get that
$\mathcal L_{(a, b)}$ is given under the probability measure $\P(\cdot\mid r_b<\infty)$ by the formula
$$\mathcal L_{(a, b)}= r_*-r_b.$$
From the discussion in the proof of Lemma \ref{proba-condit}, we see that the distribution of $\mathcal L_{(a, b)}$ is
the law of the last hitting time of $b$ for a continuous-state branching process with branching mechanism
$\psi(\lambda)=\sqrt{8/3}\,\lambda^{3/2}$ with initial distribution $\frac{3}{2}a^{3/2}\,(a+z)^{-5/2}$, conditionally
on the fact that this process visits $b$. Unfortunately, we were not able to use this interpretation to 
derive an explicit analytic expression for the law of $\mathcal L_{(a, b)}$, but the following proposition
still gives some useful information. 

\begin{proposition}
\label{first-mom}
The first moment of $\mathcal L_{(a, b)}$ is
	\begin{align*}
		\sqrt{\frac{3\pi}{2}}(a+b)\left(\sqrt{ a^{-1}}+\sqrt{b^{-1}}-\sqrt{a^{-1}+b^{-1}}\right).
	\end{align*}
Furthermore, the probability of the event $\{\mathcal L_{a, b}>u\}$ is asymptotic to $3(a+b)u^{-2}$ when $u\to\infty$.
\end{proposition}

\proof
To simplify notation, we consider first the case $a=1$ and we write $\mathcal L_b=\mathcal L_{1, b}$.
For every $x\geq 0$, we write $(Z_t)_{t\geq 0}$ for a continuous-state branching process with branching mechanism
$\psi$ that starts from $x$ under the probability measure $\P_x$. Similarly, we write
$(X_t)_{t\geq 0}$ for a spectrally positive L\'evy process with Laplace exponent $\psi$ starting from $x$ under $\P_x$,
and we also set $T_0=\inf\{t\geq 0:X_t=0\}$. By the Lamperti transformation, we have
for every measurable function $f:\mathbb R_+\to \mathbb R_+$ such that $f(0)=0$,
	\begin{align*}
		\mathbb{E}_x\left[\int_0^\infty f(Z_t)\mathrm \dd t\right]= \mathbb{E}_x\left[\int_{0}^{T_0} f(X_t)\frac{\mathrm{d}t}{X_t}\right]
	\end{align*}
	On the other hand, the potential kernel of the L\'evy process $X$ killed upon hitting $0$ is computed in the proof 
	of Theorem VII.18 in \cite{Bertoin}: for every measurable function $g:\mathbb R_+\to \mathbb R_+$,
	\begin{align}
			\mathbb E_x\left[\int_0^{T_0}g(X_t)\mathrm{d}t\right]= \int_0^\infty g(y)( W(y)-\mathbf 1_{\{x<y\}}W(y-x))\mathrm dy,
	\end{align}
where $ W(u)$ is the scale function of the L\'evy process $-X$, which is given here by $W(u)=\sqrt{(3/2\pi)u}$. Suppose then
that $Z$ starts with initial density $\frac{3}{2}\,(1+x)^{-5/2}$ under the probability measure $\P$. It follows from
the preceding two displays that
\begin{align}
\label{calcul-tech}
		\mathbb E\left[\int_0^\infty f(Z_t)\mathrm{d}t\right]&= \frac{3}{2}\sqrt{\frac{3}{2\pi}}\int_0^{\infty}\frac{\mathrm dx}{(1+x)^{5/2}}\int_0^\infty \frac{f(y)}{y}(\sqrt y-\mathbf 1_{\{x<y\}}\sqrt{y-x})\mathrm dy\nonumber\\
		&= \sqrt{\frac{3}{2\pi}}\int_0^\infty \frac{f(y)}{y} \left(\sqrt y- \frac{3}{2}\int_0^{y}\frac{\sqrt{y-x}}{(1+x)^{5/2}}\,\mathrm dx\right)\mathrm dy\nonumber\\
		&=\sqrt{\frac{3}{2\pi}} \int_0^\infty \frac{f(y)}{y}\left(\sqrt y-\frac{y^{3/2}}{1+y}\right)\mathrm dy\nonumber\\
		&=  \sqrt{\frac{3}{2\pi}}\int_0^\infty \frac{f(y)}{\sqrt y (1+y)}\mathrm dy.
\end{align}

Next let $L_b:=\sup\{t\geq 0: Z_t=b\}$, with the convention $\sup\varnothing=0$. For $u>0$, the conditional probability 
that $L_b>u$ given $Z_u$ is the probability that $Z$ started from $Z_u$ visits $b$, and it was already noticed in the proof of Lemma 
\ref{proba-condit} that this probability is equal to $1-\sqrt{(b-Z_u)^+/b}$. Hence, we get
\begin{align*}
		\mathbb P(L_b>u)=\mathbb E\left[1-\sqrt{\frac{(b-Z_u)_+}{b}}\right],
\end{align*}
and we integrate with respect to $u$, using \eqref{calcul-tech}, to get
$$\mathbb E\left[L_b\right]= \sqrt{\frac{3}{2\pi}} \int_{0}^\infty\Big(1-\sqrt{\frac{(b-y)^+}{b}}\Big)\,\frac{\mathrm dy}{\sqrt y(1+y)}.$$
After some straightforward changes of variables, we arrive at
$$\mathbb E\left[L_b\right]= \sqrt{\frac{3\pi}{2}} \left(1-\frac{\sqrt b}{\pi}  \int_{\mathbb R}\frac{x^2}{(1+b+x^2)(1+x^2)}\mathrm dx\right).$$
The integral in the right-hand side is computed via a standard application of the residue theorem, and we get 
$$\mathbb E\left[L_b\right]=\sqrt{\frac{3\pi}{2}}\left(1-\frac{\sqrt{1+b}-1}{\sqrt b}\right).$$
As discussed at the beginning of the section, the first moment of $\mathcal{L}_{1,b}$ is equal to $\E[L_b\mid L_b>0]$, and 
we know from Lemma \ref{proba-condit} that $\P(L_b>0)=\P(r_b<\infty)=(1+b)^{-1}$. Hence 
 the first moment of $\mathcal{L}_{1,b}$ is $(1+b)\,\E[L_b]$, and we get the first assertion of the proposition when $a=1$. 
 In the general case, we just have to use a scaling argument, noting that 
$\mathcal L_{(a, b)}$ has the same law as $\sqrt a \mathcal L_{1, b/a}$. 

Let us turn to the second assertion. Again, by scaling, it suffices to consider the case $a=1$. We use the fact that
$$\P_x(Z_u=0)=\exp(-\frac{3x}{2u^2}),$$
which follows from the explicit form of the Laplace transform of $Z_u$ (see e.g. formula (1) in \cite{Hull}). Then
$$\P(Z_u>0)= \frac{3}{2} \int_0^\infty (1-\exp(-\frac{3x}{2u^2}))\, \frac{\mathrm{d}x}{(1+x)^{5/2}}= \frac{9}{4u^2} \int_0^\infty \frac{x\,\mathrm{d}x}{(1+x)^{5/2}} + O(u^{-3})
=3u^{-2}+O(u^{-3}),$$
as $u\to\infty$.
Again using the Laplace transform of $Z_u$, it is straightforward to verify that $\P(Z_u\in(0,b])=O(u^{-3})$ as $u\to\infty$. 
Since $\mathbb P(Z_u>b)\leq \mathbb P(L_b>u)\leq \mathbb P(Z_u>0)$, we get that $ \mathbb P(L_b>u)=3u^{-2}+O(u^{-3})$ as $u\to\infty$.
Finally, the probability that $\mathcal{L}_{(a,b)}>u$ is equal to $(1+b)\mathbb P(L_b>u)$, which gives the desired asymptotics. \endproof

\section*{Appendix}
In this appendix, we prove Proposition \ref{law-perimeter-hull}. Recall the notation in formula \eqref{def-peri-hull}, and also set
$$\mathcal{Y}_s=\sum_{i\in I} \mathcal{Z}_s(\omega^i),$$
for every $s\leq 0$, in such a way that $\mathcal{P}_r=\mathcal{Y}_{r-r_*}$ for $r\in (0,r_*]$. It is easy to adapt the arguments of 
\cite[Section 5]{BrowDiskandtheBrowSnake} (see in particular formula (34) in this reference) to get the formula
\begin{equation}
\label{formu-appen}
\E[\mathbf{1}_{\{r_*>r\}}\exp(-\lambda \mathcal{Y}_{r-r_*})]=3r^{-3}\,\int_{-\infty}^0 \mathrm{d}s\,\E\Big[\mathcal{Y}_s\,\exp\Big(-(\lambda+\frac{3}{2r^2})\,\mathcal{Y}_s\Big)\Big].
\end{equation}
We already noticed in the proof of Lemma \ref{proba-condit} that $\mathcal{Y}_0$ has density $\frac{3}{2}a^{3/2}\,(a+z)^{-5/2}$. Using this and the special Markov property of the
Brownian snake (see e.g.~\cite[Proposition 2.2]{Hull}), we get, for every $\mu>0$ and $s<0$,
$$\E[\exp(-\mu \mathcal{Y}_s)] =\int_0^\infty \dd z\,\frac{3}{2}\,a^{3/2}(a+z)^{-5/2} \exp\Big(-z\N_0(1-\exp(-\mu \z_s))\Big).$$
According to formula (6) in \cite{Hull}, 
$$\N_0(1-\exp(-\mu \z_s))=\Big(\mu^{-1/2} +\sqrt{\frac{2}{3}}|s|\Big)^{-2}.$$
If we substitute this in the previous display, and then differentiate with respect to $\mu$, we arrive at
$$\E[\mathcal{Y}_s\,\exp(-\mu \mathcal{Y}_s)]=\frac{3}{2}\,a^{3/2}\int_0^\infty \dd z\,\frac{z}{(a+z)^{-5/2}} \Big(1+|s|\sqrt{\frac{2\mu}{3}}\Big)^{-3}
\exp\Big(-z\Big(\mu^{-1/2} +\sqrt{\frac{2}{3}}|s|\Big)^{-2}\Big).$$
We take $\mu=\lambda +\frac{3}{2r^2}$ and use formula \eqref{formu-appen} to obtain
\begin{align*}
&\E[\mathbf{1}_{\{r_*>r\}}\exp(-\lambda \mathcal{Y}_{r-r_*})]\\
&\qquad=\frac{9}{2}\,r^{-3}a^{3/2}\int_0^\infty\dd z\, \frac{z}{(a+z)^{5/2}} \int_0^\infty \dd s\,
\Big(1+s\sqrt{\frac{2\mu}{3}}\Big)^{-3}
\exp\Big(-z\Big(\mu^{-1/2} +\sqrt{\frac{2}{3}}s\Big)^{-2}\Big)\\
&\qquad=\frac{9}{4} \sqrt{\frac{3}{2}}\,r^{-3}a^{3/2}\,\mu^{-3/2}\int_0^\infty \frac{\dd z}{(a+z)^{5/2}}\,(1-e^{-\mu z})\\
&\qquad=\frac{3}{2} \sqrt{\frac{3}{2}}\,r^{-3}a^{3/2}\,\mu^{-1/2}\int_0^\infty \frac{\dd z}{(a+z)^{3/2}}\,e^{-\mu z}
\end{align*}
Writing
$$\mu^{-1/2} = \frac{1}{\sqrt{\pi}}\,\int_0^\infty \frac{\dd x}{\sqrt{x}}\,e^{-\mu x},$$
we arrive at
$$\E[\mathbf{1}_{\{r_*>r\}}\exp(-\lambda \mathcal{Y}_{r-r_*})]
=\frac{3}{2}\,\sqrt{\frac{3}{2\pi}}\,r^{-3}\,a^{3/2}\int_0^\infty \dd y \,e^{-\mu y} 
\int_0^y \frac{\dd z}{(a+z)^{3/2} (y-z)^{1/2}}.$$
Finally, a straightforward calculation gives for $y>0$,
$$\int_0^y \frac{\dd z}{(a+z)^{3/2} (y-z)^{1/2}}= 2\frac{\sqrt{y}}{\sqrt{a}\,(a+y)},$$
so that recalling $\mu=\lambda +\frac{3}{2r^2}$, we have
$$\E[\mathbf{1}_{\{r_*>r\}}\exp(-\lambda \mathcal{Y}_{r-r_*})]
=3\,\sqrt{\frac{3}{2\pi}}\,r^{-3}\int_0^\infty \dd y \,e^{-\lambda y}\,\sqrt{y}\,\frac{a}{a+y}\,e^{-3y/(2r^2)}.$$
This completes the proof.

\end{document}